%% file: paper.tex
\documentclass{article}
\usepackage{graphicx} 
\usepackage[margin=3cm]{geometry}
\usepackage{tikz,tikz-cd,tikzit}
\input{markov.tikzstyles}
\usepackage{amsmath,amsthm,amsfonts}
\usepackage{amssymb,xparse}	
\usepackage{mathtools}
\usepackage{latexsym}
\usepackage{todonotes}
\usepackage{hyperref}
\usepackage{comment}
\usepackage{cleveref}
\usepackage{quiver}
\usepackage{xcolor}
\usetikzlibrary{decorations.pathmorphing,shapes,fit}

\definecolor{probcolor}{rgb}{0.3,0.4,0.8}
\definecolor{marginalize}{rgb}{0.75,0.75,0.75}

\tikzset{
	bullet/.style={circle,fill=black, inner sep=1.5pt},  
	probnum/.style={text=probcolor, label distance=-1.5mm},
	probbar/.style={draw=probcolor, fill=probcolor},
	ar/.style={->,shorten <=1mm, shorten >=1mm},
	side/.style={dotted},
}

\newtheorem{theorem}{Theorem}[section]
\newtheorem{proposition}[theorem]{Proposition}
\newtheorem{lemma}[theorem]{Lemma}
\newtheorem{corollary}[theorem]{Corollary}
\newtheorem{definition}[theorem]{Definition}

\theoremstyle{definition}
\newtheorem{example}[theorem]{Example}

\newtheorem{remark}[theorem]{Remark}

\numberwithin{equation}{section}

\newcommand{\N}{\mathbb{N}}

\newcommand{\R}{\mathbb{R}}
 
\newcommand{\aseq}{\simeq}
\newcommand{\E}{\mathbb{E}}

\newcommand{\cat}[1]{{\mathsf{#1}}} 
\newcommand{\samp}{\mathrm{samp}}
\newcommand{\cop}{\mathrm{copy}}
\newcommand{\del}{\mathrm{del}}
\newcommand{\sinv}{\mathrm{sinv}}
\newcommand{\op}{\mathrm{op}}

\newcommand{\id}{\mathrm{id}} 		
\newcommand{\inv}{\mathrm{inv}}

\newcommand{\overbar}[1]{\mkern 1.5mu\overline{\mkern-1.5mu#1\mkern-1.5mu}\mkern 1.5mu}

\title{Categorical probability spaces, ergodic decompositions, and transitions to equilibrium}
\author{Noé Ensarguet and Paolo Perrone}
\date{}

\begin{document}
	
	\maketitle
	
	\begin{abstract}
		We study a category of probability spaces and measure-preserving Markov kernels up to almost sure equality.
		This category contains, among its isomorphisms, mod-zero isomorphisms of probability spaces. It also gives an isomorphism between the space of values of a random variable and the $\sigma$-algebra that it generates on the outcome space, reflecting the standard mathematical practice of using the two interchangeably, for example when taking conditional expectations.
		
		We show that a number of constructions and results from classical probability theory, mostly involving notions of \emph{equilibrium}, can be expressed and proven in terms of this category.
		In particular:
		\begin{itemize}
			\item Given a stochastic dynamical system acting on a standard Borel space, we show that the almost surely invariant $\sigma$-algebra can be obtained as a limit and as a colimit;
			\item In the setting above, the almost surely invariant $\sigma$-algebra gives rise, up to isomorphism of our category, to a standard Borel space;
			\item As a corollary, we give a purely categorical, almost sure version of the ergodic decomposition theorem for stochastic actions;
			\item As an example, we show how de Finetti's theorem and the Hewitt-Savage zero-one law fit in this limit-colimit picture.
		\end{itemize}
		This work uses the tools of categorical probability, in particular Markov categories, as well as the theory of dagger categories.
	\end{abstract}

	\tableofcontents

	\section{Introduction}
	
	In recent years there has been growing interest in category-theoretic structures for probability theory and related fields.
	While the first steps in this direction date back to Lawvere \cite{lawvere1962} and Čencov \cite{cencovcategories}, with their studies of the category $\cat{Stoch}$ of measurable spaces and Markov kernels, only in the last decade have there been systematic efforts to state, prove, and interpret results of probability theory by means of categorical methods. 
	
	Today there are two main formalisms in categorical probability, closely related to one another. On the one hand there are \emph{probability monads}\footnote{Currently, the term ``probability monad'' is an informal expression, just like ``forgetful functor''. Often one wants an affine monoidal monad to model probability distributions, see \cite{kock2012distributions}.} such as the Giry monad \cite{giry} and the Radon monad \cite{swirszcz}, see \cite{jacobs-pmonads} for an overview and \cite[Chapter~1]{paolo-phdthesis} for an introduction. These monads allow to form spaces of probability distributions, joints, marginals, and via the Kleisli and Eilenberg-Moore constructions they allow to talk about stochastic maps and of expectation values. 
	
	On the other hand we have particularly structured monoidal categories such as \emph{copy-discard categories} (\cite{chojacobs2019strings}, also called \emph{garbage-share monoidal categories} \cite{gadducci1996}) and \emph{Markov categories} (\cite{fritz2019synthetic}, also called \emph{affine copy-discard categories}). They incorporate as additional structures notions of determinism, stochastic independence, and conditioning, among other probabilistic concepts.
	Markov categories have been used to state and prove categorically a number of results of classical probability theory, such as the Hewitt-Savage and Kolmogorov zero-one laws (\cite{fritzrischel2019zeroone}), de Finetti's theorem (\cite{fritz2021definetti}), a d-separation criterion (\cite{fritz2022dseparation}), and an ergodic decomposition theorem for deterministic dynamical systems (\cite{moss2022ergodic}).
	
	As shown in \cite{fritz2020representable} and \cite{moss2022probability}, Markov categories and probability monads interact in a fruitful way. For example, the Kleisli category of an affine monoidal monad on a cartesian monoidal category is Markov.
	Both formalisms can also be seen from the perspective of computer science, where probability monads can add probabilistic computation to pure programs \cite{jones-plotkin}, and monoidal categories such as Markov categories can describe the categorical semantics of probabilistic programs, in particular the process of conditioning \cite{stein2021structural,stein2021conditioning}.
	
	In this work we advance towards a third formalism for categorical probability, based on \emph{dagger categories}, we study some of its connections to Markov categories, and use it to express some classical results of probability. A dagger category can be thought of as a category where morphisms can be ``walked either way'' without having to be isomorphisms, similarly to the edges of an undirected graph (see for example \cite{wayofdagger} for more details). Dagger categories have a long history of usage in categorical approaches to quantum information theory, at least since \cite{quantumprotocols,quantumwithoutsums,dagger-compact}, (see also \cite{cqt} for a more recent account).
	Here we are interested, instead, in \emph{classical} probability. 
	The situation we want to model is a category whose objects are probability spaces (i.e.~measurable spaces equipped with a probability measure) and whose morphisms are a version of stochastic maps, with the dagger structure given by Bayesian inversion. 
	Equivalently, from the point of view of transport theory \cite{villani}, we can view such a dagger category as a category of probability spaces and \emph{transport plans} between them (see for example \cite{lifting}).
	Just as the prototypical example of a Markov category is $\cat{BorelStoch}$, the main example of dagger category for our purposes is the category $\cat{PS(BorelStoch)}$ of standard Borel probability spaces and Markov kernels quotiented by almost sure equality. To the best of our knowledge, this category and its dagger structure were first studied in \cite{dahlqvist2018borel}. 
	In \cite[Section~13]{fritz2019synthetic} it is shown how to obtain this category from $\cat{BorelStoch}$ categorically, so that one can generalize this construction by replacing $\cat{BorelStoch}$ with an arbitrary Markov category with conditionals. 
	
	The main probabilistic phenomena which we study here in terms of daggers are notions of \emph{invariance} and \emph{equilibrium} for dynamical systems and Markov chains. We show that in the category $\cat{PS(BorelStoch)}$, the \emph{invariant $\sigma$-algebra} satisfies a dagger-categorical universal property, being both a limit and a colimit. This fact alone allows one to state and prove an almost sure version of a general ergodic decomposition theorem (similar to the one of \cite{moss2022ergodic}, but valid outside the deterministic case).
	We also show that all idempotents split in $\cat{PS(BorelStoch)}$ using the analogous result proven for $\cat{BorelStoch}$ in \cite{supports}, and give an interpretation of idempotents in $\cat{PS(BorelStoch)}$ in terms of ``averaging'' over the dynamics (see \Cref{equilibrium}).\footnote{The idea of using idempotent Markov kernels, for example via conditional expectations, to describe the long-time behavior of a system can be traced back at least to \cite{blackwell1942idempotent}. In the context of categorical probability, to the best of our knowledge the idea first appeared in \cite{fritz2021topos}.}
	To illustrate our formalism, we show that it allows us to combine de Finetti's theorem and the Hewitt-Savage zero-one law into a coherent, unified picture, compatible with the usage of these statements in traditional probability theory (see \Cref{examples}). 
	
	We hope that this work paves the way for further applications of dagger-categorical methods to probability and to dynamical systems, and that it provides deeper connections between classical and quantum probability.

	\subsubsection*{Outline}
	\begin{itemize}
		\item In \Cref{probstoch} we recall the construction of $\cat{ProbStoch(C)}$ (or $\cat{PS(C)}$), where $\cat{C}$ is a causal Markov category, and we establish some new facts about it. 
		First of all, we show that (almost sure) determinism, in the Markov-categorical sense, can be expressed in $\cat{PS(C)}$ in terms of dagger epicness (\Cref{dag id}). We then turn to study isomorphisms for the specific case of $\cat{PS(Stoch)}$: we show in \Cref{isomodzero} that they include mod-zero isomorphisms of measure spaces, and that any surjective random variable $f$ induces an isomorphism between the space of values and the space of outcomes (with $\sigma$-algebra generated by $f$).
		In \Cref{dynsys} we define dynamical systems in $\cat{PS(Stoch)}$, explain their interpretation as stationary Markov chains, and look at notions of invariance for morphisms, generalizing invariant measures and invariant observables.
		
		\item In \Cref{main} we state and prove the main structural results of this work.
		In \Cref{colimits} we show that for each (stochastic, measure-preserving) dynamical system in $\cat{PS(Stoch)}$, the $\sigma$-algebra of almost surely invariant sets is a colimit, compatible with the dagger structure.
		In \Cref{idempotents} we show that all idempotents split in $\cat{PS(BorelStoch)}$ (\Cref{idempsplit}), and that as a consequence, $\sigma$-algebra of almost surely invariant sets on a standard Borel space is again standard Borel up to isomorphism of $\cat{PS(Stoch)}$ (\Cref{xinvborel}). This in particular implies that every standard Borel space equipped with any sub-$\sigma$-algebra is again standard Borel up to isomorphism of $\cat{PS(Stoch)}$ (\Cref{subsigma}).
		In \Cref{limits} we show that the $\sigma$-algebra of almost surely invariant sets is also a \emph{limit} in $\cat{PS(BorelStoch)}$, not just a colimit. In \Cref{ergodic} we use this fact to state and prove an almost sure version of the ergodic decomposition theorem \Cref{erg_dec}.
		In \Cref{equilibrium} we then show how idempotents in $\cat{PS(BorelStoch)}$ can be used to express notions of ``equilibrium'' for dynamical systems and Markov chains. 
		
		\item In \Cref{examples} we apply the results of \Cref{main} to the concrete case of $\cat{C}=\cat{BorelStoch}$ to express categorically some classical constructions and statements of probability theory. 
		We start by looking at finite permutations (\Cref{finiteperm}), and then turn to the infinite case, where we give a categorical almost-sure version of de Finetti's theorem and of the Hewitt-Savage zero-one law, showing that i.i.d.\ sequences are ergodic under permutations (\Cref{definetti}), and of Bernoulli shifts (\Cref{bernoulli}).
		
		\item Finally, in \Cref{appendix} we give some background on Markov categories. We also provide references to more in-depth material for the interested readers. 
	\end{itemize}
	
	Most of the material presented here is part of the first author's dissertation ``Categorical Aspects of Markov Chains'', submitted towards the degree of MSc in Mathematics and Foundations of Computer Science at the University of Oxford.
	
	\subsubsection*{Acknowledgements}
	We want to thank Rob Cornish, Tobias Fritz, Sean Moss, Dario Stein, Ned Summers and Yuwen Wang for the helpful discussions and pointers, and Sam Staton and the rest of his group for the support and the interesting conversations.
	Research for the second author is funded by Sam Staton's ERC grant ``BLaSt -- a Better Language for Statistics''.

	\section{The $\cat{PS}$ construction}\label{probstoch}
	
	The main category of interest in this work is the category $\cat{PS(Borel)}$ of probability spaces and Markov kernels taken up to almost sure equality. It was first defined in \cite{dahlqvist2018borel} (under the name $\cat{Krn}$).
	
	First of all, let us define almost-sure equality for Markov kernels, instantiating the notion for general Markov categories (see \Cref{as}, as well as the original sources \cite[Section~5]{chojacobs2019strings} and \cite[Section~13]{fritz2019synthetic}). Consider measurable spaces $X$ and $Y$, and a probability measure $p$ on $X$. We say that two Markov kernels $f,g:X\to Y$ are \emph{$p$-almost surely equal} if for all measurable subsets $A$ of $X$ and $B$ of $Y$ we have that 
	$$
	\int_A f(B|x) \,p(dx) = \int_A g(B|x) \,p(dx) .
	$$ 
	Equivalently, if for all measurable subsets $B$ of $Y$, the quantities $f(B|x)$ and $g(B|x)$ are equal for all $x$ in a set of $p$-measure one. This set depends on $B$ in general, but it can be taken independently of $B$ if $Y$ is standard Borel.
	
	Whenever $f$ and $g$ are $p$-almost surely equal we write $f\aseq_p g$, or $f\aseq g$ whenever there is no ambiguity on the measure.
	
	\begin{definition}\label{defPSB}
		The category $\cat{PS(BorelStoch)}$, or more briefly $\cat{PS(Borel)}$, is constructed as follows: 
		\begin{itemize}
			\item The objects are standard Borel probability spaces $(X, p)$;
			\item The morphisms $(X, p) \to (Y, q)$ are Markov kernels $f : X \to Y$, considered modulo $p$-almost sure equality (i.e. $f = g$ in $\cat{PS(Borel)}$ when $f \aseq_p g$), and that are measure-preserving, i.e.~$f p =q$, or more explicitly, for every measurable subset $B$ of $Y$, we have
			$$
			q(B) = \int_X f(B|x)\,p(dx) . 
			$$
		\end{itemize}
	\end{definition}
	
	In \cite[Section~13]{fritz2019synthetic}, the same construction was extended from standard Borel spaces to arbitrary causal Markov categories. (See \Cref{appendix} for the terminology, as well as the original paper \cite{fritz2019synthetic}.)
	
	\begin{definition}[{\cite[Definition~13.8]{fritz2019synthetic}}]
		Let $\cat{C}$ be a causal Markov category. 
		The category $\cat{ProbStoch(C)}$, or more briefly $\cat{PS(C)}$, is defined as follows:
		\begin{itemize}
			\item Objects are probability spaces $(X, \psi)$
			\item Morphisms $(X, \psi) \to (Y, \phi)$ are the morphisms $f : X \to Y$ in $\cat{C}$, considered modulo $\psi$-a.s. equality, and that are measure-preserving, i.e. $f \psi =\phi$.
		\end{itemize}
	\end{definition}
	
	Since the categories $\cat{Stoch}$ and $\cat{BorelStoch}$ are causal, we can form the categories $\cat{PS(Stoch)}$ and $\cat{PS(BorelStoch)}$, and the latter recovers exactly \Cref{defPSB}.
	
	Whenever we have a parallel pair $f,g:X\to Y$ of morphisms of $\cat{C}$, and we also have a state $p$ on $X$, we write $f\aseq g$ if the two morphisms are $p$-almost surely equal (i.e.~if the resulting morphisms of $\cat{PS(C)}$ are equal), and $f=g$ if the morphisms of $\cat{C}$ are equal, which is a stronger condition.
	For $\cat{C}=\cat{BorelStoch}$, this is exactly the distinction between equality and almost-sure equality of Markov kernels. 
	
	\begin{proposition}[{\cite[Proposition~13.9]{fritz2019synthetic}}]
		Composition in $\cat{PS(C)}$ is well-defined, and $\cat{PS(C)}$ inherits the symmetric monoidal structure from $\cat{C}$.
	\end{proposition}
	
	Note however that $\cat{PS(C)}$ does not in general inherit the Markov category structure from $\cat{C}$. This comes from the fact that the copy morphism does not, in general, descend to a morphism in $\cat{PS(C)}$. Indeed, given an object $(X, p)$ in $\cat{PS(C)}$, having a morphism $\cop_{(X, p)} : (X, p) \to (X, p) \otimes (X, p)$ would impose that $\cop_{(X, p)} \circ p = p \otimes p$ (up to a unitor isomorphism). But this morphism $\cop_{(X, p)}$ cannot be $\cop_X$ in $\cat{C}$, since the previous coherence requirement would amount to the following string diagram equality:
	\ctikzfig{copy_probstoch}
	which does not hold in general (it actually holds if and only if $p$ is a deterministic state).
	
	\begin{proposition}[{\cite[Proposition~13.9 and Remark~13.10]{fritz2019synthetic}}]\label{cond implies dag}
		If $\cat{C}$ has conditionals, $\cat{PS(C)}$ is a dagger symmetric monoidal category, with daggers given by Bayesian inverses.
		In that case it is also equivalent to the category whose objects are probability spaces, whose morphisms are couplings, and whose composition is given in terms of the conditional product (see for example \cite[Remark~12.10]{fritz2019synthetic} and \cite[Section~4]{lifting}).
	\end{proposition}

	The following proposition is known in slightly different language \cite[Proposition~7.31 and Corollary~8.6]{quantum-markov}, and is related to other existing statements in the literature such as \cite[Theorem~7]{dahlqvist2018borel} and \cite[Lemma~2.3]{suffstat}.
	
	\begin{proposition}\label{dag id}
		Let $\cat{C}$ be a causal Markov category.
		Consider a state $p$ on an object $X$ and a morphism $f:X\to Y$, and suppose its Bayesian inverse $f^\dag:(Y,q)\to(X,p)$ exists. The following conditions are equivalent.
		\begin{itemize}
			\item $f$ is $p$-almost surely deterministic;
			\item $f\circ f^\dag \aseq \id_Y$.
		\end{itemize}
	\end{proposition}
	
	When $\cat{C}$ has conditionals, so that $\cat{PS(C)}$ is a dagger category, this proposition characterizes a.s.~deterministic morphisms exactly as the dagger epimorphisms (a.k.a.~\emph{coisometries}, see \cite{wayofdagger} for the terminology).
	
	\begin{proof}
		Suppose first that $f \circ f^\dag \aseq_q \id_Y$. The following diagrams then prove that $f$ is $p$-almost surely deterministic:
		\ctikzfig{id_implies_det}
		
		Note that the third equality uses relative positivity of $\cat{C}$ (see \Cref{causpos}), together with the fact that $f \circ f^\dag \aseq_q \id_Y$ is $q$-a.s.~deterministic (and using \cite[Lemma~3.12]{supports} that shows that a morphism that is a.s.~equal to an a.s.~deterministic morphism is itself a.s.~deterministic).
		
		Conversely, suppose that $f$ is $p$-almost surely deterministic. We then have the following diagrams to prove that $f \circ f^\dag \aseq_q \id_Y$:
		\ctikzfig{det_implies_id}
	\end{proof}

	\begin{proposition}\label{iso bayesian}
		Suppose $f : (X, p) \leftrightarrows (Y, q) : g$ are inverses in $\cat{PS(C)}$ (with $\cat{C}$ a causal Markov category). Then $g$ is a (in $\cat{PS(C)}$, \emph{the}) Bayesian inverse of $f$ with respect to $p$ (and vice versa).
	\end{proposition}
	
	Together with the previous proposition, this proves that any isomorphism of $\cat{PS(C)}$ is almost surely deterministic.
	In the language of dagger categories, this says that every isomorphism is a dagger isomorphism (a.k.a.~\emph{unitary}).
	
	\begin{proof}
		By causality, $\cat{C}$ is also relatively positive, such that the following string diagram equality holds:
		\ctikzfig{iso_bayesian}
		where the first equality follows from (a marginalized version of) relative positivity together with the fact that $gf \aseq_p \id_X$ is $p$-almost surely deterministic.
	\end{proof}

	Note also that by \Cref{a-s det composition}, a.s.\ deterministic morphisms form a subcategory of $\cat{PS(C)}$.

	\subsection{Isomorphisms in $\cat{PS(Stoch)}$}\label{iso}
	
	Every isomorphism of $\cat{C}$ descends to an isomorphism in $\cat{PS(C)}$.
	In the other direction, isomorphisms of $\cat{PS(C)}$ are in general much weaker than isomorphisms in $\cat{C}$.
	This is particularly true for the case of $\cat{C}=\cat{Stoch}$. It was shown in \cite[Appendix~A]{moss2022ergodic} that in general, deterministic isomorphisms in $\cat{Stoch}$ are more general than isomorphisms of measurable spaces, since they take into account the ``indistinguishability'' relation induced by $\sigma$-algebras.
	For $\cat{PS(Stoch)}$ we have a further weakening: isomorphisms of $\cat{PS(Stoch)}$ take into account a form of ``indistinguishability up to measure zero''. We will now illustrate this idea via concrete examples. 
	A possible way to interpret the difference between $\cat{Stoch}$ and $\cat{PS(Stoch)}$ is that $\cat{Stoch}$ sees ``how the points are partitioned, and where they are mapped'', while $\cat{PS(Stoch)}$ sees ``how the \emph{mass} is partitioned, and where it is mapped''. 
	
	The first example will show that given a surjective random variable or random element $\Omega\to R$, in $\cat{Stoch}$ and $\cat{ProbStoch}$ there is an isomorphism between $R$ as a probability space, and $\Omega$ as a probability space taken with the $\sigma$-algebra generated by that of $R$. 
	This is particularly convenient when taking marginals or conditional expectations: in this category there is no difference between applying the projections or just coarse-graining the $\sigma$-algebra. This idea is standard practice in probability theory, and the categories of kernels $\cat{Stoch}$ and $\cat{PS(Stoch)}$ make this precise. 
	
	\begin{proposition}\label{isoRV}
		Let $(\Omega,\Sigma_\Omega,p)$ and $(R,\Sigma_R,q)$ be probability spaces, and let $f:(\Omega,\Sigma_\Omega,p)\to(R,\Sigma_R,q)$ be a surjective measure-preserving function.
		Denote by $\Sigma_f$ the $\sigma$-algebra on $\Omega$ generated by $f$, i.e.~the one that consists of those measurable sets in the form $f^{-1}(B)$ for $B\in\Sigma_R$.
		With this $\sigma$-algebra, the morphism $\tilde{f}:(\Omega,\Sigma_f)\to(R,\Sigma_R)$ defined by $f$ is an isomorphism of $\cat{Stoch}$ (hence also of $\cat{PS(Stoch)}$).
	\end{proposition}
	
	Note that this implies that if $(R,\Sigma_R)$ is a standard Borel space (for example, $\R$ with its Borel $\sigma$-algebra), then $(\Omega,\Sigma_f)$ is standard Borel too, up to deterministic isomorphism of $\cat{Stoch}$.
	
	Note also that the proposition above, together with the next one (\Cref{isomodzero}), shows that $f$ gives an isomorphism of $\cat{PS(Stoch)}$ also in the case when it is not surjective but its image $f(\Omega)\subseteq R$ is measurable.
	
	\begin{proof}
		The Markov kernel induced by $\tilde{f}$ is as follows, for every $\omega\in\Omega$ and every $B\in\Sigma_R$:
		$$
		\delta_{\tilde{f}}(B|\omega) \coloneqq  1_{f^{-1}(B)}(\omega) =
		\begin{cases}
			1 & f(\omega) \in B ; \\
			0 & f(\omega) \notin B .
		\end{cases}
		$$
		Define now the kernel $k : (R,\Sigma_R)\to(\Omega,\Sigma_f)$ as follows: for every measurable $B\in\Sigma_R$ (since every measurable subset in $\Sigma_f$ is in the form $f^{-1}(B)$ for a unique $B$, by surjectivity of $f$) and each $r\in R$
		$$
		k(f^{-1}(B)|r) \coloneqq 1_{B}(r) = 
		\begin{cases}
			1 & r\in B ; \\
			0 & r\notin B .
		\end{cases}
		$$
		Let us now show that $k\circ \delta_{\tilde{f}}$ and $\delta_{\tilde{f}}\circ k$ are equal to the identity kernels.
		Let $B\in\Sigma_R$, $\omega\in\Omega$, and $r\in R$. Then
		\begin{align*}
			(k\circ \delta_{\tilde{f}}) (f^{-1}(B)|\omega) &= \int_R k(f^{-1}(B)|r) \, \delta_{\tilde{f}}(dr|\omega) \\
			&= \int_R 1_{B}(r) \, \delta_{\tilde{f}}(dr|\omega) \\
			&= \delta_{\tilde{f}}(B|\omega) \\
			&= 1_{f^{-1}(B)}(\omega) ,
		\end{align*}
		which is the identity kernel on $(\Omega,\Sigma_f)$, and 
		\begin{align*}
			(\delta_{\tilde{f}}\circ k) (B|r) &= \int_\Omega \delta_{\tilde{f}}(B|\omega)\, k(d\omega|r) \\
			&= \int_\Omega 1_{f^{-1}(B)}(\omega)\, k(d\omega|r) \\
			&= k(f^{-1}(B)|r) \\
			&= 1_{B}(r) ,
		\end{align*}
		which is the identity kernel on $(R,\Sigma_R)$. 
	\end{proof}
	
	Let us now see how isomorphisms ``up to measure zero'' in $\cat{Stoch}$ result in isomorphisms in $\cat{PS(Stoch)}$. 
	
	\begin{proposition}\label{isomodzero}
		Recall that given probability spaces $(X,\Sigma_X,p)$ and $(Y,\Sigma_Y,q)$, an \emph{isomorphism modulo zero} (or \emph{mod zero}) is a measure-preserving isomorphism $(X',\Sigma_{X'},p')\to (Y',\Sigma_{Y'},q')$, where $X'$ and $Y'$ are subsets of $X$ and $Y$ respectively of measure one, taken with the induced $\sigma$-algebras and measures.
		
		An isomorphism mod zero between $(X,\Sigma_X,p)$ and $(Y,\Sigma_Y,q)$ induces an isomorphism $(X,\Sigma_X,p)\to(Y,\Sigma_Y,q)$ of $\cat{PS(Stoch)}$.
	\end{proposition}
	\begin{proof}
		Let $f:(X',\Sigma_{X'},p')\to (Y',\Sigma_{Y'},q')$ be a measure-preserving map with inverse $g:(Y',\Sigma_{Y'},q')\to(X',\Sigma_{X'},p')$.
		Define now the kernels $k_f:X\to Y$ and $k_g:Y\to X$.
		For every $x\in X$ and every measurable subset $B\subseteq Y$, set
		$$
		k_f(B|x) \coloneqq 1_{f^{-1}(B\cap Y')}(x) = 
		\begin{cases}
			1 & f(x)\in B\cap Y' ; \\
			0 & f(x)\notin B\cap Y' .
		\end{cases}
		$$
		Similarly, for every $y\in Y$ and every measurable subset $A\subseteq X$, set
		$$
		k_g(A|y) \coloneqq 1_{g^{-1}(A\cap X')}(y) =
		\begin{cases}
			1 & g(y)\in A\cap X' ; \\
			0 & g(y)\notin A\cap X' .
		\end{cases}
		$$
		Let us now check that the compositions $k_g\circ k_f$ and $k_f\circ k_g$ are almost surely equal to the identities. 
		Let $A,S\in\Sigma_X$. Then
		\begin{align*}
			\int_S (k_g\circ k_f) (A|a) \,p(da) &= \int_S \int_Y k_g(A|y) \, k_f (dy|a) \,p(da) \\
			&= \int_S \int_Y 1_{g^{-1}(A\cap X')}(y) \, k_f (dy|a) \,p(da) \\
			&= \int_S k_f (g^{-1}(A\cap X')|a) \,p(da) \\
			&= \int_S 1_{f^{-1}\big(g^{-1}(A\cap X')\cap Y'\big)}(x) \,p(da) .
		\end{align*}
		Note now that since $f$ and $g$ are inverses,
		\begin{align*}
			f^{-1}\big(g^{-1}(A\cap X')\cap Y'\big) &= f^{-1}\big(g^{-1}(A\cap X')\big) \cap f^{-1}(Y') \\
			&= (A\cap X')\cap X' \\
			&= A\cap X' ,
		\end{align*}
		so we are left with 
		\begin{align*}
			\int_S (k_g\circ k_f) (A|a) \,p(da) &= \int_S 1_{f^{-1}\big(g^{-1}(A\cap X')\cap Y'\big)}(x) \,p(da) \\
			&= \int_S 1_{A\cap X'}(x) \,p(da) \\
			&= \int_S 1_{A}(x) \,p(da) ,
		\end{align*}
		since $1_A$ and $1_{A\cap X'}$ only differ on a subset of measure zero.
		The integrand is the identity kernel on $X$.
		The other direction works similarly.
	\end{proof}
	
	Here is a combination of the previous two types of isomorphism.
	
	\begin{corollary}\label{isopartition}
		Let $(X,\Sigma)$ be a standard Borel space, let $p$ be a probability measure on it, and let $(A_1,\dots,A_n)$ be a finite measurable partition of $X$, i.e.~a collection of mutually disjoint measurable subsets $A_1,\dots,A_n\subseteq X$ such that $\coprod_i A_i= X$. 
		
		Denote by $\Sigma_A$ the $\sigma$-algebra generated by $(A_1,\dots,A_n)$, and denote the restriction of $p$ to $\Sigma_A$ again by $p$.
		Without loss of generality, suppose that all the $A_1,\dots,A_k$ for $k\le n$ have positive measure.
		
		Then $(X,\Sigma_A,p)$ is isomorphic in $\cat{PS(Stoch)}$ to the set $[k]=\{1,\dots,k\}$ with the discrete $\sigma$-algebra and the measure $q$ given by $q(i)=p(A_i)$. 
	\end{corollary}
	
	In particular, $(X,\Sigma_A,p)$ is isomorphic, in $\cat{PS(Stoch)}$, to a standard Borel space, and hence it is, up to isomorphism, in $\cat{PS(Borel)}$.
	This proposition can be easily generalized to the countable case. Less easy is its generalization to an \emph{arbitrary sub-$\sigma$-algebra}, which we will prove as \Cref{subsigma}.
	
	Here is a sketch for the case of $n=3$ and $k=2$, where the blue columns denote the probabilities:
	
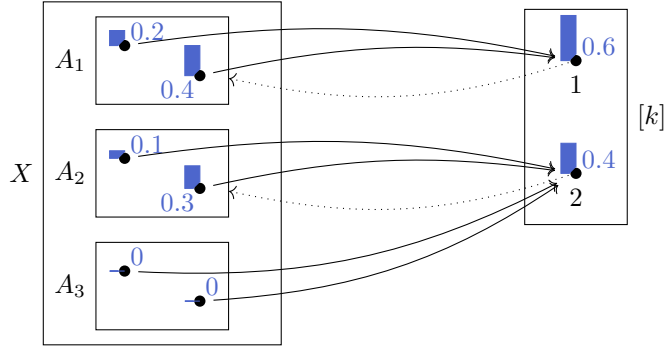
\begin{center}
		\begin{tikzpicture}[baseline,scale=1,
			ar/.append style={shorten >= 2mm},
			]
			\node[bullet, label={[probnum]45:$0.2$}] (00) at (0,0.7) {};
			\draw[probbar] (00) -- ++(-0.2,0) -- ++(0,0.2) -- ++(0.2,0) -- (00) ;
			\node[bullet, {label=[probnum]225:$0.4$}] (10) at (1,0.3) {};
			\draw[probbar] (10) -- ++(-0.2,0) -- ++(0,0.4) -- ++(0.2,0) -- (10) ;
			10
			\node[bullet, {label=[probnum]45:$0.1$}] (01) at (0,-0.8) {};
			\draw[probbar] (01) -- ++(-0.2,0) -- ++(0,0.1) -- ++(0.2,0) -- (01) ;
			\node[bullet, {label=[probnum]225:$0.3$}] (11) at (1,-1.2) {};
			\draw[probbar] (11) -- ++(-0.2,0) -- ++(0,0.3) -- ++(0.2,0) -- (11) ;
			
			\node[bullet, {label=[probnum]45:$0$}] (02) at (0,-2.3) {};
			\draw[probbar] (02) -- ++(-0.2,0) -- ++(0,0) -- ++(0.2,0) -- (02) ;
			\node[bullet, {label=[probnum]45:$0$}] (12) at (1,-2.7) {};
			\draw[probbar] (12) -- ++(-0.2,0) -- ++(0,0) -- ++(0.2,0) -- (12) ;
			
			\node[bullet, label=-90:$1$, {label=[probnum]46:$0.6$}] (1) at (6,0.5) {};
			\draw[probbar] (1) -- ++(-0.2,0) -- ++(0,0.6) -- ++(0.2,0) -- (1) ;
			\node[bullet, label=-90:$2$, {label=[probnum]46:$0.4$}] (2) at (6,-1) {};
			\draw[probbar] (2) -- ++(-0.2,0) -- ++(0,0.4) -- ++(0.2,0) -- (2) ;
			
			\node[inner xsep=1cm, inner ysep=0.5cm, draw=black,fit=(00) (12),label=left:$X$] {} ;
			\node[inner sep=0.3cm, draw=black,fit=(00) (10),label=left:$A_1$] (a1) {} ;
			\node[inner sep=0.3cm, draw=black,fit=(01) (11),label=left:$A_2$] (a2) {} ;
			\node[inner sep=0.3cm, draw=black,fit=(02) (12),label=left:$A_3$] (a3) {} ;
			\node[inner sep=0.6cm, draw=black,fit=(1) (2),label=right:${[k]}$] {} ;
			
			\draw[ar] (00) to[bend left=10] (1) ;
			\draw[ar] (10) to[bend left=10] (1) ;
			\draw[ar] (01) to[bend left=10] (2) ;
			\draw[ar] (11) to[bend left=10] (2) ;
			\draw[ar] (02) to[bend right=15] (2) ;
			\draw[ar] (12) to[bend right=15] (2) ;
			
			\draw[<-,dotted] (a1) to[bend right=15] (1) ;
			\draw[<-,dotted] (a2) to[bend right=15] (2) ;
		\end{tikzpicture}
	\end{center}

	Finally, here is a sufficient for when two sub-$\sigma$-algebras of the same probability space are isomorphic as objects of $\cat{PS(Stoch)}$.
	
	\begin{proposition}\label{sigmaiso}
		Let $(X,\Sigma,p)$ be a standard Borel space. 
		Let $\Sigma_1$ and $\Sigma_2$ be sub-$\sigma$-algebras of $\Sigma$, and suppose that $\Sigma_2\subseteq\Sigma_1$. Denote the restrictions of $p$ to $\Sigma_1$ and $\Sigma_2$ again by $p$.
		Suppose now that for every measurable subset $A_1\in\Sigma_1$ there exists a measurable $A_2\in\Sigma_2$ with $p(A_1\setminus A_2)=p(A_2\setminus A_1)=0$
		Then the deterministic kernel $(X,\Sigma_1,p)\to(X,\Sigma_2,p)$ induced by the set-theoretic identity is part of an isomorphism of $\cat{PS(Stoch)}$.
	\end{proposition}
	
	\begin{proof}
		Recall that the identity $(X,\Sigma_1,p)\to(X,\Sigma_2,p)$ defines the kernel given as follows, for $x_1\in X$ and $A_2\in\Sigma_2$:
		$$
		k_\id(A_2|x) = 1_{A_2}(x) = \begin{cases}
			1 & x\in A_2 ; \\
			0 & x\notin A_2 .
		\end{cases}
		$$
		In general, this identity kernel is a kernel $(X,\Sigma,p)\to (X,\Sigma_2,p)$, i.e.~it is $\Sigma_2$-measurable.
		Moreover, since $\Sigma_2\subseteq\Sigma_1$ this assignment is also $\Sigma_1$-measurable in $x$.
		Therefore this gives a kernel $k_\id:(X,\Sigma_1,p)\to(X,\Sigma_2,p)$. Similarly, we also obtain a kernel $k'_\id:(X,\Sigma,p)\to(X,\Sigma_2,p)$.
		Since $(X,\Sigma)$ is standard Borel, it has disintegrations, and so the kernel $k'_\id:(X,\Sigma,p)\to(X,\Sigma_2,p)$ admits a Bayesian inverse $(k'_\id)^\dag:(X,\Sigma_2,p)\to(X,\Sigma,p)$ (see \Cref{appendix}). Restricting it to the sets of $\Sigma_1$ (or postcomposing it with the kernel $(X,\Sigma)\to(X,\Sigma_1)$ induced again by the identity) we get a Bayesian inverse $(k_\id)^\dag:(X,\Sigma_2,p)\to(X,\Sigma_1,p)$ of $k_\id$, which is in the following form (almost surely):
		$$
		(k_\id)^\dag(A_1|x) = \E[1_A|\Sigma_2](x) ,
		$$
		where $\E$ denotes conditional expectation.
		By \Cref{dag id}, since $k_\id$ is deterministic, $k_\id\circ (k_\id)^\dag\simeq\id_{(X,\Sigma_2,p)}$. 
		In order to have an isomorphism, we need also that $(k_\id)^\dag\circ k_\id\simeq\id_{(X,\Sigma_1,p)}$. 
		
		Suppose now that for every measurable subset $A_1\in\Sigma_1$ there exists an $A_2\in\Sigma_2$ with $p(A_1\setminus A_2)=p(A_2\setminus A_1)=0$. This way, $1_{A_1}$ and $1_{A_2}$ are $p$-almost surely equal, and the same is true for their conditional expectations.
		Denote now by $A$ the set where $\E[1_{A_1}|\Sigma_2]$ and $ \E[1_{A_2}|\Sigma_2]=1_{A_2}$ agree. Note that it is in $\Sigma_2$, and it has $p$-measure one.
		Since $k_\id$ is measure-preserving,
		$$
		1 = p(A) = \int_{X} k_\id(A|x)\,p(dx) ,
		$$
		which means that there is a set $A'\in\Sigma_1$ of $p$-measure $1$ such that for all $x\in A'$, $k_\id(A|x)=1$. 
		Therefore for $A_1\in\Sigma_1$, for all $x\in A'$,
		\begin{align*}
			(k_\id)^\dag\circ k_\id\,(A_1|x) &= \int_X \E[1_{A_1}|\Sigma_2](x')\,k_\id\,(dx'|x) \\
			&= \int_A \E[1_{A_1}|\Sigma_2](x')\,k_\id\,(dx'|x) \\
			&= \int_A 1_{A_2}(x')\,k_\id\,(dx'|x) \\
			&= \int_X 1_{A_2}(x')\,k_\id\,(dx'|x) \\
			&= k_\id\,(A_2|x) \\
			&= 1_{A_2}(x)
		\end{align*}
		and since $1_{A_1}$ and $1_{A_2}$ are $p$-almost surely equal, we obtain that $(k_\id)^\dag\circ k_\id$ is $p$-almost surely equal to the identity kernel. Therefore $k_\id$ is an isomorphism.  
	\end{proof}
	
	\subsection{Dynamical systems and Markov chains}\label{dynsys}
	
	The other main structure we consider in this work is a \emph{dynamical system}, which we write as a diagram (i.e.~as a functor), and over which we will take limits and colimits. 
	As we will see, dynamical systems in $\cat{PS(Stoch)}$ generalize stationary Markov chains.
	
	\begin{definition}
		Let $\cat{A}$ be a category. 
		A \emph{dynamical system} in $\cat{A}$ consists of 
		\begin{itemize}
			\item An object $X$ of $\cat{A}$, on which intuitively the dynamics takes place;
			\item A monoid $M$ (seen as a one-object category), which we can think of as ``time'' or as indexing the dynamics, usually $\N$, $\R$, or a group;
			\item A functor $D:M\to\cat{A}$ which maps the single object of $M$ to $X$, and the elements (morphisms) of $M$ to morphisms $X\to X$ preserving identity and composition.
		\end{itemize}
		For each $m\in M$, we will denote the corresponding induced morphism $X\to X$ again by $m$. 
	\end{definition}
	
	Every single endomorphism $f:X\to X$ generates a ``discrete-time'' dynamical system, indexed by the natural numbers $\N$. Indeed, for each $n\in N$ we can take the $n$-fold application of $f$ with itself (and the identity for $n=0$):
	$$
	\begin{tikzcd}[column sep=small]
		X \ar{r}{f^n} & X & = &  X \ar{r}{f} & X \ar{r}{f} & \dots \ar{r}{f} & X
	\end{tikzcd}
	$$
	Moreover, every $\N$-indexed dynamical system arises in this way, that is, it is generated by the map corresponding to $1\in\N$.
	
	\begin{example}
		Let $M=\N$.
		\begin{itemize}
			\item If $\cat{A}$ is the category $\cat{Set}$ of sets and functions, a dynamical system in $\cat{A}$ is simply an endofunction $f:X\to X$.
			\item If $\cat{A}$ is the category $\cat{Top}$ of topological spaces and continuous map, a dynamical system in $\cat{A}$ is a topological dynamical system. (One can even take spaces with more structure, such as compact metric spaces.)
			\item If $\cat{A}$ is the category $\cat{Meas}$ of measurable spaces and measurable maps, a dynamical system in $\cat{A}$ is a measurable dynamical system.
			\item If $\cat{A}$ is the category $\cat{Stoch}$ of measurable spaces and Markov kernels, a dynamical system in $\cat{A}$ is a stochastic dynamical system.
			\item If $\cat{A}$ is the category $\cat{ProbMeas}$ of measure spaces and measure-preserving functions, a dynamical system in $\cat{A}$ is a measure-preserving dynamical system.
		\end{itemize}
	\end{example}
	
	\begin{example}
		If $M$ is a group, similar considerations hold: it can act continuously, measurably, and so on, depending on the choice of category $\cat{A}$. 
		Keep in mind that, in this context, we are not taking the group to be an object of our category (a topological group, etc.).
		That idea can be modeled in a similar categorical way, but we will not do it in this work.
	\end{example}
	
	In this work we will in particular look at examples where $M$ is countable (though, as a monoid, it is not necessarily $\N$, see for example \Cref{definetti}). However, our formalism is relevant regardless of the cardinality of $M$.
	
	We now want to specialize to dynamical systems in $\cat{A}=\cat{PS(C)}$, where $\cat{C}$ is a causal Markov category, or even a category with conditionals. Our prototypes are $\cat{C}=\cat{Stoch}$ and $\cat{C}=\cat{BorelStoch}$. 
	In this context, we can see a dynamical system $D:M\to\cat{PS(C)}$ as a measure-preserving stochastic dynamical system. 
	In the presence of conditionals, and in discrete time, these are the same as stationary Markov chains:
	
	\begin{theorem}[\cite{hiddenmarkov}, Proposition 2.1]
		Let $\cat{C}$ be a Markov category with conditionals, such as $\cat{BorelStoch}$. 
		Let $X$ be an object of $\cat{C}$, let $X^\N$ be its $\N$-fold Kolmogorov product, and let $p$ be a state on $X^\N$. 
		Denote the $i$-th component of the product $X^\N$ by $X_i$.
		The following conditions are equivalent.
		\begin{enumerate}
			\item \emph{Local Markov property}: for all $i>0$, $p$ exhibits conditional independence of $X_{i+1}$ and $X_{i-1}$ given $X_i$;
			\item \emph{Global Markov property}: for all disjoint subsets $R,S,T\subseteq \N$ such that for all $r\in R$ and $t\in T$ there exists $s\in S$ with $r<s<t$ or $r>s>t$, $p$ exhibits conditional independence of $X_{R}$ and $X_{T}$ given $X_S$;
			\item There is a sequence of morphisms $(f_i:X_i\to X_{i+1})_{i\in \N}$ such that for each finitary joint distribution $p_{0,\dots,n}$ on  $X_0,\dots,X_n$ marginalizing $p$ can be written as follows,
			\ctikzfig{markov_chain}
			where $p_0$ is the marginal of $p$ on $X_0$, and for morphisms (in $\cat{BorelStoch}$, transition kernels) $f_i:X_i\to X_{i+1}$ defined up to almost-sure equality.
		\end{enumerate}
	\end{theorem}
	
	In particular this defines a chain of composable morphisms in $\cat{PS(C)}$,
	$$
	\begin{tikzcd}
		(X_0,p_0) \ar{r}{f_1} & (X_1,p_1) \ar{r}{f_2} & \dots \ar{r}{f_n} & (X_n,p_n) \ar{r}{f_{n+1}} & \dots 
	\end{tikzcd}
	$$
	indexed for example by $(\N,\le)$ (seen as a poset, not as a monoid).
	
	In this work we are interested in the case where the $(X_i,p_i)$ are all equal. In the measure-theoretic setting this amounts to a \emph{stationary} Markov chain, i.e.~with an initial measure on our state space which is preserved by the transition kernels.
	
	\begin{remark}
		Recall that with \Cref{dag id} and \Cref{iso bayesian}, we have that if $\cat{C}$ has conditionals, every isomorphism is almost surely deterministic. Therefore, in that case, any dynamical system $D:M\to\cat{PS(C)}$ in which $M$ is a group is automatically acting in an almost surely deterministic way.
		(Similar considerations can be observed for dynamical systems $D:M\to\cat{C}$ if $C$ is a positive Markov category.)
	\end{remark}
	
	Let us now define invariant morphisms. 
	\begin{definition}
		Let $M$ be a monoid, and let $D:M\to\cat{A}$ be a dynamical system in $\cat{A}$ acting on the object $X$.
		\begin{itemize}
			\item A \emph{left-invariant morphism} for $D$ is a morphism $h:A\to X$ such that for all $m\in M$ we have $m\circ h=h$, i.e.~the following triangle commutes.
			$$
			
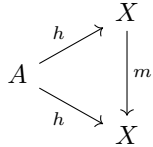
\begin{tikzcd}[row sep=small]
				& X \ar{dd}{m} \\
				A \ar{ur}{h} \ar{dr}[swap]{h} \\
				& X 
			\end{tikzcd}
			$$
			\item A \emph{right-invariant morphism} for $D$ is a morphism $f:X\to Y$ such that for all $m\in M$ we have $f\circ m=f$, i.e.~the following triangle commutes.
			$$
			
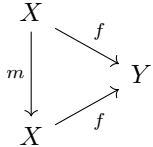
\begin{tikzcd}[row sep=small]
				X \ar{dd}[swap]{m} \ar{dr}{f} \\
				& Y \\
				X \ar{ur}[swap]{f}
			\end{tikzcd}
			$$
		\end{itemize}
	\end{definition}
	One can interpret left-invariant morphisms as parametrized invariant elements, and right-invariant functions as invariant functions or invariant observables. For further intuition, see \cite[Section~2.1]{moss2022ergodic}.
	
	If $\cat{A}=\cat{PS(C)}$, left- and right-invariant morphisms correspond to \emph{almost surely invariant kernels}: for the following diagrams to commute in $\cat{PS(C)}$, 
	$$
	\begin{tikzcd}[row sep=small]
		& (X,p) \ar{dd}{m} \\
		(A,b) \ar{ur}{h} \ar{dr}[swap]{h} \\
		& (X,p) 
	\end{tikzcd}
	\qquad\qquad
	\begin{tikzcd}[row sep=small]
		(X,p) \ar{dd}[swap]{m} \ar{dr}{f} \\
		& (Y,q) \\
		(X,p) \ar{ur}[swap]{f}
	\end{tikzcd}
	$$
	we need that $m\circ h \aseq_b h$ and $f\circ m\aseq_p f$, respectively.
	For $\cat{C}=\cat{BorelStoch}$, this means in particular that the diagrams commute on a subset of measure one. 
	This set in general depends on $m$. For example, for the case of right-invariant functions, we are saying the following:
	\begin{itemize}
		\item For every $m\in M$ there exists a subset $B_m\subseteq X$ of full measure such that for all $x\in B_m$ and all measurable $A\subseteq X$, $fm(A|x) = f(A|x)$.
	\end{itemize}
	Here $A_m$ may depend on $m$. This is in general different from the following, stronger condition:
	\begin{itemize}
		\item There exists a subset $B\subseteq X$ of full measure such that for all $m$, all $x\in B$ and all measurable $A\subseteq X$, $fm(A|x) = f(A|x)$.
	\end{itemize}
	Under some conditions on $M$, for example if it is countable, these two conditions coincide, since one can take $B=\bigcap_{m\in M} B_m$, which has again measure one. 
	This is the case in which we are the most interested (but our formalism works in general). 
	
	For diagrams, we will use the following convention: if we write the objects of the diagram in the form $X$, $Y$, et cetera, then we consider it as a diagram of $\cat{C}$, and so we say that it commutes if it commutes strictly, i.e.~if any two paths with the same endpoints have equal compositions. 
	If instead we write the objects as $(X,p)$, $(Y,q)$, et cetera, then we are considering is as a diagram of $\cat{PS(C)}$, and so we say that it commutes if it commutes almost surely, i.e.~if any two paths between endpoints $(X,p)$ and $(Y,q)$ are such that their compositions are $p$-almost surely equal.

	\section{Main results}\label{main}
	
	\subsection{The invariant $\sigma$-algebra as a colimit}\label{colimits}
	
	We will now define \emph{invariant objects} as particular colimits.
	They can be thought of as ``spaces of orbits up to indistinguishability and up to measure zero''. For additional intuition (without the measure zero part), see \cite[Section~2.1 and Appendix~A]{moss2022ergodic}.
	For $\cat{Stoch}$, they are given by the (almost surely)~invariant $\sigma$-algebra, as we show.
	
	\begin{definition}\label{defxinv}
		Let $\cat{C}$ be a causal Markov category. 
		Let $M$ be a monoid, and let $D:M\to\cat{PS(C)}$ be a dynamical system acting on $(X,p)$.
		An \emph{invariant object} for $D$ is an object $(X_\inv,p_\inv)$ together with a map $r:(X,p)\to(X_\inv,p_\inv)$ of $\cat{PS(C)}$ with the following properties.
		\begin{itemize}
			\item It is a colimit of $D$ in $\cat{PS(C)}$: for every $p$-almost surely invariant morphism $f:(X,p)\to (Y,q)$ there is (up to $p_\inv$-almost sure equality) unique morphism $\tilde{f}:(X_\inv,p_\inv)\to (Y,q)$ making the following diagram commute:
			\[\begin{tikzcd}
				{(X, p)} \\
				& {(X_\inv, p_\inv)} && {(Y, q)} \\
				{(X, p)}
				\arrow["r"', shift right, from=1-1, to=2-2]
				\arrow["r", shift left, from=3-1, to=2-2]
				\arrow["m"', from=1-1, to=3-1]
				\arrow["f", curve={height=-12pt}, from=1-1, to=2-4]
				\arrow["f"', curve={height=12pt}, from=3-1, to=2-4]
				\arrow["{\tilde{f}}"', dashed, from=2-2, to=2-4]
			\end{tikzcd}\]
			\item In the diagram above we have that $f$ is $p$-almost surely deterministic if and only if $\tilde{f}$ is $p_\inv$-almost surely deterministic.
		\end{itemize}
	\end{definition}
	
	In terms of dagger categories, this definition implies in particular that an invariant object $(X_\inv,p_\inv)$ is a \emph{dagger colimit} (see \cite[Chapter~4]{wayofdagger} for the dual case of dagger limits). 
	From this definition it also follows that invariant objects are defined not just up to a unique isomorphism, but up to a unique \emph{dagger isomorphism} compatible with the cones. (See again the reference above for more context.)
	
	We want to show the existence of invariant objects in $\cat{PS(Stoch)}$ (and in \Cref{xinvborel}, we will show that we can extend this result to $\cat{PS(Borel)}$).
	In $\cat{PS(Stoch)}$, the space $X_\inv$ is defined in terms of the (almost surely) invariant $\sigma$-algebra, which we define now.
	
	\begin{definition}
		Let $p$ be a probability measure on $X$.
		Let $M$ be a monoid, and let $D:M\to\cat{PS(Stoch)}$ be a dynamical system acting on the object $(X,p)$ (in a stochastic, measure-preserving way).
		
		A measurable subset $A\subseteq X$ is called \emph{invariant} if for every $m\in M$ we have that the equation
		$$
		m(A|x) = 1_A(x) = \begin{cases}
			1 & x\in A ;\\
			0 & x\notin A
		\end{cases}
		$$
		holds $p$-almost surely, i.e.~for every $B\in\Sigma_X$,
		$$
		\int_B m(A|x) \,p(dx) = \int_B 1_A(x)\,p(dx) = p(A\cap B) .
		$$
	\end{definition}

	Note that this definition involves invariance \emph{almost surely}, and it is different from strict invariance as defined for example in \cite[Definition~3.6]{moss2022ergodic}, although they are equivalent in some cases (see the next \Cref{otherinv}).
	It is equivalent to say that the equation $m(A|x) = 1_A(x)$ holds on a set of measure one.
	Just as for the case of invariant morphisms, this definition quantifies over $m$ in the following way:
	\begin{itemize}
		\item For every $m$ there exists a subset $B_m\subseteq X$ of full measure such that for all $x\in B_m$, $m(A|x) = 1_A(x)$.
	\end{itemize}
	This is in general different from:
	\begin{itemize}
		\item There exists a subset $B\subseteq X$ of full measure such that for all $m$ and all $x\in B$, $m(A|x) = 1_A(x)$.
	\end{itemize}
	Again, under some conditions, such as if $M$ is countable, these two conditions coincide. 
	
	\begin{proposition}
		Invariant sets as defined above form a $\sigma$-algebra.
	\end{proposition}
	
	We call this $\sigma$-algebra the \emph{invariant $\sigma$-algebra}.
	
	Note that if $M$ acts via measure-preserving functions $m:(X,\Sigma,p)\to(X,\Sigma,p)$, the invariance condition can be equivalently stated as the fact that 
	$$
	1_A = 1_{m^{-1}(A)} \qquad p\mbox{-almost surely} ,
	$$
	or also equivalently, that $A$ and $m^{-1}(A)$ only differ by a measure zero set.

	As the next proposition shows, for some countable monoids acting deterministically, this $\sigma$-algebra and that of strictly invariant sets are indistinguishable in $\cat{PS(Stoch)}$.
	
	\begin{proposition}\label{otherinv}
		Let $M$ be either a countable group or a countable commutative monoid, acting deterministically via measure-preserving functions $m:(X,\Sigma,p)\to(X,\Sigma,p)$.
		Denote by $\Sigma_\inv$ the $\sigma$-algebra of almost surely invariant sets, as constructed above, and let $\Sigma'_\inv$ be the $\sigma$-algebra of \emph{strictly} invariant sets, i.e.~those sets $A$ for which strictly, and not just almost surely,
		$$
		m^{-1}(A) = A .
		$$
		Then $(X,\Sigma_\inv,p_\inv)$ and $(X,\Sigma'_\inv,p'_\inv)$ are isomorphic in $\cat{PS(Stoch)}$.
	\end{proposition}
	
	The strictly invariant $\sigma$-algebra, outside the case where it's equivalent to the a.s.\ invariant one, will not be considered in this work.
	
	\begin{proof}
		Notice that every strictly invariant set is almost surely invariant.
		Therefore, by \Cref{sigmaiso}, it suffices to show that given any almost surely invariant set $A$ there exists a strictly invariant set $A'$ such that $p(A\setminus A')=p(A'\setminus A)=0$. 
		So let $A$ be almost surely invariant.
		Note that since $M$ acts deterministically, the fact that $A$ is almost surely invariant can be restated as the fact that $1_A$ and $1_{m^{-1}(A)}$ are $p$-almost surely equal, or equivalently that $A$ and $m^{-1}(A)$ differ only by a measure zero set. Therefore the set 
		$$
		A' \coloneqq \bigcap_{m\in M} m^{-1}\left( \bigcup_{n\in M} n^{-1}(A) \right) ,
		$$
		also differs from $A$ only by a measure zero set.
		To see that $A'$ is strictly invariant, notice first of all that for $\ell\in M$, 
		\begin{align*}
			\ell^{-1}(A') &= \bigcap_{m\in M} (\ell m)^{-1}\left( \bigcup_{n\in M} n^{-1}(A) \right) \\
			&\supseteq \bigcap_{m\in M} m^{-1}\left( \bigcup_{n\in M} n^{-1}(A) \right) = A' ,
		\end{align*}
		since $lm\in M$, and so the intersection is taken over a subset of $M$. 
		If $M$ is a group, the subset inclusion is actually an equality.
		Similarly, if $M$ is commutative, 
		\begin{align*}
			\ell^{-1}(A') &= \bigcap_{m\in M} (\ell m)^{-1}\left( \bigcup_{n\in M} n^{-1}(A) \right) \\
			&= \bigcap_{m\in M} (m\ell)^{-1}\left( \bigcup_{n\in M} n^{-1}(A) \right) \\
			&= \bigcap_{m\in M} (m)^{-1}\left( \bigcup_{n\in M} (\ell n)^{-1}(A) \right) \\
			&\subseteq \bigcap_{m\in M} (m)^{-1}\left( \bigcup_{n\in M} n^{-1}(A) \right) = A',
		\end{align*}
		since the union is taken over a subset of $M$.
	\end{proof}

	The main statement of this section is the following.
	
	\begin{theorem}\label{xinvexists}
		Invariant objects $X_\inv$ in $\cat{PS(Stoch)}$ exist, and are given by (the underlying set of) $X$ equipped with the invariant $\sigma$-algebra. 
	\end{theorem}

	Before we prove the theorem, let us take a deeper look at what it means for $f$ to make this triangle of $\cat{PS(Stoch)}$ commute:
	\[\begin{tikzcd}[sep=small]
		{(X, p)} \\
		&&& {(Y, q)} \\
		{(X, p)}
		\arrow["m"', from=1-1, to=3-1]
		\arrow["f", from=1-1, to=2-4]
		\arrow["f"', from=3-1, to=2-4]
	\end{tikzcd}\]
	This means exactly that $fm \aseq_p f$ for all $m$. In terms of string diagrams:
	\begin{equation}\label{harmonicity}
		\input{f_harmonic.tikz}
	\end{equation}
	and in terms of integrals, this means that for all measurable $B\subseteq Y$, we have $p$-almost surely that
	\begin{equation}\label{harmonic}
		\int_{X} f(B|x') \, m(dx'|x)  =  f(B|x) ,
	\end{equation}
	i.e.~for all measurable $A\subseteq X$,
	\begin{equation}\label{harmonicA}
		\int_A \int_{X} f(B|x') \, m(dx'|x)\,p(dx)  =  \int_A f(B|x) \,p(dx) .
	\end{equation}
	
	In terms of measure-theoretic probability, invariant kernels in this form are related to \emph{harmonic functions}. Recall that given a Markov kernel $m:X\to X$, a function $h:X\to\R$ is called \emph{harmonic} if for every $x\in X$ we have that
	$$
	\int_{X} h(x') \, m(dx'|x)  =  h(x)  .
	$$
	Intuitively, a harmonic function is a function that is ``spread fairly'' after a transition, or that is ``invariant in expectation'', and it defines a martingale canonically associated to the Markov chain. 
	(See any text on Markov chains, such as \cite[Section~4.1]{markovchains} or \cite[Section~17.1.2]{stochstability}.)
	
	Now a Markov kernel $f:X\to Y$ satisfies $fm=f$ (i.e.~condition \eqref{harmonic} for all $x$) if and only if for each measurable subset $B$ of $Y$, the function 
	\begin{align*}\label{defhB}
		f_B \colon X & \longrightarrow Y \\[-1ex]
		x & \longmapsto f(B|x)
	\end{align*}
	is harmonic for the kernel $m$. Similarly, $fm\aseq_pf$ if and only if the function $f_B$ defined above is \emph{$p$-almost surely harmonic}, i.e.~equation \eqref{harmonicA} holds for all $A\in\Sigma_A$.
	The proof of \Cref{xinvexists} will be related to the well-known fact that for stationary Markov chains, \emph{harmonic functions are almost surely invariant} (see the references above). Let us now state and prove this in a general, categorical setting. (While the idea and the proof are simple, the translation takes a little more work.)
	Just as invariant kernels are, in a sense, a categorical version of harmonic functions (indeed, their entries are harmonic functions, as we saw above), let us define a categorical analogue of \emph{invariant} functions.
	
	\begin{definition}\label{detinvdef}
		Consider an object $X$ in a causal Markov category with a state $p$, and a dynamical system $D$ in $\cat{PS(C)}$ acting on $(X, p)$.
		We call a morphism $f:X\to Y$ $p$-a.s.\ \emph{deterministically invariant} if and only if it does not depend, $p$-almost surely, on whether its input has transitioned or not,
		that is, if the following equality holds. 
		\begin{equation}\label{f invariant}
			\input{f_invariant.tikz}        
		\end{equation}
	\end{definition}
	
	Notice that we introduce the definition with the equality \eqref{f invariant} because it gives more intuition for the reason of this definition, but an equivalent (and easier to use) condition is:
	\begin{equation}\label{f invariant eq}
		\input{f_invariant_eq.tikz}
	\end{equation}
	
	In $\cat{Stoch}$, these kernels correspond exactly to the ones whose entries are $p$-almost surely invariant, i.e.~for which 
	$f(B|x)$ is almost surely equal to $f(B|x')$ for the measure $(A, B) \mapsto m(B|x)\,p(A)$ on $X\times X$. 
	
	Note that an a.s.\ deterministically invariant morphism is in particular a.s.\ invariant, as it can be seen by marginalizing \eqref{f invariant} over the last output:
	\ctikzfig{f_inv-2}
	In traditional probability theory, this reflects the fact that every invariant function is harmonic. To see this correspondence, note that a kernel has a.s.\ invariant entries if and only if it is a.s.\ deterministically invariant (and a.s.\ harmonic entries if and only if it is a.s.\ invariant as a kernel, see above). 
	
	Let us now use this to prove that under some conditions, every right-invariant Markov kernel is a.s.\ deterministically invariant; a categorical analogue of the fact that every harmonic function is invariant.\footnote{We thank Yuwen Wang for pointing out this fact to us.}
	
	We will use the following axiom, which first appeared in \cite[Appendix~A.5]{supports}.
	\begin{definition}
		\label{cauchy_schwarz}
		A Markov category $\cat{C}$ has the \emph{Cauchy-Schwarz property} if the following implication holds.
		\ctikzfig{cauchy_schwarz}
	\end{definition}
	
	See the original reference for interpretation and for motivation about the terminology. 
	
	\begin{lemma}\label{usingcauchy}
		Let $\cat{C}$ be a Markov category satisfying the Cauchy-Schwarz property. 
		Let $X$ be an object of $\cat{C}$, and let $m:X\to X$ be a morphism preserving a state $p$ on $X$. 
		Then every $p$-a.s.\ invariant morphism is $p$-a.s.\ deterministically invariant.
	\end{lemma}
	
	Note that since $\cat{Stoch}$ has the Cauchy-Schwarz property (\cite[Proposition A.5.2]{supports}), this lemma will apply to $\cat{Stoch}$ and $\cat{BorelStoch}$.
	
	\begin{proof}
		Recall that we can express the deterministic invariance condition as \eqref{f invariant eq}.
		
		Now assume that $fm\aseq_p f$. We can apply the Cauchy-Schwarz property as follows: we define $\tilde{f}, \tilde{g}, \tilde{h}$ to be the morphisms defined by the following string diagrams
		\ctikzfig{tildes_def_CS}
		This way, 
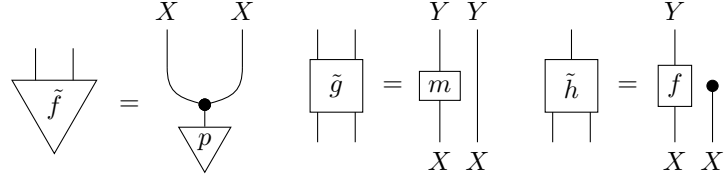
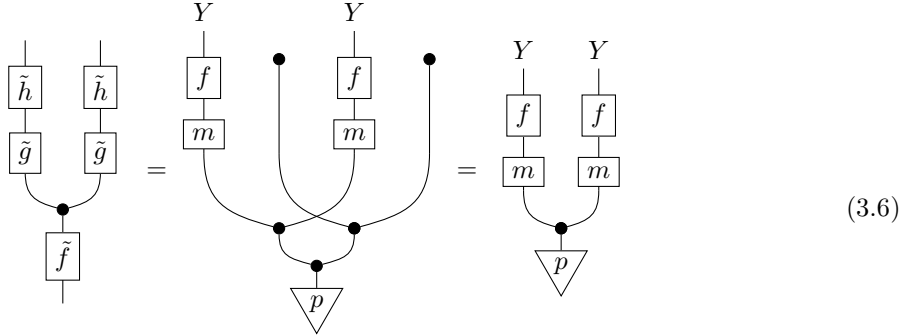
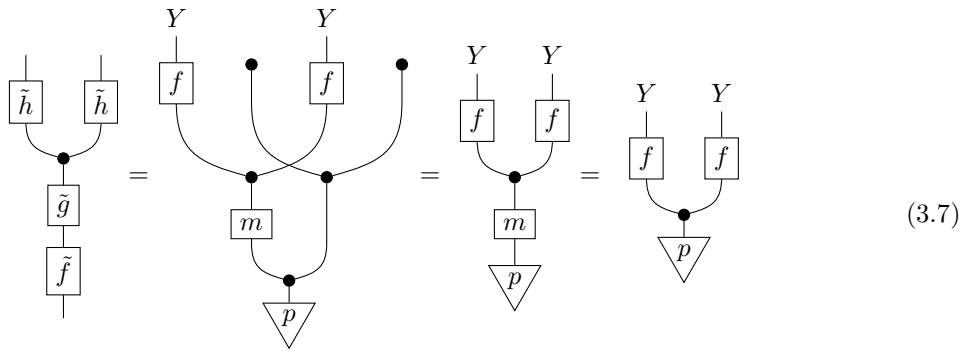
		\begin{equation}\label{CS cd 1}
			\input{CS_cond_1.tikz}
		\end{equation}
		and
		\begin{equation}\label{CS cd 2}
			\input{CS_cond_2.tikz}
		\end{equation}
		where the very last equality comes from the fact that $mp = p$. But the equality of these two diagrams follows from the fact that $fm\aseq_p f$. 
		The Cauchy-Schwarz property tells us that this diagram
		\begin{equation}\label{CS csq 1}
			\input{CS_csq_1.tikz}
		\end{equation}
		is equal to this diagram:
		\begin{equation}\label{CS csq 2}
			\input{CS_csq_2.tikz}
		\end{equation}
		Now we can rewrite \eqref{CS csq 1} as follows using $fm\aseq_p f$ and coassociativity,
		\ctikzfig{CS_csq_1_bis}
		so that equality with \eqref{CS csq 2} is (up to a reordering of the outputs) exactly the equality \ref{f invariant eq} we aimed to prove.
	\end{proof}

	It is worth pointing out that the lemma above does not admit a strict version (i.e.~where one replaces a.s.~equality with strict equality), as the following example shows:\footnote{We thank Sean Moss for suggesting this example.}
	\begin{example}\label{seanexample}
		Consider a finite set $A=\{a,b,c\}$ with the Markov chain given the transition matrix $m$ below, and the function 
		$h:A\to[0,1]$ (or equivalently, the kernel $A\to\{0,1\}$) also given below.
		\[
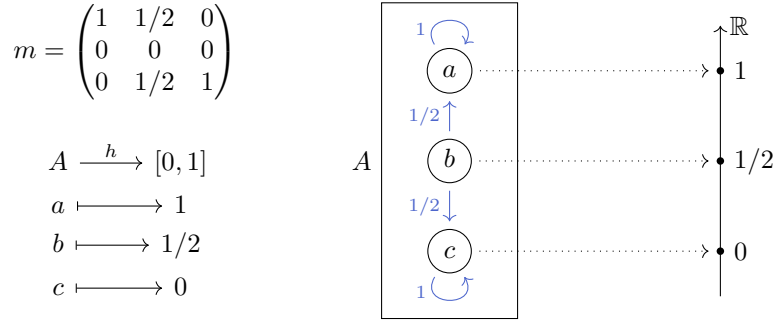
		\begin{aligned}
			m &= \left( \begin{matrix}
				1 & 1/2 & 0 \\
				0 & 0 & 0 \\
				0 & 1/2 & 1
			\end{matrix} \right) \\
			&\begin{tikzcd}[row sep=0]
				\phantom{a}\\
				A \ar{r}{h} & {[0,1]} \\ 
				a \ar[mapsto]{r} & 1 \\
				b \ar[mapsto]{r} & 1/2 \\
				c \ar[mapsto]{r} & 0 \\
			\end{tikzcd}
		\end{aligned}
		\qquad\qquad
		\begin{tikzpicture}[scale=1.2, baseline={(current bounding box.center)},
			transition/.style={midway,  font=\scriptsize, color=probcolor},]
			
			\node[circle, draw=black, inner sep=0.5mm] (c) at (0,0) {$c\strut$};
			\node[circle, draw=black, inner sep=0.5mm] (b) at (0,1) {$b\strut$};
			\node[circle, draw=black, inner sep=0.5mm] (a) at (0,2) {$a\strut$};
			\node[inner sep=0.6cm, draw=black,fit=(a) (c),label=left:$A$] () {} ;
			
			\node[bullet, inner sep=1pt, label=right:$0$] (0) at (3,0) {};
			\node[bullet, inner sep=1pt, label=right:$1/2$] (mid) at (3,1) {};
			\node[bullet, inner sep=1pt, label=right:$1$] (1) at (3,2) {};
			\draw[->] (3,-0.5) -- (3,2.5) node[right] {$\R$};
			
			\draw[ar, dotted] (a) to (1) {} ;
			\draw[ar, dotted] (b) to (mid) {} ;
			\draw[ar, dotted] (c) to (0) {} ;
			
			\draw[ar,probcolor] (a) to [out=120,in=60, looseness=5] node[transition, near start, left] {1} (a);
			\draw[ar,probcolor] (b) to node[transition, left] {1/2} (a);
			\draw[ar,probcolor] (b) to node[transition, left] {1/2} (c);
			\draw[ar,probcolor] (c) to [out=-120,in=-60, looseness=5] node[transition, near start, left] {1} (c);
		\end{tikzpicture}
		\]
		The function $h$ is harmonic (equivalently, the corresponding kernel is invariant):
		\begin{align*}
			h(b) &= \dfrac{1}{2} + 0 + 0 \\
			&= h(a)\,m(b|a) + h(b)\,m(b|b) + h(c)\,p(c|a) ,
		\end{align*}
		but not invariant (equivalently, the corresponding kernel is not deterministically invariant), since in the convex combination above, 
		the term $m(b|b)$ is zero.
		However, every invariant measure $p$ is necessarily supported on $\{a,c\}$, and assigns to $b$ probability zero, so that $p$-almost surely, $h$ is invariant.
	\end{example}

	\begin{lemma}\label{detinvthm}
		A morphism $f : (X,p) \to (Y, q)$ in $\cat{PS(Stoch)}$ is $p$-a.s. deterministically invariant if and only if $f(B|-)$ is measurable for the invariant $\sigma$-algebra generated by $m$. 
	\end{lemma}
	
	\begin{proof}
		Once again, write $f_B(x)$ for $f(B|x)$ where $B \in \Sigma_Y$ and $x \in X$.
		First of all, in $\cat{Stoch}$, $f$ being $p$-a.s. deterministically invariant is equivalent to $f_B$ being a $p$-a.s.\ invariant function for each $B\in\Sigma_Y$. 
		So let $B \in \Sigma_Y$.
		
		Suppose that $f_B(x') = f_B(x)$ holds $m(dx'|x)\,p(dx)$-almost surely (which corresponds exactly to $f_B$ being $p$-a.s.\ deterministically invariant). Then the aim is to prove that for every (Borel) measurable subset $C\in\Sigma_R$, the subset $f_B^{-1}(C)$ is ($p$-almost surely) invariant, that is,
		\begin{equation}\label{measinv}
			m(f_B^{-1}(C) | x) = 1_{f_B^{-1}(C)}(x) = 1_C(f_B(x))
		\end{equation}
		for $p$-almost all $x$.
		Now for all $A\in\Sigma_X$,
		\begin{equation*}
			\begin{split}
				\int_A m(f_B^{-1}(C) | x) \,p(dx) &=  \int_A \left( \int_{f_B^{-1}(C)} m(dx' | x) \right) p(dx) \\
				&= \int_A \int_X 1_{C}(f_B(x')) \,m(dx' | x) \,p(dx) \\
				&= \int_A \int_X 1_{C}(f_B(x)) \,m(dx' | x) \,p(dx) \\
				&= \int_A 1_{C}(f_B(x)) \left( \int_X m(dx' | x) \right) p(dx) \\
				&= \int_A 1_{f_B^{-1}(C)}(x) \,p(dx) ,
			\end{split}
		\end{equation*}
		where the third equality uses a.s.\ deterministic invariance of $f_B$.
		Therefore $f_B$ is measurable for the invariant $\sigma$-algebra.
		
		Conversely, assume $f_B$ is measurable for the invariant $\sigma$-algebra. The goal is to prove that $f_B(x) = f_B(x')$ holds $m(dx'|x)\,p(dx)$-a.s., i.e.~the sets $X \times f_B^{-1}(C)$ and $f_B^{-1}(C) \times X$ must be equal up to measure-zero (for this measure on $X \times X$) for all measurable subsets $C$ of $\R$. Let us evaluate the difference of these two sets: first consider $\left( X \times f_B^{-1}(C) \right) \setminus \left( f_B^{-1}(C) \times X \right) = \left( X \setminus f_B^{-1}(C) \right) \times f_B^{-1}(C)$. The measure of interest evaluates this set to:
		\begin{equation*}
			\begin{split}
				\int_{X \setminus f_B^{-1}(C)} m(f_B^{-1}(C) | x) p(dx) &= \int_{X \setminus f_B^{-1}(C)} 1_{f_B^{-1}(C)}(x) p(dx) \\
				&= p\left( \left( X \setminus f_B^{-1}(C) \right) \cap f_B^{-1}(C) \right) = 0
			\end{split}
		\end{equation*}
		where the first equality holds because $f_B$ is measurable for the invariant $\sigma$-algebra, and so $f_B^{-1}(C)$ is an invariant set.
		
		For the other difference, $f_B^{-1}(C) \times \left( X \setminus f_B^{-1}(C) \right)$, the argument is analogous.
		Therefore the symmetric difference of $X \times f_B^{-1}(C)$ and $f_B^{-1}(C) \times X$ also has measure zero.
	\end{proof}

	We are now ready to prove the main theorem. 
	\begin{proof}[Proof of \Cref{xinvexists}]
		Let $f:(X,p)\to (Y,q)$ be a $p$-almost surely invariant morphism, i.e.~making the outer triangle in the following diagram commute $p$-almost surely.
		\[\begin{tikzcd}
			{(X, p)} \\
			& {(X_\inv, p_\inv)} && {(Y, q)} \\
			{(X, p)}
			\arrow["r"', shift right, from=1-1, to=2-2]
			\arrow["r", shift left, from=3-1, to=2-2]
			\arrow["m"', from=1-1, to=3-1]
			\arrow["f", curve={height=-12pt}, from=1-1, to=2-4]
			\arrow["f"', curve={height=12pt}, from=3-1, to=2-4]
			\arrow["{\tilde{f}}"', dashed, from=2-2, to=2-4]
		\end{tikzcd}\]
		Denote by $X_\inv$ the set $X$ equipped with the (almost surely) invariant $\sigma$-algebra, and denote by $r:X\to X_\inv$ the kernel induced by the set-theoretical identity: for every $x\in X$ and every invariant set $B$,
		$$
		r(B|x) \coloneqq 1_B(x) .
		$$
		Denote also by $p_\inv$ the restriction of $p$ to invariant sets. The triangle involving $r$ in the diagram above commutes $p$-almost surely: for every invariant set $B$ and every measurable $A\subseteq X$,
		$$
		\int_A\int_X r(B|x')\,m(x'|x)\,p(dx) = \int_A m(B|x)\,p(dx) = \int_A 1_B(x)\,p(dx)
		$$
		precisely by invariance of $B$. 
		Notice that, since $r$ is the set-theoretical identity, there is a unique (almost surely) possible kernel $\tilde{f}:X_\inv\to Y$ making the diagram above commute, namely, set-theoretically, the same as $f$:
		$$
		\tilde{f}(B|x) \coloneqq f(B|x) 
		$$
		for every measurable $B\subseteq Y$, $p_\inv$-almost surely in $x\in X$ . 
		It remains to show that $\tilde{f}$ is indeed a Markov kernel $X_\inv\to Y$, i.e.~that it is measurable for the invariant $\sigma$-algebra.
		Now since $f$ is ($p$-a.s.) invariant for every $m$, so that by \Cref{usingcauchy} it is deterministically invariant for every $m$, and by \Cref{detinvthm} it is measurable for the invariant $\sigma$-algebra induced by each $m$. Moreover, by construction, $\tilde{f}$ maps $p_\inv$ to $q$. In other words, $\tilde{f}$ is the unique morphism $(X_\inv,p_\inv)\to(Y,q)$ of $\cat{PS(Stoch)}$ making the diagram above commute, and so $(X_\inv,p_\inv)$ is a colimit of $D$ with colimiting cocone $r$.
		
		It now remains to prove that $f$ is a.s.\ deterministic if and only if $\tilde{f}$ is. Now suppose that $f$ is a.s.\ deterministic. 
		Since $f\aseq \tilde{f}\circ r$, and since $r$ is deterministic, this means that 
		\ctikzfig{asdetf}
		Composing with $r$ on the first output, and using again determinism of $r$, we get that
		\ctikzfig{asdetf2}
		i.e.~that $\tilde{f}$ is $p_\inv$-a.s.\ deterministic (recall that $p_\inv=r\circ p)$. 
		Conversely, if $\tilde{f}$ is $p_\inv$-a.s.~deterministic, then $f \aseq_p \tilde{f} \circ r$ is $p$-a.s.~deterministic as the composition of two almost surely deterministic morphisms (recall lemma \Cref{a-s det composition}).
	\end{proof}
	
	As we will show in the next section, if $X$ is standard Borel, one can take as invariant object $X_\inv$ a standard Borel space too.

	\subsection{Idempotents in $\cat{PS(Borel)}$}\label{idempotents}
	
	In this section we focus on invariant objects of \emph{idempotent} stationary Markov chains.
	Recall that a morphism $e:X\to X$ in a category is called \emph{idempotent} if $e\circ e=e$. 
	Any idempotent morphism $e:X\to X$ defines an $\N$-indexed dynamical system acting on $X$ via $n\mapsto e$ for $n\ge 1$, and $0\mapsto\id_X$. 
	For this dynamical system, its colimits are particularly convenient, and coincide with its limits. Let us see how.
	Recall that a \emph{splitting} of an idempotent $e:X\to X$ consists of an object $E$ and maps $\iota:E\to X$ and $\pi:X\to E$ such that $e=\iota\circ\pi$ and $\pi\circ\iota=\id_E$. Consider now the following statement:
	
	\begin{proposition}[{\cite[Proposition~1]{cauchycompletion}}]\label{eqsplit}
		For an idempotent $e:X\to X$, the following are equivalent.
		\begin{itemize}
			\item $e$ has a splitting $(E,\iota,\pi)$;
			\item The pair $(e,\id_X)$ has an equalizer
			\begin{equation*}
				\begin{tikzcd}
					E \ar{r}{\iota} & X \ar[shift left]{r}{e} \ar[shift right]{r}[swap]{\id} & X ;
				\end{tikzcd}
			\end{equation*}
			\item The pair $(e,\id_X)$ has a coequalizer
			\begin{equation*}
				\begin{tikzcd}
					& X \ar[shift left]{r}{e} \ar[shift right]{r}[swap]{\id} & X \ar{r}{\pi} & E .
				\end{tikzcd}
			\end{equation*}
		\end{itemize}
		Moreover, the equalizer and coequalizer above, if they exist, are absolute (i.e.~preserved by every functor).
	\end{proposition}
	
	If $\cat{C}$ is a causal Markov category and $e:(X,p)\to (X,p)$ is idempotent, which we see as an $\N$-indexed dynamical system, the proposition above implies that if $X_\inv$ exists, then it is not only the colimit of the system, but also the limit. 
	
	Under some conditions, every dynamical system, indexed by an arbitrary monoid, induces an idempotent dynamical system, which can often be interpreted as ``averaging over the orbits''. Let us look at this in detail (further insight will be given in \Cref{equilibrium}). 
	
	\begin{definition}\label{defeD}
		Let $\cat{C}$ be a causal Markov category.
		Consider a dynamical system $D$ in $\cat{PS(C)}$, indexed by a monoid $M$, acting on $(X,p)$.
		Suppose moreover that the invariant object $X_\inv$ exists, and that $r:X\to X_\inv$ has a Bayesian inverse $r^\dag:X_\inv\to X$ relative to $p$.
		The \emph{idempotent associated to $D$} (or \emph{transition to equilibrium}) is the morphism $e_D:(X,p)\to(X,p)$ given by the following composition.
		\begin{equation*}
			\begin{tikzcd}
				(X,p) \ar{r}{r} & (X_\inv,p_\inv) \ar{r}{r^\dag} & (X,p).
			\end{tikzcd}
		\end{equation*}
	\end{definition}
	
	When $\cat{C}$ is $\cat{Stoch}$, this definition can be applied to objects $X$ in $\cat{BorelStoch}$ since they have disintegrations (see \cite[Section 10.6]{BogachevVladimirI2007MT}).
	
	The name ``transition to equilibrium'' will be explained in \Cref{equilibrium}, with further intuition. For a concrete example, the interested reader can look at \Cref{finiteperm} before continuing this section. 
	Let us now see some immediate consequences of the definition, and of \Cref{eqsplit}.
	
	First of all, since $r$ is $p$-a.s.~deterministic, we have that $r\circ r^\dag$ is a.s.~equal to the identity of $(X_\inv,p_\inv)$ by \Cref{dag id}, so that $e_D$ is a split idempotent, with splitting $(X_\inv,r^\dag,r)$. 
	When $\cat{C}$ has conditionals, and so $\cat{PS(C)}$ is a dagger category, a split idempotent in this form, where the section and the retraction are the dagger of each other, is called a \emph{dagger-split idempotent}, see \cite{daggeridempotents}. 
	
	\begin{corollary}\label{invagain}
		If we consider $e_D:(X,p)\to(X,p)$ as (generating) the dynamical system $\N\to\cat{PS(C)}$ (via $1\mapsto e_D$), then by \Cref{eqsplit} its invariant object is again $X_\inv$. 
	\end{corollary}
	
	This in particular implies that 
	\begin{itemize}
		\item every idempotent $e:(X,p)\to(X,p)$ of $\cat{PS(C)}$ is trivially in the form $e_D$ for some $D$ (take $D$ to be the one generated by $e$ itself);
		\item every invariant object $X_\inv$ of every $D$ is given by the splitting of some idempotent (take $e_D$).
	\end{itemize}
	
	We can now use this formalism to prove that if $(X,p,M,D)$ is a dynamical system in $\cat{PS(Borel)}$, the object $X_\inv$ can be taken to be standard Borel (up to isomorphism of $\cat{PS(Stoch)}$), so that $\cat{PS(Borel)}$ is closed under taking invariant objects.
	We will proceed as follows. 
	It was proven in \cite[Corollary~4.4.5]{supports} that all idempotents of $\cat{BorelStoch}$ split.
	We can use that fact to prove that all idempotents of $\cat{PS(Borel)}$ are split as well. 
	This implies that for every idempotent $e:(X,p)\to (X,p)$, its invariant object $X_\inv$ is standard Borel (up to isomorphism of $\cat{PS(Stoch)}$).
	Since for every $D$ we can obtain $X_\inv$ equivalently as the invariant object of an idempotent, we get that $X_\inv$ is standard Borel (up to isomorphism of $\cat{PS(Stoch)}$) for every dynamical system $D$.
	
	\begin{theorem}\label{idempsplit}
		All idempotents split in $\cat{PS(Borel)}$. 
	\end{theorem}
	
	We prove the theorem by means of the following lemmas. 
	
	\begin{lemma}\label{absorbingset}
		Let $p$ be a probability measure on a standard Borel space $X$, and let $k:(X,p)\to(X,p)$ be a measure-preserving kernel. Then every measure one subset $A_0\subseteq X$ admits an \emph{absorbing} subset of full measure, i.e.~a subset $A\subseteq A_0$ such that
		\begin{itemize}
			\item $p(A)=1$;
			\item for all $a\in A$ we have $k(A|a)=1$. 
		\end{itemize}
		Note that the last property holds for all $a\in A$, not just for almost all.
	\end{lemma}
	\begin{proof}[Proof of~\Cref{absorbingset}]
		Let $A_0$ be a subset of measure one.
		Since $k$ preserves the measure $p$, and since $A_0$ has measure one, we have that 
		\begin{equation*}
			1 = p(A_0) = \int_X k(A_0|x)\,p(dx) ,
		\end{equation*}
		which means that $k(A_0|x)=1$ for $p$-almost all $x$. 
		Therefore there exists a measurable subset $A_1\subseteq X$ with $p(A_1)=1$, and such that for all $a\in A_1$, $k(A_0|a)=1$.
		
		Once again, since $k$ preserves the measure $p$, and since $A_1$ has measure one, we have that 
		\begin{equation*}
			1 = p(A_1) = \int_X k(A_1|x)\,p(dx) ,
		\end{equation*}
		which means that $k(A_1|x)=1$ for $p$-almost all $x$. 
		Therefore there exists a measurable subset $A_2\subseteq X$ with $p(A_2)=1$, and such that for all $a\in A_2$, $k(A_1|a)=1$.
		
		We can go on like this for all $n\in\N$: each step takes a set $A_n$ of measure one, and since $k$ preserves the measure $p$, we have that 
		\begin{equation*}
			1 = p(A_n) = \int_X k(A_n|x)\,p(dx) ,
		\end{equation*}
		which means that $k(A_n|x)=1$ for $p$-almost all $x$. 
		Therefore there exists a measurable subset $A_{n+1}\subseteq X$ with $p(A_{n+1})=1$, and such that for all $k\in A_{n+1}$, $k(A_n|a)=1$.
		
		Take now the countable intersection
		$$
		A \coloneqq \bigcap_{n=0}^\infty A_n .
		$$
		Since a countable intersection of sets of measure $1$ has measure $1$, we have that $p(A)=1$.
		Now let $a\in A$. For each $n\in\N$,
		$$
		k(A_n|a) = 1 ,
		$$
		since in particular $a\in A_{n+1}$. 
		So the probability measure $k(-|a)$ gives measure $1$ to all the $A_n$, and so it must give measure $1$ to their countable intersection, i.e.~$k(A|a)=1$.
	\end{proof}
	
	\begin{lemma}\label{asidemp}
		Let $p$ be a probability measure on a standard Borel space $X$.
		Every $p$-almost surely idempotent kernel on $(X,p)$ which preserves $p$ is $p$-almost surely equal to an idempotent kernel.
	\end{lemma}
	\begin{proof}[Proof of \Cref{asidemp}]
		First of all, the condition that $ee \aseq_p e$ means that there exists a measurable subset $A_0\subseteq X$ with $p(A_0)=1$, and such that for all $x\in A_0$ and all measurable $B\subseteq X$, we have $ee(B|x)=e(B|x)$. 
		By \Cref{absorbingset} there exists a measure one subset $A\subseteq A_0$ such that $e(A|a)=1$ for each $a\in A$. 
		
		Define now the kernel $e':X\to X$ as follows, for all $x\in X$ and all measurable $B\subseteq X$
		\begin{equation*}
			e'(B|x) \coloneqq 
			\begin{cases}
				e(B|x) & x\in A ; \\
				\delta(B|x) = 1_B(x) & x\notin A .
			\end{cases}        
		\end{equation*}
		Since $p(A)=1$, we have that $e'\aseq e$. 
		Now, for each $a\in A$ and each measurable $B\subseteq X$,
		\begin{align*}
			(e'\circ e') (B|a) &= \int_X e'(B|x')\,e'(dx'|a) \\
			&= \int_X e'(B|x')\,e(dx'|a) \\
			&= \int_A e'(B|x')\,e(dx'|a) \\
			&= \int_A e(B|x')\,e(dx'|a) \\
			&= \int_X e(B|x')\,e(dx'|a) \\
			&= (e\circ e)(B|a) \\
			&= e(B|a) \\
			&= e'(B|a) .
		\end{align*}
		On the complement, i.e.~for each $x\in X\setminus A$ and each measurable $B\subseteq X$,
		\begin{align*}
			(e'\circ e') (B|x) &= \int_X e'(B|x')\,e'(dx'|x) \\
			&= \int_X e'(B|x')\,\delta(dx'|x) \\
			&= e'(B|x) ,
		\end{align*}
		so for all $x\in X$, and not just $p$-almost surely, and for all measurable $B\subseteq X$,
		$$
		(e'\circ e') (B|x) = e'(B|x) ,
		$$
		that is, $e':X\to X$ is idempotent. 
	\end{proof}
	
	\begin{proof}[Proof of \Cref{idempsplit}]
		Let $e:(X,p)\to(X,p)$ be an idempotent of $\cat{PS(Borel)}$. This means that $e:X\to X$ preserves the measure $p$, and that $e\circ e\aseq_p e$.
		By \Cref{asidemp} there exists a (strictly, not just a.s.)\ idempotent kernel $e':X\to X$ such that $e\aseq e'$. Therefore, without loss of generality, we can assume that our kernel $e$ is idempotent (in $\cat{BorelStoch}$). 
		Since all idempotents in $\cat{BorelStoch}$ split, in particular $e$ does, so that we can write it as $e=\iota\circ\pi$ for $\iota:E\to X$ and $\pi:X\to E$, where $E$ is standard Borel, and $\pi\circ\iota=\id_E$. 
		
		Let now $q=\pi\circ p$, and consider the following diagram.
		\begin{equation*}
			\begin{tikzcd}
				(X,p) \ar[shift left]{r}{\pi} \ar[loop left]{}{e} & (E,q) \ar[shift left]{l}{\iota} \ar[loop right]{}{\id_E}
			\end{tikzcd}
		\end{equation*}
		We have that $\iota\circ q = \iota\circ\pi\circ p = e\circ p = p$, so that the diagram above is a diagram in $\cat{PS(Borel)}$.
		Moreover, $e=\iota\circ\pi$, and $\pi\circ\iota=\id_{(E,p)}$. 
		Therefore $(E,q)$, together with $\iota$ and $\pi$, splits the idempotent $e$ in $\cat{PS(Borel)}$ as well.
	\end{proof}

	\begin{corollary}\label{xinvborel}
		Let $X$ be a standard Borel space, and let $D$ be a dynamical system acting on $(X,p)$ indexed by a monoid $M$ of arbitrary cardinality.
		Then $(X_\inv,p_\inv)$, i.e.~the set $X$ equipped with the almost surely invariant $\sigma$-algebra, is a standard Borel space up to isomorphism of $\cat{PS(Stoch)}$.
	\end{corollary}
	\begin{proof}
		Form the idempotent $e_D:X_\inv\to X_\inv$. By \Cref{invagain}, we can obtain the invariant object $X_\inv$ of $D$ equivalently as the coequalizer of $e_D$ and the identity. But this coequalizer is exactly the splitting of the idempotent $e_D$, which by \Cref{idempsplit} is standard Borel (up to isomorphism of $\cat{PS(Stoch)}$).
	\end{proof}

	Here is a measure-theoretic consequence.
	\begin{corollary}\label{subsigma}
		Let $(X,p)$ be a standard Borel space. Let $(X',p')$ be obtained by equipping $X$ with a smaller (coarser) $\sigma$-algebra, with $p'$ the restriction of $p$. Then $(X',p)$ is isomorphic to a standard Borel space up to isomorphism of $\cat{PS(Stoch)}$.
	\end{corollary}
	\begin{proof}
		Denote by $r:X\to X'$ the set-theoretical identity. 
		$X'$ is the splitting of the idempotent $r^\dag\circ r$, so we can apply \Cref{idempsplit}.
	\end{proof}

	\subsection{The invariant $\sigma$-algebra as a limit}\label{limits}
	
	\Cref{cond implies dag} shows that when a Markov category $\cat{C}$ has conditionals, $\cat{PS(C)}$ is a dagger category.
	In this context, the invariant object $X_\inv$ and the idempotent $e_D$ can be interpreted in terms of the dagger. 
	The most important result of this paragraph is that $X_\inv$ is not only a \emph{colimit}, but also a \emph{limit}:
	
	\begin{theorem}\label{xinvlimit}
		Let $\cat{C}$ be a Markov category with conditionals, and such that $\cat{PS(C)}$ has all invariant objects (as for example for $\cat{C}=\cat{BorelStoch}$). 
		Consider a dynamical system $D$ in $\cat{PS(C)}$ on the object $(X,p)$, with monoid $M$.
		
		Then the colimit $(X_\inv, p_\inv)$ is also the \emph{limit} of $D$ in $\cat{PS(C)}$, with limiting cone given by $r^\dag:(X_\inv, p_\inv)\to(X,p)$, the Bayesian inverse of the colimiting cone $r:(X,p)\to(X_\inv, p_\inv)$.
	\end{theorem}
	
	Note that we are not requiring, as we did in \Cref{defxinv}, that the limit is compatible with a.s.\ deterministic morphisms (but one can see an analogous compatibility condition on the Bayesian inverses).
	Also, notice that the statement does not just follow from the dagger structure of $\cat{PS(C)}$: the dagger structure, rather, tells us that the limit of $D$ is the colimit of the \emph{opposite} diagram of $D$, the diagram $\dagger\circ D^\op:M^\op\to\cat{PS(C)}$. 
	
	We will now prove the theorem using some auxiliary results.
	In the hypotheses of the theorem, let $f:(X,p)\to(Y,q)$ be an invariant morphism for the dynamical system $D$ (i.e.~cocone).
	Because $(X_\inv, p_\inv)$ is a colimit, $f$ uniquely factors through $(X_\inv, p_\inv)$ and $r$ via a unique $\tilde{f}$. 
	Moreover, since $\cat{C}$ has conditionals, we have a Bayesian inverse $r^\dag:(X_\inv,p_\inv)\to(X,p)$.
	The setting can be summarised in the following diagram:
	
	\[\begin{tikzcd}
		{(X, p)} \\
		& {(X_\inv, p_\inv)} && {(Y, q)} \\
		{(X, p)}
		\arrow["r", shift left, from=1-1, to=2-2]
		\arrow["{r^\dag}", shift left, from=2-2, to=1-1]
		\arrow["r"', shift right, from=3-1, to=2-2]
		\arrow["{r^\dag}"', shift right, from=2-2, to=3-1]
		\arrow["{m}", from=1-1, to=3-1]
		\arrow["f", curve={height=-12pt}, from=1-1, to=2-4]
		\arrow["f"', curve={height=12pt}, from=3-1, to=2-4]
		\arrow["{\tilde{f}}"', dashed, from=2-2, to=2-4]
	\end{tikzcd}\]
	
	We can notice that necessarily, $\tilde{f} \aseq_{p_\inv} f r^\dag$.
	Indeed, since $r^\dag$ is the Bayesian inverse of the $p$-a.s.~deterministic morphism $r$, by \Cref{dag id} we have $r r^\dag \aseq \id_{X_\inv}$, so that
	$$
	f r^\dag \aseq \tilde{f} r r^\dag \aseq \tilde{f} . 
	$$
	
	\begin{lemma}\label{rdaginv}
		In the hypotheses of \Cref{xinvlimit}, for all $m\in M$ we have that $r^\dag \aseq_{p_\inv} m r^\dag$, making the triangle of the diagram above involving $m$ and $r^\dag$ commute.
	\end{lemma}
	This statement can be seen as a generalization of part of the proof of \cite[Theorem~3.15]{moss2022ergodic} for the case where the morphisms $m$ are not necessarily deterministic.
	As such, it will play a role in a version of the ergodic decomposition theorem (\Cref{erg_dec}).
	
	\begin{proof}
		Note first that we have $rm \aseq_p r$ and so $rm$ is $p$-a.s.~deterministic. This fact, together with the relative positivity of $\cat{C}$ (which we have assumed to be causal), allows us to derive the following diagram:
		\ctikzfig{rm_relative_positivity}
		which can be marginalized to obtain:
		\begin{equation}\label{rm RP csq}
			\input{rm_relative_positivity_csq.tikz}
		\end{equation}
		Using the dagger properties and \eqref{rm RP csq}, we have the following equalities:
		\ctikzfig{r_dag}
		But now we can use the fact that $rm \aseq_p r$, and so:
		\ctikzfig{m_r_dag}
		and so we obtain the string-diagrammatic proof that $r^\dag \aseq_{p_\inv} m r^\dag$.
	\end{proof}

	A convenient way to prove the theorem is by looking at the time-reversed dynamics, which we define here. 
	
	\begin{definition}
		Given a dynamical system $D$ on a category $\cat{E}$ with monoid $M$, we define the \emph{time-reversed dynamical system} $D^\dag$ to be the dynamical system on $\cat{E}$ with monoid $M^\op$ (formally, $D^\dag$ is the functor $\dag \circ D^\op : M^\op \to \cat{E}^\op \to \cat{E}$).
	\end{definition}
	
	Notice that if a morphism $m:(X,p)\to (X,p)$ of $\cat{PS(C)}$ preserves the state $p$, so does its dagger $m^\dag$ (since by \Cref{cond implies dag}, $\cat{PS(C)}$ is a dagger category).
	Moreover, the time-reversed dynamics is also compatible with those invariant morphisms which are almost surely deterministic (and one can think of as ``invariant observables''), as the following known statement shows.
	
	\begin{lemma}[{\cite[Lemma~5.2]{fritz2021definetti}}]\label{det preserves backwards}
		Let $\cat{C}$ be a Markov category with conditionals.
		Let $m:X\to X$ be a morphism preserving the state $p$ on $X$ (i.e.~giving a morphism $(X,p)\to(X,p)$ of $\cat{PS(C)}$), and take a Bayesian inverse $m^\dagger:X\to X$. If $f:X\to Y$ is a $p$-almost surely deterministic morphism, and it is a.s.~invariant for the reversed dynamics, that is,
		$$
		f m^\dag \aseq_p f ,
		$$
		then it is also a.s.~invariant for the forward dynamics:
		$$
		f m \aseq_p f .
		$$
	\end{lemma}
	
	By taking $m^\dag$ instead of $m$, one can similarly prove that if $fm \aseq_p f$, then $fm^\dag = f$. 
	
	In the language of \cite{wayofdagger}, we are saying that $X_\inv$ is also the dagger limit of the functor $\cat{ZigZag}(M)\to\cat{PS(C)}$ corresponding to $D:M\to\cat{PS(C)}$ (see the reference for the terminology).
	Equivalently, this is saying that if all invariant objects exist, then the canonical functor $\cat{PS(C)}\to\cat{PS(C)}^M$ has an \emph{ambidextrous} adjoint.\footnote{We thank Tobias Fritz for pointing this out.}
	
	Now, and throughout this section, denote by $(\overbar{X_\inv}, \overbar{p_\inv})$ the invariant object for the dynamical system $D^\dag$. (It exists in the hypotheses of \Cref{xinvlimit}, since we are assuming that all invariant objects exist.) As we will prove, from the dagger structure it follows that this object is the \emph{limit} of $D$ in $\cat{PS(C)}$, and from the lemma above, we will see that it is in fact also isomorphic (in $\cat{PS(C)}$) to $(X_\inv, p_\inv)$ (the colimit for $D$).

	\begin{lemma}\label{colimit reversed}
		Let $(X_\inv, p_\inv)$ be the colimit for the dynamics $D$, and $(\overbar{X_\inv}, \overbar{p_\inv})$ be the colimit for the reversed dynamics $D^\dag$. Then $(X_\inv, p_\inv)$ and $(\overbar{X_\inv}, \overbar{p_\inv})$ are isomorphic up to a unique isomorphism commuting with the colimiting cocones.
	\end{lemma}
	
	In other words, up to isomorphism of $\cat{PS(C)}$, $D$ and $D^\dagger$ have the same colimit.
	Note that we are \emph{not} saying that it follows from the definitions that the functor $\dag:\cat{PS(C)}^\op\to\cat{PS(C)}$ preserves colimits. Since the dagger is a contravariant functor, it turns colimits for the dynamics $D$ into limits for the time-reversed dynamics $D^\dag$ (by applying the dagger to the whole diagram). What we are proving here is that it turns out the two colimits for the different dynamics $D$ and $D^\dag$ are actually isomorphic (up to a unique isomorphism).
	
	\begin{proof}
		Let $r : (X, p) \to (X_\inv, p_\inv)$ and $s : (X, p) \to (\overbar{X_\inv}, \overbar{p_\inv})$ be the colimiting cocones of $D$ and $D^\dag$, respectively, and note that both are almost surely deterministic. 
		
		Thanks to \Cref{det preserves backwards}, we have that $r$, which is right-invariant for $D$, is right-invariant also for the time-reversed dynamics $D^\dag$, and so there exists a unique $\tilde{r} : (\overbar{X_\inv}, \overbar{p_\inv}) \to (X_\inv, p_\inv)$ such that $r \aseq \tilde{r} s$. 
		
		Similarly, $s$ is right-invariant not only for $D^\dag$, but also for $D$, and so there exists a unique $\tilde{s} : (X_\inv, p_\inv) \to (\overbar{X_\inv}, \overbar{p_\inv})$ such that $s \aseq \tilde{s} r$. 
		
		This means that all the triangles in the following diagram commute:
		\[\begin{tikzcd}
			{(X, p)} \\
			& {(X_\inv, p_\inv)} && {(\overbar{X_\inv}, \overbar{p_\inv})} && {(X_\inv, p_\inv)} \\
			{(X, p)}
			\arrow["r", shift right, from=1-1, to=2-2]
			\arrow["r"', shift left, from=3-1, to=2-2]
			\arrow["{m^\dag}", shift left, from=3-1, to=1-1]
			\arrow["s"', curve={height=-6pt}, from=1-1, to=2-4]
			\arrow["{\tilde{s}}"{description}, dashed, from=2-2, to=2-4]
			\arrow["m", shift left, from=1-1, to=3-1]
			\arrow["{\tilde{r}}"{description}, dashed, from=2-4, to=2-6]
			\arrow["r", curve={height=-12pt}, from=1-1, to=2-6]
			\arrow["s", curve={height=6pt}, from=3-1, to=2-4]
			\arrow["r"', curve={height=12pt}, from=3-1, to=2-6]
		\end{tikzcd}\]
		
		Now, $\tilde{r} \circ \tilde{s} : (X_\inv, p_\inv)$ is a morphism that makes the diagram commute, but so does $\id_{X_\inv} : (X_\inv, p_\inv) \to (X_\inv, p_\inv)$, and so, by uniqueness, $\tilde{r} \circ \tilde{s} \aseq \id_{X_\inv}$. 
		
		Similarly, if we swap the roles of $r$ and $s$ (and of $(X_\inv, p_\inv)$ and $(\overbar{X_\inv}, \overbar{p_\inv})$), we obtain that $\tilde{s} \circ \tilde{r} \aseq \id_{\overbar{X_\inv}}$, and so $(X_\inv, p_\inv)$ and $(\overbar{X_\inv}, \overbar{p_\inv})$ are isomorphic up to an a.s.\ unique isomorphism commuting with $r$ and $s$.
	\end{proof}
	
	As a consequence, we can extend \Cref{det preserves backwards} to morphisms which are not necessarily a.s.~deterministic:\footnote{In order to generalize from \Cref{det preserves backwards} we have used the additional assumption that invariant objects exist.}
	
	\begin{corollary}\label{right inv preserves backwards}
		Under the hypotheses of \Cref{xinvlimit}, let $f : X\to Y$ be any morphism.
		Then $f$ is $p$-almost surely invariant for $D$, i.e.~for all $m\in M$, $f m\aseq_p f$ if and only if it is $p$-almost surely invariant for $D^\dag$, i.e.~$f m^\dag \aseq_p f$.
	\end{corollary}
	
	We are now ready to prove the main theorem of this section.
	
	\begin{proof}[Proof of \Cref{xinvlimit}]
		The setting we are looking at now is as follows: suppose we have $g : (A, q) \to (X, p)$ that is (a.s.)~\emph{left}-invariant, i.e.~$mg \aseq g$ for all $m\in M$. 
		We have to prove that there exists a unique (almost surely) morphism $\tilde{g}:(A,q)\to (X_\inv,p_\inv)$ such that $g\aseq r^\dag\tilde{g}$, as in the following diagram.
		
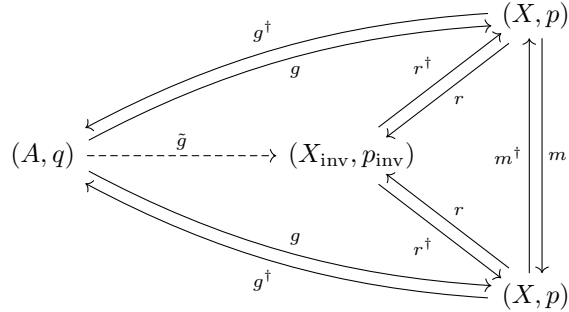
\[\begin{tikzcd}
			&&&& {(X, p)} \\
			\\
			{(A, q)} &&& {(X_\inv, p_\inv)} \\
			\\
			&&&& {(X, p)}
			\arrow["r", shift left, from=1-5, to=3-4]
			\arrow["{r^\dag}", shift left, from=3-4, to=1-5]
			\arrow["r"', shift right, from=5-5, to=3-4]
			\arrow["{r^\dag}"', shift right, from=3-4, to=5-5]
			\arrow["m", shift left, from=1-5, to=5-5]
			\arrow["g", shift left, curve={height=12pt}, from=3-1, to=5-5]
			\arrow["{\tilde{g}}", dashed, from=3-1, to=3-4]
			\arrow["g"', shift right, curve={height=-12pt}, from=3-1, to=1-5]
			\arrow["{g^\dag}"', shift right, curve={height=12pt}, from=1-5, to=3-1]
			\arrow["{g^\dag}", shift left, curve={height=-12pt}, from=5-5, to=3-1]
			\arrow["{m^\dag}", shift left, from=5-5, to=1-5]
		\end{tikzcd}\]
		
		Taking the Bayesian inverse of both sides of $mg \aseq g$, we obtain $g^\dag m^\dag \aseq g^\dag$.
		By \Cref{colimit reversed}, $(X_\inv, p_\inv)$ is also a \emph{colimit} for the dynamics $D^\dag$, and so there exists (a unique) $\widetilde{(g^\dag)}$ such that $g^\dag \aseq \widetilde{(g^\dag)} r$.
		
		Taking Bayesian inverses again, we get $g \aseq r^\dag \widetilde{(g^\dag)}^\dag$, so taking $\tilde{g}\coloneqq \widetilde{(g^\dag)}^\dag$ proves the desired result.
		
		Now let $h : (A, q) \to (X_\inv, p_\inv)$ be another morphism such that $h \aseq r^\dag h$. But then, using the fact that $r r^\dag \aseq \id_{X_\inv}$, we can write:
		$$
		\tilde{g} \aseq r r^\dag \tilde{g} \aseq rg \aseq r r^\dag h \aseq h
		$$
		and so $h$ and $\tilde{g}$ are $q$-a.s. equal (and so equal in $\cat{PS(Stoch)}$).
	\end{proof}

	\subsection{Ergodic decompositions}\label{ergodic}
	
	We can interpret \Cref{xinvlimit} as an almost sure version of the ergodic decomposition theorem. 
	An ergodic decomposition theorem in terms of Markov categories was stated and proven in \cite{moss2022ergodic}, and the present result is an improvement in the following ways:
	\begin{itemize}
		\item It also holds in the case where the morphisms $m:X\to X$ are not deterministic;
		\item It also proves the \emph{uniqueness} of the ergodic decomposition.
	\end{itemize}
	The price to pay is that everything is done almost surely instead of strictly, but that is sufficient, for example, to interpret de Finetti's theorem, (see the next section\footnote{Also, \Cref{seanexample} can be used to show that the strictly invariant (not just almost surely) $\sigma$-algebra is not a colimit in the nondeterministic case, and so it does not seem to have a categorical description.}).
	Our proof seems also to be the only one in the literature, for the nondeterministic case, where one does not reduce the statement to the deterministic case via an embedding (and so it is less sensitive to the particular structure of the monoid -- in particular it does not need to be $\N$, or even free). 
	
	First of all, let us see what we mean by \emph{decomposition} in this categorical setting. (See also \cite[Section 3.1]{moss2022ergodic} and \cite[Example 1.2]{fritz2022dilations} for additional context, but keep in mind that the present setting is slightly different.)
	
	Consider a probability measure $p$ on a standard Borel space $X$.
	We can view it as a morphism $p:1\to X$ of $\cat{BorelStoch}$, or even as the unique morphism $p:(1,u)\to (X,p)$, where $u$ denotes the unique probability measure on the one-point space $1$. 
	
	Let now $A$ be another space, with a probability measure $b:1\to A$, and suppose we have a morphism $k:(A,b)\to (X,p)$ of $\cat{PS(Borel)}$. That means that for every measurable subset $B\subseteq X$,
	$$
	p(B) = \int_A k(B|a)\,b(da) .
	$$
	Equivalently, we can write the equation above as a measure-valued integral,
	\begin{equation}\label{convcomb}
		p = \int_A k_a\,b(da)
	\end{equation}
	where $k_a$ is the measure on $X$ given by $B\mapsto k(B|a)$. 
	We can interpret equation \eqref{convcomb} as the fact that \emph{we are expressing the measure $p$ as a convex mixture of the measures $k_a$, indexed by $a\in A$ and with weights given by the measure $b$}. 
	If $A$ is finite, in particular we have
	\begin{equation}\label{convcombfinite}
		p = \sum_{a\in A} k_a \,b(a) ,
	\end{equation}
	where the $b(a)$ are non-negative numbers summing to $1$, i.e.~we are writing $p$ as a finite convex combination of the measures $k_a$.
	Notice also that morphisms in $\cat{PS(C)}$ are defined only up to almost sure equality, and this is compatible with our interpretation as convex mixture, since the integral \eqref{convcomb}, as well as its finite analogue, do not change if we replace some of the $k_a$ for $a$ in a subset that is given measure (or weight) zero by $b$. 
	
	Suppose now that two terms in \eqref{convcombfinite} are equal, say, $k_a=k_{a'}$, for $a\neq a'$.
	Then we can \emph{reduce} the expression \eqref{convcombfinite} by combining $k_a$ and $k_{a'}$, replacing their terms in the sum by a single term with the sum of the weights, such as $k_a\,(b(a)+ b(a'))$.
	The result is the same, but the expression is now simpler.
	What happened is that the kernel $k:A\to X$ actually (a.s.)~\emph{factors through a deterministic one}, as follows,
	$$
	\begin{tikzcd}[row sep=small]
		(A,b) \ar{dr}[swap]{d} \ar{rr}{k} && (X,p) \\
		& (A',b') \ar{ur}[swap]{k'}
	\end{tikzcd}
	$$
	where $d$ is deterministic (and not injective: $d(a)=d(a')$).
	When this happens, we call the decomposition given by $k'$ a \emph{reduction} of the one given by $k$. 
	
	Let us now generalize this to kernels. We can view a kernel $f:Y\to X$ as an indexed family of probability measures, measurably parametrized by $Y$  (i.e.~through the mapping $y\mapsto f_y$), and we want to decompose it pointwise (almost surely).
	Given measures $p$ on $X$ and $q$ on $Y$, if $f$ gives a morphism $(Y,q)\to(X,p)$ we can view it as an indexed family of probability measures on $X$, but with a notion of ``where the bulk of them is'', or ``which ones are the typical, observable ones''. In particular, we can view $f$ as a parametrized family of measures ``which make up $p$'', or whose ``collective fuzzy image is $p$''.
	If now we consider a probability space $(A,b)$ and a measure-preserving kernel $k:(A,b)\to(X,p)$, we might want to ask if, just as we did above for $p$, we can express (almost) \emph{all} of the measures $f_y$ in terms of the measures $k_a$. That is, if we can express (almost) each $f_y$ as a mixture of the $k_a$. Explicitly, we would like, almost surely, to have
	$$
	f_y = \int_A k_a \,g_y(da)
	$$
	for some measure $g_y$ on $A$. If we want this to depend measurably on $y$, this amounts exactly to a Markov kernel $g:Y\to A$, and we are asking that $q$-almost surely, the following diagram commutes.
	$$
	\begin{tikzcd}
		(Y,q) \ar{rr}{f} \ar{dr}[swap]{g} && (X,p) \\
		& (A,b) \ar{ur}[swap]{k}
	\end{tikzcd}
	$$
	In other words, expressing a kernel (seen as a parametrized family of measures) pointwise as a convex combination of other kernels amounts to expressing it (almost surely) as the composition of two kernels. 
	The diagram above expresses $f$ as a convex combination in terms of $k$ (i.e.~of the $k_a$), with coefficients or weighting given by $g$.
	
	\begin{definition}
		A \emph{decomposition} of a morphism $f:(Y,q)\to(X,p)$ is a factorization of $f$ in $\cat{PS(C)}$. In other words, it consists of an object $(A,b)$ and morphisms (i.e.~equivalence classes) $g:(Y,q)\to(A,b)$ and $k:(A,b)\to (X,p)$ such that $f\aseq_q k\circ g$.
		
		Given decompositions $(g,k)$ and $(g',k')$ of $f$, respectively through $(A,b)$ and $(A',b')$, we say that $(g',k')$ is a \emph{reduction} of the decomposition $(g,k)$ if there exists an a.s.~deterministic morphism $d:(A,b)\to(A',b')$ such that $k'\circ d\aseq k$ and $d\circ g\aseq g'$, as in the following commutative diagram:
		\[\begin{tikzcd}[sep=tiny]
			&&& {(A', b')} \\
			{(Y, q)} &&&&& {(X, p)} \\
			&& {(A, b)}
			\arrow["{g'}", from=2-1, to=1-4]
			\arrow["{k'}", from=1-4, to=2-6]
			\arrow["g"', from=2-1, to=3-3]
			\arrow["k"', from=3-3, to=2-6]
			\arrow["d", dotted, from=3-3, to=1-4]
		\end{tikzcd}\]
	\end{definition}

	Let us now define what we mean by ergodicity. (See \cite[Section 3.3]{moss2022ergodic} for additional context, once again keeping in mind that the present setting is slightly different.)
	Recall that classically, an ergodic measure is an invariant measure which evaluates to $0$ or $1$ on all (almost surely) invariant sets.
	
	\begin{definition}\label{deferg}
		Let $\cat{C}$ be a causal Markov category.
		Let $D$ be a dynamical system $M\to\cat{PS(C)}$ acting on the object $(X,p)$, with invariant object $r:(X,p)\to(X_\inv,p_\inv)$.
		The state $p$ is called \emph{ergodic} for $D$ if any of the following equivalent conditions hold:
		\begin{itemize}
			\item The state $p_\inv=r p$ on the resulting invariant object (with colimiting cocone $r$) is deterministic;
			\item For every $p$-a.s.~deterministic, $p$-a.s.~invariant $f:(X,p)\to (Y,q)$, the state $q$ is deterministic.  
		\end{itemize}
	\end{definition}

	(To see that the two conditions are equivalent, take $f=r$.)
	For $\cat{C}=\cat{Stoch}$, this corresponds to the usual definition in terms of a.s.\ invariant sets.
	
	Let us now define what we mean by ergodicity in the context of a convex decomposition: roughly, it's a parametrized version of ergodic states, or a mixture such as \eqref{convcomb} where almost all the values of the integrand are ergodic measures in a certain sense.
	First of all, let's define precisely in which sense we mean that all the values in the integrand are almost surely ergodic.
	\begin{definition}
		Let $D$ be a dynamical system $M\to\cat{PS(C)}$ acting on the object $(X,p)$, with invariant object $r:(X,p)\to(X_\inv,p_\inv)$.
		Let $k:(A,q)\to(X,p)$ be a morphism. We say that $k$ is \emph{$q$-almost surely ergodic} if and only if 
		\begin{itemize}
			\item $rk:A\to X_\inv$ is $q$-almost surely deterministic;
			\item $k$ is $q$-almost surely (left-)invariant.
		\end{itemize}
	\end{definition}
	
	(Note that for the case of $\cat{C}=\cat{BorelStoch}$, we are taking $k$ to be almost surely invariant for every $m$ separately. Under some conditions, such as if $M$ is countable, this is equivalent to say that $k$ is almost surely invariant jointly for all $m$.)
	
	Let us now put the two ideas together, defining what we mean by ``almost sure ergodic decomposition''.
	
	\begin{definition}
		Let $D$ be a dynamical system $M\to\cat{PS(C)}$ acting on the object $(X,p)$, with invariant object $r:(X,p)\to(X_\inv,p_\inv)$.
		Let $g:(Y, q)\to (X, p)$. An \emph{almost sure ergodic decomposition} of $g$ is a decomposition $g\aseq k\circ h$, where $k$ is almost surely ergodic. 
	\end{definition}
	
	We have that
	\begin{itemize}
		\item If $g:(Y, q)\to (X, p)$ admits an a.s.\ ergodic decomposition, then it is necessarily a.s.\ left-invariant, since $k$ is (and left-invariant morphisms are closed under precomposition);
		\item If $k:(A,q)\to(X,p)$ is almost surely ergodic, and $d:(B,s)\to(A,q)$ is almost surely deterministic, then the precomposition $k\circ d$ is almost surely ergodic too.
		\item Therefore, if $(h',k')$ is an a.s.\ ergodic decomposition of $g$ that reduces a decomposition $(h,k)$, then $(h,k)$ is an a.s.\ ergodic decomposition too (since $k=k'\circ d$ is a.s.~ergodic too).
	\end{itemize}

	We are now ready for the main statement.
	
	\begin{theorem}\label{erg_dec}
		Let $\cat{C}$ be a Markov category with conditionals, and suppose that all invariant objects of $\cat{PS(C)}$ exist. (For example, take $\cat{C}=\cat{BorelStoch}$.)
		Let $D$ be a dynamical system $M\to\cat{PS(C)}$ acting on the object $(X,p)$, with invariant object $r:(X,p)\to(X_\inv,p_\inv)$.
		Then 
		\begin{itemize}
			\item The morphism $r^\dagger:(X_\inv,p_\inv)\to(X,p)$ is almost surely ergodic;
			\item Every left-invariant morphism $g:(Y,q)\to (X,p)$ can be a.s.\ ergodically decomposed as $r^\dag\circ \tilde{g}$ for a unique $\tilde{g}:(Y,q)\to(X_\inv,p_\inv)$ (up to $q$-a.s.~equality);
			\item For every left-invariant morphism $g:(Y, q)\to (X, p)$, any (other) a.s.\ ergodic decomposition of $g$ can be uniquely reduced to the decomposition $(\tilde{g}, r^\dag)$.
		\end{itemize}
	\end{theorem}
	
	Before the proof, let us interpret this for $\cat{C}=\cat{BorelStoch}$, taking a standard Borel representative for $X_\inv$. 
	We get a (generalized, but almost sure) version of the traditional ergodic decomposition theorem:
	
	\begin{corollary}
		Let $D$ be a dynamical system acting via Markov kernels on a standard Borel space $X$, preserving the measure $p$. 
		Then 
		\begin{itemize}
			\item The kernel $r^\dag:X_\inv\to X$, defining for each $x\in X$ (technically, $X_\inv$) the measure $r^\dag_x$, satisfies the following properties:
			\begin{itemize}
				\item For almost all $x$, and for all invariant sets $B$, $r^\dag_x(B)$ is either $0$ or $1$;
				\item For all $m\in M$, for almost all $x\in X$, the measure $r^\dag_x$ on $X$ is $m$-invariant.
			\end{itemize}
			\item Given a left-invariant kernel $g:(Y,q)\to(X,p)$ (which we can view as a parametrized family of invariant measures, in the sense that for every $m\in M$, for $q$-almost all $y\in Y$, the measure $g_y$ on $X$ is $m$-invariant),
			for $q$-almost all $y\in Y$ the measure $g_y$ can be expressed as a following mixture of the $r^\dag_x$: 
			$$
			g_y \aseq_q \int_{X_\inv} r^\dag_x \,\tilde{g}(dx|y)
			$$
			for an a.s.~unique kernel $\tilde{g}:(Y,q)\to(X_\inv,p)$, i.e.~in a unique way as a convex decomposition.
			\item For every (other) a.s.\ ergodic decomposition of $g$ given by $h:(Y, q) \to (A, b)$ and $k:(A, b) \to (X, p)$, with $k$ a.s.~ergodic, there is a measurable function $f:A\to X_\inv$ such that for $b$-almost all $a\in A$, $k_a=r^\dag_{f(a)}$.
		\end{itemize}
	\end{corollary}
	
	The proof of the theorem follows almost as a corollary from the results of the previous section. 
	
	\begin{proof}[Proof of \Cref{erg_dec}]
		First of all, let us prove that $r^\dag$ is $p_\inv$-almost surely ergodic. 
		\begin{itemize}
			\item $rr^\dag$ is almost surely deterministic, since it is almost surely equal to the identity;
			\item $r^\dag$ is almost surely invariant by \Cref{rdaginv}.
		\end{itemize}
		
		Now let $g:(A,q)\to (X,p)$ be a left-invariant morphism.
		By \Cref{xinvlimit}, and by the universal property of limits there exists a unique $\tilde{g}:(Y,q)\to(X_\inv,p)$ such that $g\aseq r^\dag\circ\tilde{g}$. 
		
		Now suppose we have another a.s.\ decomposition of $g$ given by $h:(Y, q) \to (A, b)$, $k:(A, b) \to (X, p)$ of $g$, with $k$ a.s.~ergodic. Then, because $k$ is left-invariant, it factors through $(X_\inv, p_\inv)$ and some $\tilde{k}$ such that $k \aseq r^\dag \tilde{k}$, as in the following diagram.
		\[\begin{tikzcd}[sep=tiny]
			&&& {(X_\inv, p_\inv)} \\
			{(Y, q)} &&&&& {(X, p)} \\
			&& {(A, b)}
			\arrow["{\tilde{g}}", from=2-1, to=1-4]
			\arrow["{r^\dag}", from=1-4, to=2-6]
			\arrow["h"', from=2-1, to=3-3]
			\arrow["k"', from=3-3, to=2-6]
			\arrow["{\tilde{k}}"', from=3-3, to=1-4, near end]
		\end{tikzcd}\]
		Since $k$ is ergodic, we have that $rk$ is $q$-a.s.~deterministic. But $rk \aseq rr^\dag \tilde{k} \aseq \tilde{k}$, and so $\tilde{k}$ is almost surely deterministic. Also, since $g \aseq k \circ h \aseq r^\dag \circ (\tilde{k} \circ h)$ and by uniqueness of $\tilde{g}$, we also have that $\tilde{g} \aseq \tilde{k} \circ h$. Therefore the diagram above commutes, and since $\tilde{k}$ is $b$-a.s.~deterministic, the decomposition $(\tilde{g}, r^\dag)$ is a reduction of $(h, k)$.
	\end{proof}
	
	In some sense, $(X_\inv,p_\inv)$ plays (up to isomorphism of $\cat{PS(C)}$) the role of \emph{indexing the ergodic measures}.
	It can be thought of, especially when $M$ is countable, as the space of ergodic measures themselves.
	One has to keep in mind, however, that everything is defined only up to measure zero, and so $(X_\inv,p_\inv)$ does not keep track of each single ergodic measure in general. One can think of $(X_\inv,p_\inv)$ as of keeping track of the ``bulk'' of those ergodic measures that, together, form the invariant measure $p$.

	\subsection{Transitions to equilibrium}\label{equilibrium}
	
	We will now look in more detail at the map $e_D$, and show its significance in terms of equilibrium. 
	(As we remarked, the first mention of using idempotents to describe notions of stochastic equilibrium categorically can be found in \cite{fritz2021topos}.)
	Very roughly, we can interpret $e_D$ as ``averaging things over orbits'', in a way that we will make more precise in a moment. 
	
	We will work in a Markov category with conditionals $\cat{C}$ such that $\cat{PS(C)}$ has all invariant objects (and so, in particular, all idempotents of $\cat{PS(C)}$ split).
	Thanks to \Cref{xinvborel}, an example of such a $\cat{C}$ is the category $\cat{BorelStoch}$.
	
	Throughout this section, let $D:M\to\cat{PS(C)}$ be a dynamical system acting on the object $(X,p)$.
	We will exhibit how the idempotent $e_D:(X,p)\to(X,p)$ associated to $D$ has the interpretation of (almost surely) representing the ``long-time behavior'' of the system, its transition to equilibrium. 
	
	\begin{example}\label{discchain}
		Consider a discrete Markov chain given by the following stochastic matrix on $X=\{a,b,c,d,e\}$.
		\begin{equation*}
			m=\left( \begin{matrix}
				1/2 & 0 & 0 & 0 & 0 \\
				1/2 & 1 & 0 & 0 & 0 \\
				0 & 0 & 1 & 0 & 0 \\
				0 & 0 & 0 & 1/2 & 1/3 \\
				0 & 0 & 0 & 1/2 & 2/3
			\end{matrix} \right)
			\qquad\qquad
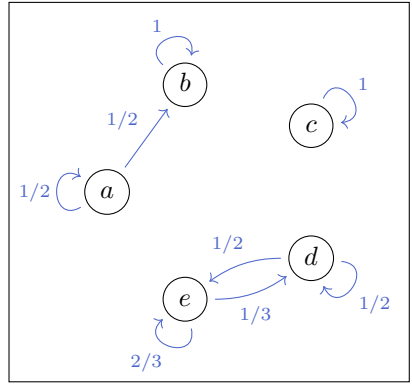
			\begin{tikzpicture}[baseline={(current bounding box.center)},
				transition/.style={midway, inner sep=1mm, font=\scriptsize, color=probcolor}]
				
				\foreach \i/\j in {1/a,2/b,3/c,4/d,5/e}
				\node[circle, draw=black, inner sep=0.5mm] (\i) at (252-72*\i:1.5) {$\j\strut$};
				
				\draw[ar,probcolor] (1) to node[transition, above left] {1/2} (2);
				\draw[ar,probcolor] (1) to [out=210,in=150, looseness=5] node[transition, left] {1/2} (1);
				\draw[ar,probcolor] (2) to [out=138,in=78, looseness=5] node[transition, above left] {1} (2);
				\draw[ar,probcolor] (3) to [out=66,in=6, looseness=5] node[transition, above right] {1} (3);
				\draw[ar,probcolor] (4) to [out=-6,in=-66, looseness=5] node[transition, below right] {1/2} (4);
				\draw[ar,probcolor] (4) to [out=-180,in=36] node[transition, above left, pos=0.4] {1/2} (5);
				\draw[ar,probcolor] (5) to [out=0,in=-144] node[transition, below right, near start] {1/3} (4);
				\draw[ar,probcolor] (5) to [out=-78,in=-138, looseness=5] node[transition, below left] {2/3} (5);
				
				\node[inner xsep=1cm, inner ysep=0.8cm, draw=black,fit=(1) (2) (3) (4) (5)] {} ;
			\end{tikzpicture}
		\end{equation*}
		Consider also the following invariant measure $p$:
		$$
		p(a)=0 , \quad p(b)=0 , \quad p(c) = \dfrac{1}{3} , \quad p(d) = \dfrac{2}{3}\cdot \dfrac{2}{5}= \dfrac{4}{15} , \quad p(e) = \dfrac{2}{3}\cdot \dfrac{3}{5}=  \dfrac{6}{15} .
		$$
		The transition matrix $m$ has three recurrent classes, $\{b\}$, $\{c\}$, and $\{d,e\}$, corresponding to three ergodic measures supported on them. 
		The $p$-a.s.~invariant $\sigma$-algebra, up to isomorphism of $\cat{PS(FinStoch)}$, is the one generated by the partition $\{\{b\},\{c\},\{d,e\}\}$ (since $a$ is transient), and so we can take as $X_\inv$, equivalently, a three-element set $\{x,y,z\}$ with the discrete $\sigma$-algebra, and measure
		\begin{itemize}
			\item $p(x)=0$ (corresponding to the class $\{b\}$);
			\item $p(y)=1/3$ (corresponding to the class $\{c\}$);
			\item $p(z)=2/3$ (corresponding to the class $\{d,e\}$).
		\end{itemize}
		Note that, up to isomorphism of $\cat{PS(FinStoch)}$, we could have even left $x$ out entirely, since it has measure zero.  
		(In this case it is easy to keep track of the ``invisible'' recurrent class $\{b\}$ even if it has measure zero, but this does not generalize outside the discrete case, since in general there may not be an invariant measure which gives nonzero measure to \emph{all} ergodic ones.)
		
		The map $r:X\to X_\inv$, since it is almost surely deterministic, can be thought of as a function. It maps each element of positive measure to the element of $X_\inv$ corresponding to its recursive class. For transient elements, and for other elements of measure zero, $r$ can be arbitrary:
		
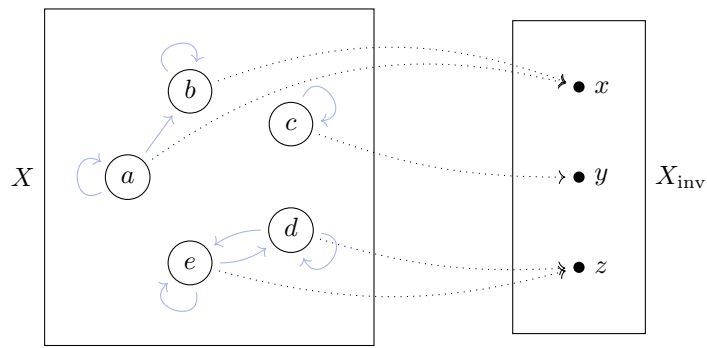
\begin{center}
			\begin{tikzpicture}[baseline={(current bounding box.center)},scale=0.8,
				transition/.style={midway, inner sep=1mm, font=\scriptsize, color=probcolor}]
				
				\foreach \i/\j in {1/a,2/b,3/c,4/d,5/e}
				\node[circle, draw=black, inner sep=0.5mm] (\i) at (252-72*\i:1.5) {$\j\strut$};
				
				\draw[ar,probcolor!50] (1) to (2);
				\draw[ar,probcolor!50] (1) to [out=210,in=150, looseness=5] (1);
				\draw[ar,probcolor!50] (2) to [out=138,in=78, looseness=5] (2);
				\draw[ar,probcolor!50] (3) to [out=66,in=6, looseness=5] (3);
				\draw[ar,probcolor!50] (4) to [out=-6,in=-66, looseness=5] (4);
				\draw[ar,probcolor!50] (4) to [out=-180,in=36] (5);
				\draw[ar,probcolor!50] (5) to [out=0,in=-144] (4);
				\draw[ar,probcolor!50] (5) to [out=-78,in=-138, looseness=5] (5);
				
				\node[inner xsep=0.8cm, inner ysep=0.8cm, draw=black,fit=(1) (2) (3) (4) (5), label=left:$X$] {} ;
				
				\node[bullet, label=right:$x$] (x) at (6,1.5) {} ;
				\node[bullet, label=right:$y$] (y) at (6,0) {} ;
				\node[bullet, label=right:$z$] (z) at (6,-1.5) {} ;
				\node[inner xsep=0.8cm, inner ysep=0.8cm, draw=black,fit=(x) (y) (z), label=right:$X_\inv$] {} ;
				
				\draw[ar,dotted] (1) to[bend left=25] (x) ;
				\draw[ar,dotted] (2) to[bend left=20] (x) ;
				\draw[ar,dotted] (3) to[bend right=10] (y) ;
				\draw[ar,dotted] (4) to[bend right=10] (z) ;
				\draw[ar,dotted] (5) to[bend right=15] (z) ;
			\end{tikzpicture}
		\end{center}
		The Bayesian inverse $r^\dag:X_\inv\to X$, instead, is in general not deterministic. It assigns each element of $X_\inv$ (i.e., up to isomorphism, each recursive class) to the corresponding elements in the class, distributed according to the ergodic measure associated to the class. (Again, technically it has to do so only for those classes that are assigned nonzero measure by $p$, and outside of that, it can be arbitrary.) Here is a representation:
		
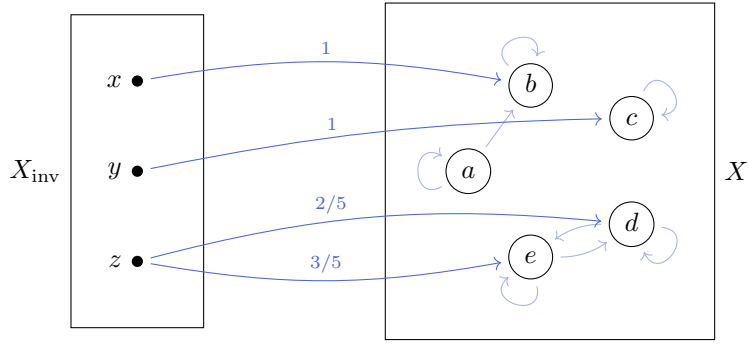
\begin{center}
			\begin{tikzpicture}[baseline={(current bounding box.center)},scale=0.8,
				transition/.style={midway, inner sep=1mm, font=\scriptsize, color=probcolor}]
				
				\foreach \i/\j in {1/a,2/b,3/c,4/d,5/e}
				\node[circle, draw=black, inner sep=0.5mm] (\i) at (252-72*\i:1.5) {$\j\strut$};
				
				\draw[ar,probcolor!50] (1) to (2);
				\draw[ar,probcolor!50] (1) to [out=210,in=150, looseness=5] (1);
				\draw[ar,probcolor!50] (2) to [out=138,in=78, looseness=5] (2);
				\draw[ar,probcolor!50] (3) to [out=66,in=6, looseness=5] (3);
				\draw[ar,probcolor!50] (4) to [out=-6,in=-66, looseness=5] (4);
				\draw[ar,probcolor!50] (4) to [out=-180,in=36] (5);
				\draw[ar,probcolor!50] (5) to [out=0,in=-144] (4);
				\draw[ar,probcolor!50] (5) to [out=-78,in=-138, looseness=5] (5);
				
				\node[inner xsep=0.8cm, inner ysep=0.8cm, draw=black,fit=(1) (2) (3) (4) (5), label=right:$X$] {} ;
				
				\node[bullet, label=left:$x$] (x) at (-7,1.5) {} ;
				\node[bullet, label=left:$y$] (y) at (-7,0) {} ;
				\node[bullet, label=left:$z$] (z) at (-7,-1.5) {} ;
				\node[inner xsep=0.8cm, inner ysep=0.8cm, draw=black,fit=(x) (y) (z), label=left:$X_\inv$] {} ;
				
				\draw[ar,probcolor] (x) to[bend left=10] node[transition, above] {1} (2) ;
				\draw[ar,probcolor] (y) to[bend left=5] node[transition, above,pos=0.4] {1} (3) ;
				\draw[ar,probcolor] (z) to[bend left=10] node[transition, above, pos=0.4] {2/5} (4) ;
				\draw[ar,probcolor] (z) to[bend right=10] node[transition, above] {3/5} (5) ;
			\end{tikzpicture}
		\end{center}
		
		The idempotent $e_D=r^\dag r:X\to X$ is given (again, up to a set of measure zero) by the following idempotent Markov chain:
		\begin{equation*}
			e_D=\left( \begin{matrix}
				0 & 0 & 0 & 0 & 0 \\
				1 & 1 & 0 & 0 & 0 \\
				0 & 0 & 1 & 0 & 0 \\
				0 & 0 & 0 & 2/5 & 3/5 \\
				0 & 0 & 0 & 2/5 & 3/5
			\end{matrix} \right)
			\qquad\qquad
			\begin{tikzpicture}[baseline={(current bounding box.center)},
				transition/.style={midway, inner sep=1mm, font=\scriptsize, color=probcolor}]
				
				\foreach \i/\j in {1/a,2/b,3/c,4/d,5/e}
				\node[circle, draw=black, inner sep=0.5mm] (\i) at (252-72*\i:1.5) {$\j\strut$};
				
				\draw[ar,probcolor] (1) to node[transition, above left] {1} (2);
				\draw[ar,probcolor] (2) to [out=138,in=78, looseness=5] node[transition, above left] {1} (2);
				\draw[ar,probcolor] (3) to [out=66,in=6, looseness=5] node[transition, above right] {1} (3);
				\draw[ar,probcolor] (4) to [out=-6,in=-66, looseness=5] node[transition, below right] {2/5} (4);
				\draw[ar,probcolor] (4) to [out=-180,in=36] node[transition, above left, pos=0.4] {3/5} (5);
				\draw[ar,probcolor] (5) to [out=0,in=-144] node[transition, below right, near start] {2/5} (4);
				\draw[ar,probcolor] (5) to [out=-78,in=-138, looseness=5] node[transition, below left] {3/5} (5);
				
				\node[inner xsep=1cm, inner ysep=0.8cm, draw=black,fit=(1) (2) (3) (4) (5)] {} ;
			\end{tikzpicture}
		\end{equation*}
		This idempotent matrix can be seen as the ``long-time dynamics'', and satisfies the property that $e_Dm=me_D=e_D$, which we can prove in general in the next proposition.
	\end{example}
	
	\begin{proposition}\label{edinv}
		For every $m\in M$, denote as usual the induced morphism $(X,p)\to (X,p)$ again by $m$.
		Then we have that 
		\begin{equation*}
			e_D m \aseq_p m e_D \aseq_p e_D .
		\end{equation*}
		In other words, $e_D$ is both left- and right-invariant.
	\end{proposition}
	\begin{proof}
		Recalling that $e_D = r^\dag r$, and that $r$ is right-invariant, we have
		\begin{equation*}
			e_D m = r^\dag r m \aseq r^\dag r = e_D .
		\end{equation*}
		Similarly, using \Cref{rdaginv},
		\begin{equation*}
			m e_D = m r^\dag r \aseq r^\dag r = e_D .
		\end{equation*}
	\end{proof}

	\begin{proposition}\label{invfrome}
		Let $f:(X,p)\to(Y,q)$ be any morphism. The morphism $f$ is right-invariant (for every $m$) if and only if it is right-invariant for $e_D$, i.e.~if $f\circ e_D\aseq_p f$.
		
		Similarly, let $g:(A,q)\to (X,p)$ be any morphism. The morphism $g$ is left-invariant (for every $m$) if and only if it is left-invariant for $e_D$, i.e.~if $e_D\circ g\aseq_q g$.
	\end{proposition}
	\begin{proof}
		Suppose $f$ is right-invariant for every $m$. Then by the universal property of $(X_\inv,p_\inv)$ as a colimit of $D$,
		$$
		fe_D \aseq \tilde{f} r e_D = \tilde{f} r r^\dag r \aseq \tilde{f} r = f .
		$$
		Conversely, suppose that $f\circ e_D\aseq f$. Then by \Cref{edinv}, for every $m\in M$, 
		$$
		f m \aseq f e_D m \aseq f e_D \aseq f .
		$$
		The proof for left-invariance is analogous.
	\end{proof}
	
	\begin{corollary}
		Let $f:(X,p)\to(Y,q)$ be any morphism. The morphism $f\circ e_D:(X,p)\to(Y,q)$ is right-invariant. Moreover, every right-invariant morphism arises in this way.
		
		Similarly, given any $g:(A,q)\to (X,p)$, $e_D\circ g:(A,q)\to (X,p)$ is left-invariant. Moreover, every left-invariant morphism arises in this way.
	\end{corollary}
	
	A possible interpretation of \Cref{invfrome} is the following:
	\begin{itemize}
		\item Given any morphism $f:X\to Y$, deterministic or not, precomposing with $e_D$ ``mixes'' or ``averages'' $f$ stochastically to make it (probabilistically) invariant. (Componentwise, this means harmonic.)
		\item Similarly, given any morphism $g:A\to X$, which we can see as giving families of elements of $X$ or measures on $X$ ``making up'' $p$ (in the sense of convex decompositions), postcomposing with $e_D$ ``mixes'' or ``averages'' the measures until they are invariant.
	\end{itemize}
	
	In other words, composing with $e_D$ on either side has a meaning of ``mixing until equilibrium''. This has to do with \emph{conditional expectations}, which can also be described categorically, and will be addressed in full generality in future work. 
	This is also related to the \emph{ergodic theorem}, which is about \emph{tending} to equilibrium, and which we also relegate to future work.
	
	Again following the idea of ``long-time behavior'', we also have the following equivalent characterization of ergodicity.  It intuitively says that a measure is ergodic if and only if each event, after a transition to equilibrium, is independent from any other event. 
	
	\begin{proposition}
		An invariant state $p$ for the system $D$ is ergodic (according to \Cref{deferg}) if and only if the following equation holds.
		\begin{equation}\label{alt-erg}
			\input{alt-erg.tikz}
		\end{equation}
	\end{proposition}
	In $\cat{BorelStoch}$ this says that for every measurable $A$ and $B$, 
	$$
	\int_A e_D(B|x)\,p(dx) = p(A)\,p(B) .
	$$
	(Compare for example with \cite[Theorem~3, Item 6]{tao}.)
	\begin{proof}
		First of all, suppose that $p$ is ergodic, meaning that $r\circ p$ is deterministic. Then by positivity, 
		\ctikzfig{alt-erg1}
		
		Conversely, suppose \eqref{alt-erg} holds. Then by determinism of $r$,
		\ctikzfig{alt-erg2}
		which shows that $r\circ p$ is deterministic.
	\end{proof}
	
	Before we leave this section, let us also notice that the morphism $e_D$ is self-adjoint, in the sense that it is its own Bayesian inverse. 
	Indeed,
	$$
	e_D^\dag = (r^\dag r)^\dag = r^\dag (r^\dag)^\dag = r^\dag r = e_D .
	$$
	This can be interpreted as the fact that, as a Markov chain, it is \emph{reversible}, or that it satisfied \emph{detailed balance}. In the discrete case, this condition reads as follows,
	$$
	p(x)\,e_D(y|x) = p(y) \, e_D(y|x)
	$$
	and is an even stronger notion of equilibrium than left- and right-invariance. In the terminology of \cite[Proposition~4.1.10]{supports}, this is related to the idea of \emph{balanced} idempotents (see the reference for the details). 
	From the point of view of dagger categories, we are saying that $e_D$ (hence, all idempotents of $\cat{PS(Borel)}$) are \emph{dagger idempotents} (which follows from the fact that they are dagger-split), see \cite{daggeridempotents}.

	\section{Further examples}\label{examples}
	
	Here we give concrete instances of the results of the previous section in the case of $\cat{C}=\cat{BorelStoch}$, for different choices of spaces and dynamical systems, connecting to classical ideas of probability theory.
	
	\subsection{Finite permutations}\label{finiteperm}
	
	Let $X$ be a finite set. For a finite $n\in \N$, denote by $X^n$ the $n$-fold tensor product of $X$ with itself,
	$$
	X^n \coloneqq \underbrace{X\otimes \dots \otimes X}_{n \mbox{\scriptsize\ times}} .
	$$
	The symmetric group $S_n$ acts on $X^n$ by permuting the components (i.e.~by application of the braiding of the monoidal category). Concretely, given $\sigma\in S_n$, 
	$$
	(x_1,\dots,x_n) \mapsto (x_{\sigma^{-1}(1)},\dots,x_{\sigma^{-1}(n)}) .
	$$
	
	The invariant $\sigma$-algebra is exactly the one generated by the orbits of this action. 
	For example, if $X=\{0,1\}$, and $n=3$, the set $\{0,1\}^3$ can be partitioned as follows
	
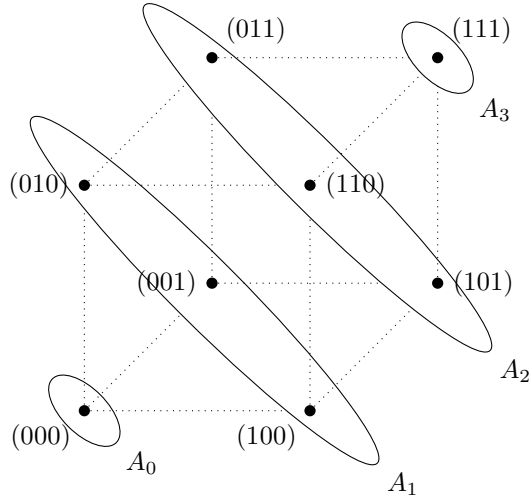
\begin{center}
		\begin{tikzpicture}[baseline,
			partition/.style={ellipse,fill=white, fill opacity=0.7, draw=black, minimum width=0.6cm, minimum height=1.2cm, rotate fit=45},
			x={(0:3cm)},y={(90:3cm)},z={(45:2.4cm)}]
			\node[bullet] (000) at (0,0,0) {};
			\node[bullet] (001) at (0,0,1) {};
			\node[bullet] (010) at (0,1,0) {}; 
			\node[bullet] (011) at (0,1,1) {};
			\node[bullet] (100) at (1,0,0) {};
			\node[bullet] (101) at (1,0,1) {}; 
			\node[bullet] (110) at (1,1,0) {}; 
			\node[bullet] (111) at (1,1,1) {}; 
			
			\node[partition,fit=(000),label=below:$A_0$] {} ;
			\draw[side] (000) -- (001) ;
			\draw[side] (000) -- (010) ;
			\draw[side] (000) -- (100) ;
			
			\node[partition,fit=(010) (001) (100),label=below:$A_1$] {} ;
			\draw[side] (010) -- (011) ;
			\draw[side] (100) -- (101) ;
			\draw[side] (001) -- (011) ;
			\draw[side] (100) -- (110) ;
			\draw[side] (001) -- (101) ;
			\draw[side] (010) -- (110) ;
			
			\node[partition,fit=(011) (110) (101),label=below:$A_2$] {} ;
			\draw[side] (110) -- (111) ;
			\draw[side] (101) -- (111) ;
			\draw[side] (011) -- (111) ;
			
			\node[partition,fit=(111),label=below:$A_3$] {} ;
			
			\node[bullet,label=225:$(000)$] () at (0,0,0) {};
			\node[bullet,label=180:$(001)$] () at (0,0,1) {};
			\node[bullet,label=180:$(010)$] () at (0,1,0) {};
			\node[bullet,label=45:$(011)$] () at (0,1,1) {};
			\node[bullet,label=225:$(100)$] () at (1,0,0) {};
			\node[bullet,label=0:$(101)$] () at (1,0,1) {};
			\node[bullet,label=0:$(110)$] () at (1,1,0) {};
			\node[bullet,label=45:$(111)$] () at (1,1,1) {};
		\end{tikzpicture}
	\end{center}
	where 
	\begin{itemize}
		\item $A_0$ is the set of sequences containing all zeros;
		\item $A_1$ is the set of sequences containing a single $1$;
		\item $A_2$ is the set of sequences containing exactly two occurrences of $1$;
		\item $A_3$ is the set of sequences containing all ones.
	\end{itemize}
	
	A measure $p$ on $X^n$ is invariant if and only if each value $p(x_1,\dots,x_n)$ is constant within each orbit, i.e.~if it does not depend on the order of the $x_i$. Equivalently, a measure is invariant if and only if it only depends on the number of occurrences of the different $x_i$, regardless of when they occur.
	
	There are four ergodic measures, which here are supported on single orbits and constant within each orbit, and so they are in bijection with the orbits:
	$$
	
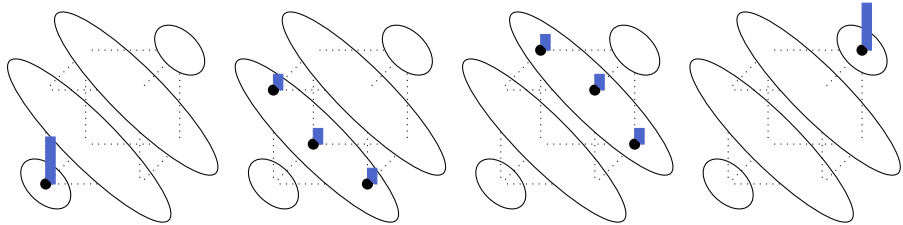
\begin{tikzpicture}[baseline,
		partition/.style={ellipse,fill=white, fill opacity=0.7, draw=black, minimum width=0.4cm, minimum height=0.8cm, rotate fit=45},
		x={(0:1.25cm)},y={(90:1.25cm)},z={(45:0.75cm)}]
		\node[bullet] (000) at (0,0,0) {};
		\node[bullet] (001) at (0,0,1) {};
		\node[bullet] (010) at (0,1,0) {}; 
		\node[bullet] (011) at (0,1,1) {};
		\node[bullet] (100) at (1,0,0) {};
		\node[bullet] (101) at (1,0,1) {}; 
		\node[bullet] (110) at (1,1,0) {}; 
		\node[bullet] (111) at (1,1,1) {}; 
		
		\node[partition,fit=(000)] {} ;
		\draw[side] (000) -- (001) ;
		\draw[side] (000) -- (010) ;
		\draw[side] (000) -- (100) ;
		
		\node[partition,fit=(010) (001) (100)] {} ;
		\draw[side] (010) -- (011) ;
		\draw[side] (100) -- (101) ;
		\draw[side] (001) -- (011) ;
		\draw[side] (100) -- (110) ;
		\draw[side] (001) -- (101) ;
		\draw[side] (010) -- (110) ;
		
		\node[partition,fit=(011) (110) (101)] {} ;
		\draw[side] (110) -- (111) ;
		\draw[side] (101) -- (111) ;
		\draw[side] (011) -- (111) ;
		
		\node[partition,fit=(111)] {} ;
		
		\node[bullet] (000) at (0,0,0) {};
		
		\draw[probbar] (000) -- ++(0.1,0,0) -- ++(0,0.5,0) -- ++(-0.1,0,0) -- (000) ; 
	\end{tikzpicture}
	\begin{tikzpicture}[baseline,
		partition/.style={ellipse,fill=white, fill opacity=0.7, draw=black, minimum width=0.4cm, minimum height=0.8cm, rotate fit=45},
		x={(0:1.25cm)},y={(90:1.25cm)},z={(45:0.75cm)}]
		\node[bullet] (000) at (0,0,0) {};
		\node[bullet] (001) at (0,0,1) {};
		\node[bullet] (010) at (0,1,0) {}; 
		\node[bullet] (011) at (0,1,1) {};
		\node[bullet] (100) at (1,0,0) {};
		\node[bullet] (101) at (1,0,1) {}; 
		\node[bullet] (110) at (1,1,0) {}; 
		\node[bullet] (111) at (1,1,1) {}; 
		
		\node[partition,fit=(000)] {} ;
		\draw[side] (000) -- (001) ;
		\draw[side] (000) -- (010) ;
		\draw[side] (000) -- (100) ;
		
		\node[partition,fit=(010) (001) (100)] {} ;
		\draw[side] (010) -- (011) ;
		\draw[side] (100) -- (101) ;
		\draw[side] (001) -- (011) ;
		\draw[side] (100) -- (110) ;
		\draw[side] (001) -- (101) ;
		\draw[side] (010) -- (110) ;
		
		\node[partition,fit=(011) (110) (101)] {} ;
		\draw[side] (110) -- (111) ;
		\draw[side] (101) -- (111) ;
		\draw[side] (011) -- (111) ;
		
		\node[partition,fit=(111)] {} ;
		
		\node[bullet] (001) at (0,0,1) {};
		\node[bullet] (010) at (0,1,0) {}; 
		\node[bullet] (100) at (1,0,0) {};
		
		\draw[probbar] (001) -- ++(0.1,0,0) -- ++(0,0.5/3,0) -- ++(-0.1,0,0) -- (001) ; 
		\draw[probbar] (010) -- ++(0.1,0,0) -- ++(0,0.5/3,0) -- ++(-0.1,0,0) -- (010) ; 
		\draw[probbar] (100) -- ++(0.1,0,0) -- ++(0,0.5/3,0) -- ++(-0.1,0,0) -- (100) ; 
	\end{tikzpicture}
	\begin{tikzpicture}[baseline,
		partition/.style={ellipse,fill=white, fill opacity=0.7, draw=black, minimum width=0.4cm, minimum height=0.8cm, rotate fit=45},
		x={(0:1.25cm)},y={(90:1.25cm)},z={(45:0.75cm)}]
		\node[bullet] (000) at (0,0,0) {};
		\node[bullet] (001) at (0,0,1) {};
		\node[bullet] (010) at (0,1,0) {}; 
		\node[bullet] (011) at (0,1,1) {};
		\node[bullet] (100) at (1,0,0) {};
		\node[bullet] (101) at (1,0,1) {}; 
		\node[bullet] (110) at (1,1,0) {}; 
		\node[bullet] (111) at (1,1,1) {}; 
		
		\node[partition,fit=(000)] {} ;
		\draw[side] (000) -- (001) ;
		\draw[side] (000) -- (010) ;
		\draw[side] (000) -- (100) ;
		
		\node[partition,fit=(010) (001) (100)] {} ;
		\draw[side] (010) -- (011) ;
		\draw[side] (100) -- (101) ;
		\draw[side] (001) -- (011) ;
		\draw[side] (100) -- (110) ;
		\draw[side] (001) -- (101) ;
		\draw[side] (010) -- (110) ;
		
		\node[partition,fit=(011) (110) (101)] {} ;
		\draw[side] (110) -- (111) ;
		\draw[side] (101) -- (111) ;
		\draw[side] (011) -- (111) ;
		
		\node[partition,fit=(111)] {} ;
		
		\node[bullet] (011) at (0,1,1) {};
		\node[bullet] (101) at (1,0,1) {}; 
		\node[bullet] (110) at (1,1,0) {}; 
		
		\draw[probbar] (011) -- ++(0.1,0,0) -- ++(0,0.5/3,0) -- ++(-0.1,0,0) -- (011) ; 
		\draw[probbar] (101) -- ++(0.1,0,0) -- ++(0,0.5/3,0) -- ++(-0.1,0,0) -- (101) ; 
		\draw[probbar] (110) -- ++(0.1,0,0) -- ++(0,0.5/3,0) -- ++(-0.1,0,0) -- (110) ; 
	\end{tikzpicture}
	\begin{tikzpicture}[baseline,
		partition/.style={ellipse,fill=white, fill opacity=0.7, draw=black, minimum width=0.4cm, minimum height=0.8cm, rotate fit=45},
		x={(0:1.25cm)},y={(90:1.25cm)},z={(45:0.75cm)}]
		\node[bullet] (000) at (0,0,0) {};
		\node[bullet] (001) at (0,0,1) {};
		\node[bullet] (010) at (0,1,0) {}; 
		\node[bullet] (011) at (0,1,1) {};
		\node[bullet] (100) at (1,0,0) {};
		\node[bullet] (101) at (1,0,1) {}; 
		\node[bullet] (110) at (1,1,0) {}; 
		\node[bullet] (111) at (1,1,1) {}; 
		
		\node[partition,fit=(000)] {} ;
		\draw[side] (000) -- (001) ;
		\draw[side] (000) -- (010) ;
		\draw[side] (000) -- (100) ;
		
		\node[partition,fit=(010) (001) (100)] {} ;
		\draw[side] (010) -- (011) ;
		\draw[side] (100) -- (101) ;
		\draw[side] (001) -- (011) ;
		\draw[side] (100) -- (110) ;
		\draw[side] (001) -- (101) ;
		\draw[side] (010) -- (110) ;
		
		\node[partition,fit=(011) (110) (101)] {} ;
		\draw[side] (110) -- (111) ;
		\draw[side] (101) -- (111) ;
		\draw[side] (011) -- (111) ;
		
		\node[partition,fit=(111)] {} ;
		
		\node[bullet] (111) at (1,1,1) {}; 
		
		\draw[probbar] (111) -- ++(0.1,0,0) -- ++(0,0.5,0) -- ++(-0.1,0,0) -- (111) ; 
	\end{tikzpicture}
	$$
	Every invariant measure is a unique convex combination of these ergodic measures.
	
	Let now $p$ be an invariant measure. A function $f:X^n\to Y$ is $p$-almost surely invariant if and only if it is constant within each orbit of nonzero measure. Outside the support of $p$, $f$ can be arbitrary.
	Allowing for measure zero points, we know that $(X^n_\inv,p_\inv)$, i.e.~$X^n$ equipped with the invariant $\sigma$-algebra, is isomorphic in $\cat{PS(Borel)}$ to a finite set, which one can take of as being the set of orbits. For the case above of $\{0,1\}^3$ above, we can take the four-element set $\{0,1,2,3\}$, in bijection with the cells $A_i$ of the partition (if the measure $p$ is not supported on all the cells, one of the elements of $\{0,1,2,3\}$ will have measure zero).
	More generally, the set $\{0,\dots,k-1\}^n$ will have as many orbits as the points in the discrete $(k-1)$-dimensional simplex
	\[
	\{ (s_1,\dots,s_k) \in\N^k : s_1+\dots+s_k = n \} ,
	\]
	which is $\binom{n+k-1}{k}$.
	This way for $k=2$ we get a $1$-dimensional simplex, and the cells $A_0,A_1,A_2,A_3$ as above can be seen as four points along a segment.
	For another example $k=3$ and $n=3$ give us the following discrete triangle,
	$$
	
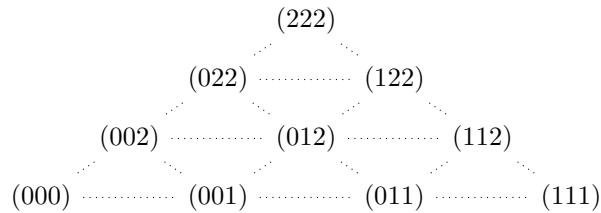
\begin{tikzcd}[column sep=0.2em, row sep=0.5em, every arrow/.append style={dash,dotted}]
		&&& (222) \ar{dl} \ar{dr} \\
		&& (022) \ar{rr} \ar{dl} \ar{dr} && (122) \ar{dl} \ar{dr} \\
		& (002) \ar{rr} \ar{dl} \ar{dr} && (012) \ar{rr} \ar{dl} \ar{dr} && (112) \ar{dl} \ar{dr} \\
		(000) \ar{rr} && (001) \ar{rr} && (011) \ar{rr} && (111)
	\end{tikzcd}
	$$
	where we have taken the representative $(x_1,x_2,x_3)$ in its orbit which first appears in lexicographic order.
	(This is related to multisets and multinomials, a categorical account on them has been given in \cite{jacobs2021multinomial}.)
	Note that, as we saw in \Cref{discchain}, the isomorphism of $\cat{PS(Stoch)}$ between ergodic measures and $X_\inv$ is not quite a bijection of sets (up to indistinguishability): the measure plays a crucial role. We could interpret this isomorphism as a correspondence between \emph{those ergodic measures that contribute to forming $p$, taken with their weights} and $(X_\inv,p_\inv)$. Those ergodic measures that do not contribute are taken with weight zero, and up to isomorphism of $\cat{PS(Stoch)}$, that is the same as not including them at all.
	While in the discrete case we can easily keep track of the measure zero orbits, we will see that in more general cases the role of the measure is more essential.
	
	The map $r:(X^n,p)\to(X^n_\inv,p_\inv)$ is the one that (almost surely) maps each point to its orbit. 
	Its Bayesian inverse, the map $r^\dag:(X^n_\inv,p_\inv)\to (X^n,p)$, is a stochastic map mapping each orbit, stochastically, \emph{to all the points of the orbit with equal probability}:
	
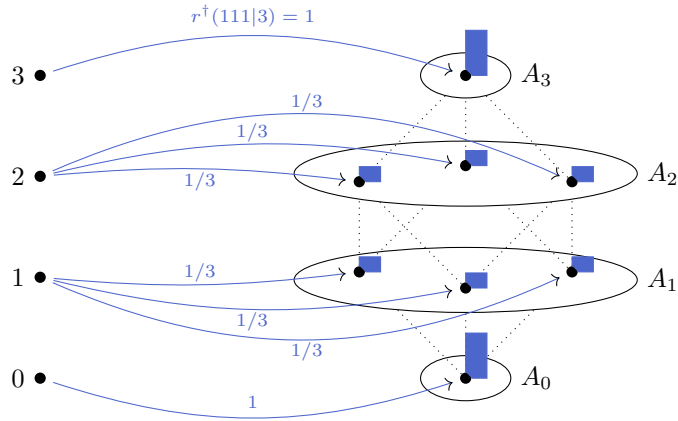
\begin{center}
		\begin{tikzpicture}[baseline,
			partition/.style={ellipse,fill=white, fill opacity=0.7, draw=black, minimum width=0.6cm, minimum height=1.2cm, rotate fit=90},
			transition/.style={midway,  font=\scriptsize, color=probcolor},
			x={(45:2cm)},y={(135:2cm)},z={(90:1.2cm)}]
			\node[bullet] (000) at (0,0,0) {};
			\node[bullet] (001) at (0,0,1) {};
			\node[bullet] (010) at (0,1,0) {}; 
			\node[bullet] (011) at (0,1,1) {};
			\node[bullet] (100) at (1,0,0) {};
			\node[bullet] (101) at (1,0,1) {}; 
			\node[bullet] (110) at (1,1,0) {}; 
			\node[bullet] (111) at (1,1,1) {}; 
			
			\node[partition,fit=(000),label=below:$A_0$] {} ;
			\draw[side] (000) -- (001) ;
			\draw[side] (000) -- (010) ;
			\draw[side] (000) -- (100) ;
			
			\node[partition,fit=(010) (001) (100),label=below:$A_1$] {} ;
			\draw[side] (010) -- (011) ;
			\draw[side] (100) -- (101) ;
			\draw[side] (001) -- (011) ;
			\draw[side] (100) -- (110) ;
			\draw[side] (001) -- (101) ;
			\draw[side] (010) -- (110) ;
			
			\node[partition,fit=(011) (110) (101),label=below:$A_2$] {} ;
			\draw[side] (110) -- (111) ;
			\draw[side] (101) -- (111) ;
			\draw[side] (011) -- (111) ;
			
			\node[partition,fit=(111),label=below:$A_3$] {} ;
			
			\node[bullet] () at (0,0,0) {};
			\node[bullet] () at (0,0,1) {};
			\node[bullet] () at (0,1,0) {}; 
			\node[bullet] () at (0,1,1) {};
			\node[bullet] () at (1,0,0) {};
			\node[bullet] () at (1,0,1) {}; 
			\node[bullet] () at (1,1,0) {}; 
			\node[bullet] () at (1,1,1) {}; 
			
			\node[bullet,label={180:$0$}] (a0) at (-2,2,0) {};
			\node[bullet,label={180:$1$}] (a1) at (1/3-2,1/3+2,1/3) {};
			\node[bullet,label={180:$2$}] (a2) at (2/3-2,2/3+2,2/3) {};
			\node[bullet,label={180:$3$}] (a3) at (1-2,1+2,1) {};
			
			\draw[probbar] (000) -- ++(0.1,-0.1,0) -- ++(0,0,0.5) -- ++(-0.1,0.1,0) -- (000) ; 
			\draw[probbar] (001) -- ++(0.1,-0.1,0) -- ++(0,0,0.5/3) -- ++(-0.1,0.1,0) -- (001) ; 
			\draw[probbar] (010) -- ++(0.1,-0.1,0) -- ++(0,0,0.5/3) -- ++(-0.1,0.1,0) -- (010) ; 
			\draw[probbar] (011) -- ++(0.1,-0.1,0) -- ++(0,0,0.5/3) -- ++(-0.1,0.1,0) -- (011) ; 
			\draw[probbar] (100) -- ++(0.1,-0.1,0) -- ++(0,0,0.5/3) -- ++(-0.1,0.1,0) -- (100) ; 
			\draw[probbar] (101) -- ++(0.1,-0.1,0) -- ++(0,0,0.5/3) -- ++(-0.1,0.1,0) -- (101) ; 
			\draw[probbar] (110) -- ++(0.1,-0.1,0) -- ++(0,0,0.5/3) -- ++(-0.1,0.1,0) -- (110) ; 
			\draw[probbar] (111) -- ++(0.1,-0.1,0) -- ++(0,0,0.5) -- ++(-0.1,0.1,0) -- (111) ; 
			
			\draw[ar,draw=probcolor] (a0) to[bend right=18] node[above,transition]{$1$} (000) ;
			\draw[ar,draw=probcolor] (a1) to[bend right=12] node[inner sep=1pt,pos=0.55,below,transition]{$1/3$} (001) ;
			\draw[ar,draw=probcolor] (a1) to[bend right=6] node[inner sep=1pt,above,transition]{$1/3$} (010) ;
			\draw[ar,draw=probcolor] (a1) to[bend right=24] node[inner sep=1pt,below,transition]{$1/3$} (100) ;
			\draw[ar,draw=probcolor] (a2) to[bend left=6] node[inner sep=1pt,below,transition]{$1/3$} (011) ;
			\draw[ar,draw=probcolor] (a2) to[bend left=24] node[inner sep=1pt,above,transition]{$1/3$} (101) ;
			\draw[ar,draw=probcolor] (a2) to[bend left=12] node[inner sep=1pt,pos=0.55,above,transition]{$1/3$} (110) ;
			\draw[ar,draw=probcolor] (a3) to[bend left=18] node[above,transition]{$r^\dag(111|3)=1$} (111) ;
		\end{tikzpicture}
	\end{center}
	Therefore, the idempotent $e_D=r^\dag r:(X^n,p)\to (X^n,p)$ takes each point and replaces it with a random point of the same orbit, where all points are taken with equal probability. In other words, it \emph{averages the probability within the orbit}.
	More generally, given any morphism $g:(A,b)\to(X,p)$, we have that $e_Dg:(A,b)\to(X,p)$ is an invariant state, which we can interpret as the fact that
	every component of the measure $p$ (parametrized by $A$) is replaced by an average over its orbit.
	
	Similarly, given a kernel $f:(X,p)\to(Y,q)$, the kernel $f e_D:(X,p)\to(Y,q)$ is necessarily invariant. For $Y=\{0,1\}$, this kernel is equivalently a function $X\to[0,1]$, and after precomposing with $e_D$, the resulting function is harmonic, i.e.~constant within the orbits. 
	So once again, $e_D$ can be seen as ``averaging the levels'' within an orbit.

	\subsection{De Finetti's theorem and the Hewitt-Savage law}\label{definetti}
	
	This example in some sense, will be the infinite analogue of the previous example. 
	Consider a standard Borel space $X$, and form the countably infinite product $X^\N$. (Categorically, this is a \emph{Kolmogorov product}, see \cite{fritzrischel2019zeroone} for the definitions and the notation we use.)
	Consider the group $S_\infty$ of all possible finite permutations of $\N$, which acts on $X^\N$ by permuting at each time finitely many components. This defines a countably indexed, deterministic dynamical system $D:S_\infty\to\cat{BorelStoch}$ acting on the object $X^\N$. 
	A measure $p$ on $X^\N$ which is invariant under all these permutations is called an \emph{exchangeable measure}.
	If we have such a measure, we also get a dynamical system in $\cat{PS(Borel)}$. 
	
	In \cite[Theorem~4.4]{fritz2021definetti}, and using the addition of \cite[Section~8]{moss2022probability}, the following category-theoretic version of de Finetti's theorem is stated and proven categorically:
	
	\begin{theorem}\label{definettithm}
		The limit in $\cat{BorelStoch}$ of $D:S_\infty\to\cat{BorelStoch}$ is given by the space of measures $PX$, with the following cone.
		$$
		\begin{tikzcd}
			PX \ar{r}{\cop_\N} & (PX)^\N \ar{r}{\samp^\N} & X^\N 
		\end{tikzcd}
		$$
	\end{theorem}
	
	While in the original source the theorem is proven purely in terms of categorical axioms, here we are mostly interested in its instantiation in $\cat{BorelStoch}$.
	Let us write the kernel above in measure-theoretic terms. For brevity, we will denote the set $A_1\times\dots\times A_n\times X\times X\times\dots$ simply by $A_1\times\dots\times A_n$. Now on such a set, and for $p\in PX$, the kernel above gives
	\begin{equation}\label{infprodmeas}
		(\samp^\N\circ\cop_\N) (A_1\times\dots\times A_n| p) = p(A_1)\cdots p(A_n) ,
	\end{equation}
	i.e.~this kernel is taken repeated, independent (identical) copies of $p$. 
	
	Let us now show that $PX$, with this cone, is also a limit in $\cat{PS(Borel)}$.
	
	\begin{theorem}\label{samelimit}
		Let $D:M\to\cat{BorelStoch}$ be a dynamical system acting on $X$, with $M$ countable, and let $\ell:L\to X$ be a limiting cone for $D$ in $\cat{BorelStoch}$.
		Let $p$ be an invariant probability measure on $X$. Then
		\begin{itemize}
			\item there exists a unique probability measure $\lambda$ on $L$ making the kernel $\ell$ measure-preserving;
			\item (the equivalence class of) $\ell:(L,\lambda)\to (X,p)$ is a limiting cone in $\cat{PS(Borel)}$ for the corresponding measure-preserving dynamical system.
		\end{itemize}
	\end{theorem}
	\begin{proof}
		First of all, if $p$ is an invariant measure on $X$, then in particular it is a morphism $p:1\to X$ such that for all $m\in M$, $m\circ p=p$. By the universal property of $L$ in $\cat{BorelStoch}$, there is then a unique morphism (i.e.~measure) $\lambda:1\to L$ such that $\ell\circ\lambda=p$.
		
		Consider now a morphism $g:(Y,q)\to (X,p)$ which is almost surely left-invariant, i.e.~$mg\aseq g$ for all $m\in M$. 
		Since $Y$ is standard Borel and $M$ is countable, there exists a subset $A\subseteq Y$ of measure one (and also standard Borel), such that for all $m\in M$, for all $a\in A$, and for all measurable $B\subseteq X$, 
		$$
		\int_X m(B|x)\,g(dx|a) = g(B|a) .
		$$
		In other words, if $i:A\to Y$ denotes the kernel induced by the inclusion map, for all $m\in M$ the outer triangle in the diagram below commutes strictly, not just almost surely, and so it is a cone in $\cat{BorelStoch}$.
		$$
		
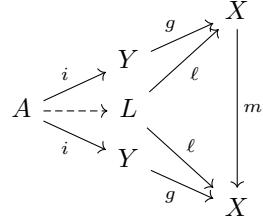
\begin{tikzcd}[row sep=tiny]
			&& X \ar{dddd}{m} \\
			& Y \ar{ur}{g} \\
			A \ar{ur}{i} \ar{dr}[swap]{i} \ar[dashed]{r} & L \ar{uur}[swap]{\ell} \ar{ddr}{\ell} \\
			& Y \ar{dr}[swap]{g} \\
			&& X
		\end{tikzcd}
		$$    
		By the universal property of the limit in $\cat{BorelStoch}$, there exists a unique kernel $\tilde{g}:A\to L$ making the diagram above commute.
		Note that $\tilde{g}$ is measure-preserving as a kernel: since $g\circ i=\ell\circ\tilde{g}$, we have that $p=g\circ i\circ q'=\ell\circ\tilde{g}\circ q'$, and since $\lambda$ is the unique measure on $L$ such that $\ell\circ\lambda=p$, it must be that $\tilde{g}\circ q'=\lambda$, so $\tilde{g} : (A, q') \to (L, \lambda)$ is a kernel in $\cat{PS(Borel)}$.
		
		Now by \Cref{isomodzero}, the map $i$ induces an isomorphism $(A,q')\to(Y,q)$ of $\cat{PS(Borel)}$, where $q'$ denotes the restriction of $q$ to $A$. Take now a kernel $j:Y\to A$ which, in $\cat{PS(Borel)}$, is inverse to $i$, and since the following diagram commutes,
		$$
		
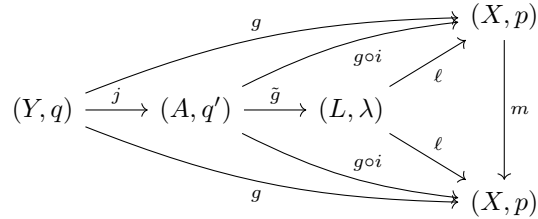
\begin{tikzcd}
			&&& (X,p) \ar{dd}{m} \\
			(Y,q) \ar[bend left=10]{urrr}{g} \ar{r}{j} \ar[bend right=10]{drrr}[swap]{g} & (A,q') \ar[bend left=10]{urr}[swap]{g\circ i} \ar[bend right=10]{drr}{g\circ i} \ar{r}{\tilde{g}} & (L,\lambda) \ar{ur}[swap]{\ell} \ar{dr}{\ell} \\
			&&& (X,p)
		\end{tikzcd}
		$$
		we have a kernel, namely $\tilde{g}\circ j$, such that $\ell\circ(\tilde{g}\circ j)\aseq g$.
		
		For uniqueness, suppose that another kernel $h:Y\to L$ satisfies $\ell\circ h\aseq g$. 
		Then once again, there exists a subset $B\subseteq Y$ of measure one where $\ell\circ h$ and $g$ are equal. More explicitly, denote by $i':B\to Y$ the kernel induced by the inclusion map, and by $j':Y\to B$ a kernel giving the inverse of $i'$ in $\cat{PS(Borel)}$.
		Then in the following diagram of $\cat{BorelStoch}$,
		$$
		\begin{tikzcd}
			&&& X\ar{dd}{m} \\
			B \ar{r}{i'} & Y \ar[bend left=20]{urr}{g} \ar{r}[shift left]{\tilde{g}\circ j} \ar[shift right]{r}[swap]{h} \ar[bend right=20]{drr}[swap]{g} & L \ar{ur}[swap]{\ell} \ar{dr}{\ell} \\
			&&& X
		\end{tikzcd}
		$$
		we have that $\ell\circ (\tilde{g}\circ j)\circ i'=\ell\circ h\circ i'$. 
		By the universal property of $L$, by uniqueness it must then mean that $(\tilde{g}\circ j)\circ i'=h\circ i'$.
		But then
		$$
		h \aseq h\circ i'\circ j' \aseq (\tilde{g}\circ j)\circ i' \circ j' \aseq \tilde{g}\circ j .
		$$
		Therefore there exists a unique ($q$-a.s.)~kernel $Y\to L$ making the diagram above commute (almost surely), and so $(L,\lambda)$ is a limit in $\cat{PS(Borel)}$. 
	\end{proof}
	
	\begin{corollary}[Almost-sure version of de Finetti's theorem]\label{asdefinetti}
		Consider the system $D:S_\infty\to\cat{BorelStoch}$ acting on $X^\N$ by finite permutations. 
		Consider an exchangeable measure $p$ on $X^\N$, which induces a dynamical system $D'$ in $\cat{PS(Borel)}$.
		Then 
		\begin{itemize}
			\item There is a unique measure $\mu$ on $PX$ such that $\samp^\N\circ \cop_\N\circ \mu=p$;
			\item The limit of $D'$ in $\cat{PS(Borel)}$ of is given by $(PX,\mu)$, with the following cone.
			$$
			\begin{tikzcd}
				(PX,\mu) \ar{r}{\cop_\N} & (PX^\N,\cop_\N\circ\mu) \ar{r}{\samp^\N} & (X^\N,p)
			\end{tikzcd}
			$$
		\end{itemize}        
	\end{corollary}
	
	It follows that $(PX,\mu)$ is also the colimit of $D'$. In other words, $(X^\N_\inv,p_\inv)$ and $(PX,\mu)$ are isomorphic in $\cat{PS(Borel)}$.
	We have a commutative diagram of $\cat{PS(Borel)}$ as follows.
	\begin{equation*}
		\begin{tikzcd}[column sep=small]
			& (X^\N,p) \ar{dr}{r} \\
			(PX,\mu) \ar{ur}{\samp^\N\circ \cop_\N} \ar[leftrightarrow]{rr}[swap]{\cong} && (X_\inv,p_\inv) 
		\end{tikzcd}
	\end{equation*}
	Therefore we can take as colimit $(X_\inv,p_\inv)$ exactly the space $(PX,\mu)$.
	This way, the map $\samp^\N\circ \cop_\N:PX\to X^\N$ is the map $r^\dag$.
	A representative of the map $r$ was called $p_{|tail}^\sharp$ in \cite{fritz2021definetti}.
	
	Recall from the previous section that for permutations of finite sequences we can take $X_\inv$ to be a discrete simplex.
	The situation is similar here: for a finite set $X$, the set $PX$ can be seen as the whole simplex. 
	The \emph{law of large numbers} is related to the fact that for increasingly longer sequences, the discrete simplex becomes closer and closer to the simplex $PX$. (See also \cite{fritz2019kantorovich}.)
	A categorical treatment of the law of large numbers will be left for future work.
	
	In this setting we also get a version of the Hewitt-Savage zero-one law.
	A categorical version of such law was stated and proven in \cite[Theorem~5.4]{fritzrischel2019zeroone}.
	Here we can categorically prove the almost-sure~version, here stated for $\cat{C}=\cat{BorelStoch}$:
	
	\begin{theorem}\label{hs}
		Let $p$ be an exchangeable measure on $X^\N$.
		The following conditions are equivalent.
		\begin{enumerate}
			\item\label{conderg} The measure $p$ is ergodic (under permutations);
			\item\label{conddet} The corresponding measure $p_\inv$ on $X_\inv$ is deterministic;
			\item\label{condinv} For every $p$-a.s. deterministic, $p$-a.s. permutation-invariant morphism $s:X^\N\to Y$, the measure $s\circ p$ is deterministic.
			\item\label{conddelta} The corresponding measure $\mu$ on $PX$ is a Dirac delta $\delta_q$ at some $q\in PX$;
			\item\label{condind} The measure $p$ displays independence of all the $X$ in $X^\N$, i.e.~it is an infinite product measure (using Kolmogorov's extension theorem, all its finite marginals are product measures).
		\end{enumerate}
	\end{theorem}
	
	\begin{proof}
		\begin{itemize}
			\item $\ref{conderg}\Leftrightarrow\ref{conddet}\Leftrightarrow\ref{condinv}$: This is exactly \Cref{deferg}.
			
			\item $\ref{conddet}\Leftrightarrow\ref{conddelta}$: Since $(PX,\mu)$ and $(X_\inv,p_\inv)$ are isomorphic in $\cat{PS(Borel)}$, the measure $\mu$ is deterministic if and only if the measure $p_\inv$ is. Now since $PX$ is standard Borel (not just up to isomorphism), $\mu$ is deterministic if and only if it is a Dirac delta. 
			
			\item $\ref{conddelta}\Rightarrow\ref{condind}$: Recall that the limit cone kernel $\samp^\N\circ \cop_\N:PX\to X^\N$ maps a measure $q$ to the infinite product measure \eqref{infprodmeas}. 
			If $\mu=\delta_q$,  
			\begin{align*}
				p(A_1\times\dots\times A_n) &= (\samp^\N\circ \cop_\N\circ\delta_q) (A_1\times\dots\times A_n) \\
				&= \int_{PX} (\samp^\N\circ \cop_\N) (A_1\times\dots\times A_n|q') \,\delta_q(dq') \\
				&= q(A_1)\cdots q(A_n) .
			\end{align*}
			
			\item $\ref{condind}\Rightarrow\ref{conddelta}:$ Since $p$ is exchangeable, it must then be the infinite product of the \emph{same} measure $q$, for $q\in PX$ which can be obtained as the first (or any other) marginal of $p$ on $X$. 
			In other words, necessarily $p$ is in the form $\samp^\N\circ \cop_\N\circ \delta_q$ for some $q\in PX$. 
			Since the universal cone map has a left-inverse (its Bayesian inverse), we must have that 
			$$
			\mu = (\samp^\N\circ \cop_\N)^\dag \circ \samp^\N\circ \cop_\N \circ \delta_q = \delta_q .
			$$
		\end{itemize}
	\end{proof}

	\subsection{Bernoulli shifts}\label{bernoulli}
	
	As in the previous section, consider a Kolmogorov product $X^\N$. 
	Instead of permutations, we consider \emph{shifts}. Denote by $\sigma:X^\N\to X^\N$ the map discarding the first coordinate $(x_0,x_1,x_2,\dots)\mapsto(x_1,x_2,x_3,\dots)$, induced by the endomorphism $n\mapsto n+1$ of $\N$. This map was called $X^s$ in \cite{fritz2021definetti}.
	
	Note that the map $\cop_\N:X\to X^\N$ is not only permutation-(left-)invariant, but shift-invariant as well.
	Because of this, every exchangeable measure $p$, in the sense of permutations, as in the previous section, is shift-invariant as well. Indeed, using \Cref{asdefinetti},
	\begin{align*}
		\sigma\circ p &= \sigma\circ\samp^\N\circ\cop_\N\circ\mu \\
		&= \samp^\N\circ\sigma\circ\cop_\N\circ\mu \\
		&= \samp^\N\circ\cop_\N\circ\mu \\
		&= p .
	\end{align*}
	
	In this section we will only consider exchangeable measures.
	We also use the following statement, which can be considered a categorical version of the fact that the permutation-invariant $\sigma$-algebra and the shift-invariant $\sigma$-algebra are isomorphic up to measure zero. 
	
	\begin{proposition}[{\cite[Proposition~4.5]{fritz2021definetti}}]
		Let $p$ be an exchangeable measure on $X^\N$.
		A morphism $f:(X,p)\to (Y,q)$ is a.s.~invariant under permutations if and only if it is a.s.~invariant under shifts.
	\end{proposition}
	
	Explicitly, here is what this means in terms of $\sigma$-algebras:
	
	\begin{corollary}
		Let $p$ be an exchangeable measure on $X^\N$.
		Denote by $X^\N_\inv$ the set $X^\N$ equipped with the $p$-a.s.\ permutation-invariant $\sigma$-algebra, and by $X^\N_\sinv$ the set $X^\N$ equipped with the $p$-a.s.\ shift-invariant $\sigma$-algebra (and do the same for measures). Then $(X^\N_\inv,p_\inv)$ and $(X^\N_\sinv,p_\sinv)$ are isomorphic in $\cat{PS(Borel)}$.
	\end{corollary}
	\begin{proof}
		Denote by $r:(X^\N,p)\to (X^\N_\inv,p_\inv)$ the colimit cone for permutations, and by $s:(X^\N,p)\to (X^\N_\sinv,p_\sinv)$ the colimit cone for shifts. By the proposition above, $r$ is also a.s.\ shift-invariant, and so there exists a unique a.s.\ deterministic morphism $\tilde{r}:(X^\N_\sinv,p_\sinv)\to (X^\N_\inv,p_\inv)$ such that $r\aseq \tilde{r}\circ s$. Similarly, $s$ is also a.s.\ permutation-invariant, and so there exists a unique a.s.\ deterministic morphism $\tilde{s}:(X^\N_\inv,p_\inv)\to (X^\N_\sinv,p_\sinv)$ such that $s\aseq\tilde{s}\circ r$.
		By uniqueness, $\tilde{r}$ and $\tilde{s}$ are inverse to each other.
	\end{proof}
	
	\begin{corollary}
		The cone $(PX,\mu)$ as constructed in \Cref{asdefinetti}, together with $\samp^\N\circ\cop_\N:PX\to X^\N$, is the limit of $(X^\N,p)$ also under \emph{shifts}, not just permutations.
	\end{corollary}
	Note that this depends crucially on the fact that we chose $p$ to be exchangeable. Not all shift-invariant measures are exchangeable, and so they do not all give $PX$ as a limit.
	\begin{proof}
		In $\cat{PS(Borel)}$, we have $(PX,\mu) \cong (X_\inv,p_\inv) \cong (X_\sinv,p_\sinv)$, therefore $(PX,\mu)$ is also an invariant object for shifts. (Hence, in particular, a colimit and a limit.)
	\end{proof}
	
	We get an analogue of \Cref{hs}:
	
	\begin{theorem}\label{bsh}
		Let $p$ be an exchangeable measure on $X^\N$.
		The following conditions are equivalent.
		\begin{enumerate}
			\item
			The measure $p$ is ergodic under shifts;
			\item
			The corresponding measure $p_\inv$ on $X_\sinv$ is deterministic;
			\item
			For every $p$-almost surely deterministic, $p$-almost surely shift-invariant morphism $s:X^\N\to Y$, the measure $s\circ p$ is deterministic.
			\item
			The corresponding measure $\mu$ on $PX$ is a Dirac delta $\delta_q$ at some $q\in PX$;
			\item
			The measure $p$ displays independence of all the $X$ in $X^\N$, i.e.~it is an infinite product measure (using Kolmogorov's extension theorem, all its finite marginals are product measures).
		\end{enumerate}
	\end{theorem}
	
	The proof is completely analogous to the one of \Cref{hs}.
	
	When any of the conditions above is satisfied (and so all of them are), the shifts $\sigma:(X^\N,p)\to(X^\N,p)$ are sometimes called \emph{Bernoulli shifts}, and they are known to be an ergodic dynamical system.

	\appendix

	\section{A short review of Markov categories}\label{appendix}
	
	Here we review the main notions of the theory of Markov categories used in the rest of this work.
	For further explanations, see \cite[Sections~2.2 and 2.3]{moss2022ergodic}, as well as the original sources \cite{chojacobs2019strings,fritz2019synthetic}.
	Intuitively, a Markov category is a monoidal category where the morphisms can be considered stochastic map of some kind. The precise definition will be given shortly.
	We use the \emph{string diagram} notation for monoidal categories, where morphisms are oriented from bottom to top, as follows:
	\ctikzfig{morphism}
	
	If we have two morphisms $k : X \to Y$ and $h : Z \to W$, we can represent the tensor morphism $k \otimes h : W \otimes Z \to Y \otimes W$ by the following string diagram:
	\ctikzfig{independence}
	We can interpret this product in terms of probability by considering the transitions given by $k$ and $h$ to be independent (hence why the right-hand side diagram is disconnected).
	
	\begin{definition}
		A Markov category is a symmetric monoidal category $(\cat{C}, \otimes, I)$ where:
		\begin{itemize}
			\item Each object $X \in \cat{C}$ is equipped with two maps $\cop_X : X \to X \otimes X$ and $\del_X : X \to I$. They are represented by the following string diagrams:
			\ctikzfig{copy_discard}
			and equip $X$ with the structure of a commutative comonoid, i.e.~they satisfy the following properties:
			\ctikzfig{properties}
			\item $\mathcal{C}$ is semi-cartesian, i.e.\ the monoidal unit $I$ is terminal.
		\end{itemize}
	\end{definition}

	Here are the main examples of Markov categories which we consider for categorical probability.
	
	\begin{example}
		The category $\cat{FinStoch}$ describes finite sets and stochastic maps between them. It is defined as follows:
		\begin{itemize}
			\item Objects are finite sets, that can be interpreted as sets of states;
			\item Given two finite sets $X$ and $Y$, a morphism $k : X \to Y$ is a \emph{stochastic matrix}, i.e.\ a map $X \times Y \to [0, 1]$, whose entries we denote as $k(y|x)$, such that $\sum_{y \in Y} k(y|x) = 1$. We can interpret $k(y|x)$ as the transition probability from state $x$ to state $y$;
			\item The identity morphisms are identity matrices, and the composition of two morphisms $k : X \to Y$, $h: Y \to Z$ is given by the Chapman-Kolmogorov formula:
			$$
			h \circ k(z|x) = \sum_{y \in Y} h(z|y) k(y|x) ,
			$$
			which says that sequential transitions are independent, as in a Markov chain;
			\item The monoidal unit is the one-point set $1=\{u\}$;
			\item The tensor product is given, on objects, by the cartesian product of the sets of states, and on morphisms, by the product of the transition probabilities: given $k:X\to Y$ and $h:Z\to W$, 
			$$
			(k\otimes h) (y,w|x,z) \coloneqq k(y|x)\,k(w|z) ;
			$$
			\item The copy and discard maps are given as follows:
			$$
			\cop(x'',x'|x) = \begin{cases}
				1 & x=x'=x'' ; \\
				0 & \mbox{otherwise} ;
			\end{cases}
			\qquad 
			\del(u|x) \coloneqq 1 .
			$$
		\end{itemize}
	\end{example}

	\begin{example}
		The category $\cat{Stoch}$ describes measurable spaces and measurable stochastic maps (Markov kernels) between them. It is defined as follows:
		\begin{itemize}
			\item Objects are measurables spaces, that can again be interpreted as sets of states;
			\item Given two measurable spaces $(X, \Sigma_X)$ and $(Y, \Sigma_Y)$, a morphism $k : X \to Y$ is a \emph{Markov kernel}. Explicitly, it is a map $X \times \Sigma_Y \to [0, 1]$, which we denote by $k(B|x)$, such that $k(- | x):\Sigma_Y\to[0,1]$ is a probability measure and that $k(B| -):X\to [0,1]$ is a measurable function. We can interpret $k(y|x)$ as the probability to arrive in $B$ if the current state is $x$.
			\item The composition of Markov kernels $k : X \to Y$, $h: Y \to Z$ is again given by the Chapman-Kolmogorov formula, this time with an integral instead of a sum:
			$$
			h \circ k(B|x) = \int_Y h(B|y) \, k(dy|x)
			$$
			\item The monoidal structure and the copy and discard maps are defined similarly to $\cat{FinStoch}$.
		\end{itemize}
		
		We call $\cat{BorelStoch}$ the full subcategory of $\cat{Stoch}$ whose objects are standard Borel spaces.
	\end{example}

	In a Markov category, we call a \emph{state} on the object $X$ a morphism $p:I\to X$, from the monoidal unit. In string diagrams, we denote it as follows.
	\ctikzfig{distribution}
	In $\cat{FinStoch}$, $\cat{BorelStoch}$ and $\cat{Stoch}$, states are exactly the probability distributions, which one can interpret as ``random states''. 
	
	One of the most important notions in the theory is that of \emph{determinism}.
	\begin{definition}
		A morphism $f:X\to Y$ in a Markov category is said to be \emph{deterministic} if the following equation holds.
		\ctikzfig{determinism}
	\end{definition}
	
	They formalise the following intuitive ideas of determinism:
	\begin{itemize}
		\item In $\cat{FinStoch}$, deterministic morphisms are precisely the stochastic matrices $f:X\to Y$ such that for all $x$ and $y$, $f(y|x)$ is zero or one. 
		This can always be written as $f(y|x) = \delta_{y,g(x)}$ for some function $g:X\to Y$.
		
		\item In $\cat{Stoch}$, deterministic morphisms are precisely those Markov kernels $f:X\to Y$ such that for all $x\in X$ and all measurable $B\subseteq Y$, $f(B|x)$ is zero or one. 
		If $Y$ is standard Borel, such kernels can always be written as 
		\begin{equation}\label{delta}
			f(B|x) = \delta_{g(x)}(B) = 1_B(g(x)) = \begin{cases}
				1 & g(x) \in B ; \\
				0 & g(x) \notin B 
			\end{cases}
		\end{equation}
		for some measurable function $g:X\to Y$ (where the choice of $g$ does not depend on $x$ and $B$).
		Even when $Y$ is not standard Borel, there exists a measurable space $DY$, (deterministically) isomorphic to $Y$ in $\cat{Stoch}$, such that every deterministic morphism $f:X\to DY$ can be written in the form \eqref{delta} for a unique measurable function $g:X\to DY$. (For more on this, see \cite[Section~5]{moss2022probability}.)
	\end{itemize}
	It can be easily checked that deterministic morphisms are closed under composition. 
	
	\subsection{Almost-sure equality and conditionals}\label{as}
	
	\begin{definition}\label{defas}
		In a Markov category, let $p$ be a state on $X$, and let $f,g:X\to Y$ be morphisms. We say that $f$ and $g$ are \emph{$p$-almost surely equal} (or \emph{$p$-a.s.\ equal}), and write $f\aseq_p g$ (or $f\aseq g$ when $p$ is clear) if the following equality holds.
		\ctikzfig{ase-states}
	\end{definition}
	
	A particularly important condition for this work is \emph{almost sure determinism}:
	
	\begin{definition}
		In a Markov category, let $p$ be a state on $X$. A morphism $f:X\to Y$ is said to be \emph{$p$-a.s.\ deterministic} if the following equation holds.
		\ctikzfig{asdet}
	\end{definition}
	If $f$ is a.s.\ equal to a deterministic morphism, then it is a.s.\ deterministic. The converse is true in $\cat{BorelStoch}$. 
	
	Almost surely deterministic morphisms are, in a certain sense, closed under composition:
	\begin{proposition}\label{a-s det composition}
		Let $f : X \to Y$, $g : Y \to Z$ be two morphisms of a causal Markov category $\cat{C}$, and $p$ be a state on $X$. Assume $f$ is $p$-a.s.~deterministic and $g$ is $(f \circ p)$-a.s.~deterministic. Then $g \circ f$ is $p$-a.s.~deterministic.
	\end{proposition}
	
	\begin{proof}
		The hypotheses give us the following string diagrams equality:
		\ctikzfig{det_eq_strengthening}
		where the first equality holds because of $[f \circ p]$-a.s.~determinism of $g$ together with the causality of $\cat{C}$, and the second one holds because $f$ is $p$-a.s.~deterministic. But marginalizing over the second output and using $p$-a.s.~determinism of $f$ again, we obtain:
		\ctikzfig{gf_as_det}
		and so $g \circ f$ is $p$-a.s.~deterministic.
	\end{proof}
	
	Let us now turn to conditionals. 
	
	\begin{definition}
		A Markov category $\cat{C}$ has \emph{conditional distributions} if for every distribution $\psi : I \to X \otimes Y$, there exists a morphism $\psi_{|X} : X \to Y$ that satisfies:
		\ctikzfig{conditional_distributions}
		
		It has \emph{conditionals} if for every morphism $f : A \to X \otimes Y$, there exists a morphism $f_{|X} : X \to Y$ that satisfies:
		\ctikzfig{conditionals}
	\end{definition}
	
	The categories $\cat{FinStoch}$ and $\cat{BorelStoch}$ have conditionals, and they corresponds to the usual \emph{regular conditional distributions}. See \cite[Section~3]{chojacobs2019strings} and \cite[Section~11]{fritz2019synthetic} for more on this.
	
	A particular conditional distribution is given by the Bayesian inverse of a morphism.
	\begin{definition}
		In a Markov category, let $p$ be a state on $X$, and let $f:X\to Y$ be a morphism.
		A \emph{Bayesian inverse} of $f$ relative to $p$ is a morphism $f^\dag_p:Y\to X$ (or simply $f^\dag$, when $p$ is clear), such that
		\ctikzfig{bayinv}
	\end{definition}
	
	In $\cat{FinStoch}$ and $\cat{BorelStoch}$, this recovers the usual notions of Bayesian inverses. (See again \cite[Section~3]{chojacobs2019strings} and \cite[Section~11]{fritz2019synthetic}.)
	If a Bayesian inverse $f^\dag_p$ of $f$ exists, then any (other) morphism $g:Y\to X$ is a Bayesian inverse of $f$ if and only if it is $(f\circ p)$-a.s.\ equal to $f^\dag_p$. Similar things can be said about conditional distributions.

	\begin{definition}[{\cite[Definition~2.8]{moss2022ergodic}}]
		An object $X$ in a Markov category is said to have \emph{disintegrations} if for every state $p$ on $X$ and every deterministic morphism $f:X\to Y$, the Bayesian inverse $f^\dag_p$ exists.
	\end{definition}
	
	By Rokhlin's disintegration theorem, in the category $\cat{Stoch}$, every standard Borel space has disintegrations. (See also \cite[Section~2.3]{moss2022ergodic}.)

	\subsection{Causality and positivity}\label{causpos}
	
	Here are additional axioms on Markov categories. For more details, see \cite[Section~11]{fritz2019synthetic} and \cite{fritz2022dilations}.
	
	\begin{definition}
		A Markov category $\cat{C}$ is said to be \emph{causal} if, whenever the morphisms $f, g, h_1$ and $h_2$ satisfy:
		\ctikzfig{causality_cond}
		then they also satisfy:
		\ctikzfig{causality_csq}
	\end{definition}
	
	\begin{proposition}[{\cite[Proposition~11.34]{fritz2019synthetic}}]
		If a Markov category has conditionals, it is causal.
	\end{proposition}
	
	Since they have conditionals, $\cat{FinStoch}$ and $\cat{BorelStoch}$ are causal. 
	It turns out that $\cat{Stoch}$ is causal too, see \cite[Example~11.35]{fritz2019synthetic}.

	\begin{definition}
		A Markov category is said to be positive if, when $f : X \to Y$ and $g : Y \to Z$ are such that $gf$ is deterministic, then the following string diagram equality holds:
		\ctikzfig{positivity}
		
		It is said to be relatively positive if, when $p : A \to X$, $f : X \to Y$ and $g : Y \to Z$ are such that $gf$ is $p$-almost surely deterministic, then the following string diagram equality holds:
		\ctikzfig{relative_positivity}
	\end{definition}
	
	Note that a relatively positive Markov category is positive, for we can take $p$ to be the identity in the definition. 
	
	\begin{proposition}[{\cite[Corollary~2.28]{fritz2022dilations}}]
		A causal Markov category is relatively positive.
	\end{proposition}
	
	Therefore $\cat{FinStoch}$, $\cat{BorelStoch}$ and $\cat{Stoch}$ are positive and relatively positive.
	In a positive Markov category and so in particular in these categories, every isomorphism is deterministic \cite[Remark~11.28]{fritz2019synthetic}. (An almost sure version of this statement is given by \Cref{iso bayesian}.)

	\bibliographystyle{alpha}
	\bibliography{markov}
	
\end{document}

%% file: markov.tikzstyles
\tikzstyle{morphism}=[fill=white, draw=black, shape=rectangle]
\tikzstyle{medium box}=[fill=white, draw=black, shape=rectangle, minimum width=0.7cm, minimum height=0.7cm]
\tikzstyle{large morphism}=[fill=white, draw=black, shape=rectangle, minimum width=1.7cm, minimum height=1cm]
\tikzstyle{bn}=[fill=black, draw=black, shape=circle, inner sep=1.5pt]
\tikzstyle{state}=[fill=white, draw=black, regular polygon, regular polygon sides=3, minimum width=0.8cm, shape border rotate=180, inner sep=0pt]
\tikzstyle{medium state}=[fill=white, draw=black, regular polygon, regular polygon sides=3, minimum width=1.3cm, inner sep=0pt, shape border rotate=180]
\tikzstyle{large state}=[fill=white, draw=black, regular polygon, regular polygon sides=3, minimum width=2.2cm, shape border rotate=180, inner sep=0pt]
\tikzstyle{wide state}=[fill=white, draw=black, shape=isosceles triangle, minimum width=0.8cm, shape border rotate=270, inner sep=1.4pt, minimum height=0.5cm, isosceles triangle apex angle=80]
\tikzstyle{wn}=[fill=white, draw=black, shape=circle, inner sep=1.5pt]
\tikzstyle{blue morphism}=[fill=white, draw={rgb,255: red,15; green,0; blue,150}, shape=rectangle, text={rgb,255: red,15; green,0; blue,150}, tikzit category=blue]
\tikzstyle{red morphism}=[fill=white, draw={rgb,255: red,150; green,0; blue,2}, shape=rectangle, text={rgb,255: red,150; green,0; blue,2}, tikzit category=red]
\tikzstyle{blue state}=[fill=white, draw={rgb,255: red,15; green,0; blue,150}, shape=circle, regular polygon, regular polygon sides=3, minimum width=0.8cm, shape border rotate=180, inner sep=0pt, text={rgb,255: red,15; green,0; blue,150}, tikzit category=blue]
\tikzstyle{blue node}=[fill={rgb,255: red,15; green,0; blue,150}, draw={rgb,255: red,15; green,0; blue,150}, shape=circle, tikzit category=blue, inner sep=1.5pt]
\tikzstyle{blue}=[text={rgb,255: red,15; green,0; blue,150}, tikzit draw={rgb,255: red,191; green,191; blue,191}, tikzit category=blue, tikzit fill=white, inner sep=0mm]
\tikzstyle{blue wide state}=[fill=white, draw={rgb,255: red,15; green,0; blue,150}, text={rgb,255: red,15; green,0; blue,150}, shape=isosceles triangle, minimum width=0.8cm, shape border rotate=270, inner sep=1.4pt, minimum height=0.5cm, isosceles triangle apex angle=80]
\tikzstyle{red node}=[fill={rgb,255: red,150; green,0; blue,2}, draw={rgb,255: red,150; green,0; blue,2}, shape=circle, inner sep=1.5pt]
\tikzstyle{Purple node}=[fill={rgb,255: red,120; green,0; blue,120}, draw={rgb,255: red,120; green,0; blue,120}, text={rgb,255: red,120; green,0; blue,120}, shape=circle, inner sep=1.5pt]
\tikzstyle{red}=[text={rgb,255: red,150; green,0; blue,2}, inner sep=0mm, tikzit fill=white, tikzit draw={rgb,255: red,191; green,191; blue,191}]
\tikzstyle{purple}=[text={rgb,255: red,150; green,0; blue,150}, inner sep=0mm, tikzit fill=white, tikzit draw={rgb,255: red,191; green,191; blue,191}]
\tikzstyle{white morphism}=[fill=white, draw=white, shape=rectangle, tikzit draw={rgb,255: red,139; green,139; blue,139}]
\tikzstyle{leak morphism}=[fill=white, draw={rgb,255: red,120; green,0; blue,85}, shape=rectangle, text={rgb,255: red,120; green,0; blue,85}, tikzit category=leak]
\tikzstyle{leak}=[text={rgb,255: red,120; green,0; blue,85}, inner sep=0mm, tikzit fill=white, tikzit draw={rgb,255: red,191; green,191; blue,191}, tikzit category=leak]
\tikzstyle{leak node}=[fill={rgb,255: red,120; green,0; blue,85}, draw={rgb,255: red,120; green,0; blue,85}, shape=circle, inner sep=1.5pt, tikzit category=leak]
\tikzstyle{horiz state}=[fill=white, draw=black, regular polygon, regular polygon sides=3, minimum width=1cm, shape border rotate=90, inner sep=0pt]
\tikzstyle{grey node}=[fill={rgb,255: red,194; green,194; blue,194}, draw={rgb,255: red,194; green,194; blue,194}, shape=circle, inner sep=1.5pt]
\tikzstyle{grey morphism}=[fill=white, draw={rgb,255: red,194; green,194; blue,194}, shape=rectangle]

\tikzstyle{arrow}=[->]
\tikzstyle{dashed box}=[-, dashed]
\tikzstyle{blue arrow}=[-, draw={rgb,255: red,15; green,0; blue,150}, tikzit category=blue]
\tikzstyle{red arrow}=[-, draw={rgb,255: red,150; green,0; blue,2}, tikzit category=red]
\tikzstyle{purple arrow}=[->, draw={rgb,255: red,120; green,0; blue,120}, >=stealth, shorten <=2pt, shorten >=2pt]
\tikzstyle{protected purple arrow}=[->, draw={rgb,255: red,120; green,0; blue,120}, >=stealth, shorten <=2pt, shorten >=2pt, preaction={line width=1.8pt, white, draw}]
\tikzstyle{mapsto}=[{|->}]
\tikzstyle{double wire}=[-, double]
\tikzstyle{curly brace}=[-, draw=none, tikzit draw={rgb,255: red,128; green,0; blue,128}]
\tikzstyle{protected}=[-, preaction={line width=1.8pt,white,draw}]
\tikzstyle{leak arrow}=[-, tikzit draw={rgb,255: red,150; green,0; blue,120}]
\tikzstyle{protected leak arrow}=[-, tikzit draw={rgb,255: red,150; green,0; blue,120}]
\tikzstyle{hollow arrow}=[-, very thin, white, preaction={line width=0.7pt,draw={rgb,255: red,120; green,0; blue,85}}, tikzit category=leak, tikzit draw={rgb,255: red,150; green,0; blue,120}]
\tikzstyle{protected hollow arrow}=[-, very thin, white, preaction={line width=0.7pt,draw={rgb,255: red,120; green,0; blue,85},preaction={line width=2.1pt,white,draw}}, tikzit category=leak, tikzit draw={rgb,255: red,150; green,0; blue,120}]
\tikzstyle{over arrow}=[-, black, preaction={draw=white, double}]
\tikzstyle{grey arrow}=[-, draw={rgb,255: red,194; green,194; blue,194}]

%% file: figures/f_harmonic.tikz
\begin{tikzpicture}
	\begin{pgfonlayer}{nodelayer}
		\node [style=state] (0) at (-2.5, -1) {$p$};
		\node [style=bn] (1) at (-2.5, 0) {};
		\node [style=none] (2) at (-3.5, 1) {};
		\node [style=none] (3) at (-1.5, 1) {};
		\node [style=morphism] (4) at (-1.5, 1.5) {$m$};
		\node [style=morphism] (5) at (-1.5, 3) {$f$};
		\node [style=none] (6) at (-1.5, 4.25) {};
		\node [style=none] (7) at (-3.5, 4.25) {};
		\node [style=none] (8) at (-3.5, 4.75) {$X$};
		\node [style=none] (9) at (-1.5, 4.75) {$Y$};
		\node [style=none] (10) at (0, 2) {=};
		\node [style=state] (11) at (2, -1) {$p$};
		\node [style=bn] (12) at (2, 0) {};
		\node [style=none] (13) at (1, 1) {};
		\node [style=none] (14) at (3, 1) {};
		\node [style=morphism] (16) at (3, 2.25) {$f$};
		\node [style=none] (17) at (3, 4.25) {};
		\node [style=none] (18) at (1, 4.25) {};
		\node [style=none] (19) at (1, 4.75) {$X$};
		\node [style=none] (20) at (3, 4.75) {$Y$};
	\end{pgfonlayer}
	\begin{pgfonlayer}{edgelayer}
		\draw (1) to (0);
		\draw (7.center) to (2.center);
		\draw (6.center) to (3.center);
		\draw [in=165, out=-90, looseness=1.25] (2.center) to (1);
		\draw [in=-90, out=15, looseness=1.25] (1) to (3.center);
		\draw (12) to (11);
		\draw (18.center) to (13.center);
		\draw (17.center) to (14.center);
		\draw [in=165, out=-90, looseness=1.25] (13.center) to (12);
		\draw [in=-90, out=15, looseness=1.25] (12) to (14.center);
	\end{pgfonlayer}
\end{tikzpicture}

%% file: figures/f_invariant.tikz
\begin{tikzpicture}
	\begin{pgfonlayer}{nodelayer}
		\node [style=state] (0) at (-4, -1.75) {$p$};
		\node [style=bn] (1) at (-4, -0.75) {};
		\node [style=none] (2) at (-5, 0.25) {};
		\node [style=none] (3) at (-3, 0.25) {};
		\node [style=morphism] (4) at (-3, 0.75) {$m$};
		\node [style=none] (6) at (-3, 2) {};
		\node [style=none] (7) at (-5, 2) {};
		\node [style=bn] (9) at (-5, 2) {};
		\node [style=none] (10) at (-7, 4) {};
		\node [style=none] (12) at (-3, 4) {};
		\node [style=none] (13) at (-3, 2) {};
		\node [style=bn] (14) at (-3, 2) {};
		\node [style=none] (15) at (-5, 4) {};
		\node [style=none] (16) at (-1, 4) {};
		\node [style=none] (17) at (-7, 5.75) {};
		\node [style=none] (19) at (-3, 5.75) {};
		\node [style=none] (20) at (-1, 5.75) {};
		\node [style=morphism] (21) at (-7, 4.5) {$f$};
		\node [style=bn] (22) at (-5, 5) {};
		\node [style=none] (23) at (-7, 6.25) {$Y$};
		\node [style=none] (24) at (-3, 6.25) {$X$};
		\node [style=none] (25) at (-1, 6.25) {$X$};
		\node [style=none] (26) at (0, 2.5) {=};
		\node [style=state] (27) at (4, -1.75) {$p$};
		\node [style=bn] (28) at (4, -0.75) {};
		\node [style=none] (29) at (3, 0.25) {};
		\node [style=none] (30) at (5, 0.25) {};
		\node [style=morphism] (31) at (5, 0.75) {$m$};
		\node [style=none] (32) at (5, 2) {};
		\node [style=none] (33) at (3, 2) {};
		\node [style=bn] (34) at (3, 2) {};
		\node [style=none] (35) at (1, 4) {};
		\node [style=none] (36) at (5, 4) {};
		\node [style=none] (37) at (5, 2) {};
		\node [style=bn] (38) at (5, 2) {};
		\node [style=none] (39) at (3, 4) {};
		\node [style=none] (40) at (7, 4) {};
		\node [style=none] (41) at (1, 5) {};
		\node [style=none] (42) at (5, 5.75) {};
		\node [style=none] (43) at (7, 5.75) {};
		\node [style=morphism] (44) at (3, 4.5) {$f$};
		\node [style=none] (47) at (5, 6.25) {$X$};
		\node [style=none] (48) at (7, 6.25) {$X$};
		\node [style=bn] (49) at (1, 5) {};
		\node [style=none] (50) at (3, 5.75) {};
		\node [style=none] (51) at (3, 6.25) {$Y$};
	\end{pgfonlayer}
	\begin{pgfonlayer}{edgelayer}
		\draw (1) to (0);
		\draw (7.center) to (2.center);
		\draw (6.center) to (3.center);
		\draw [in=165, out=-90, looseness=1.25] (2.center) to (1);
		\draw [in=-90, out=15, looseness=1.25] (1) to (3.center);
		\draw [in=165, out=-90, looseness=1.25] (10.center) to (9);
		\draw [in=-90, out=15, looseness=1.25] (9) to (12.center);
		\draw [in=165, out=-90, looseness=1.25] (15.center) to (14);
		\draw [in=-90, out=15, looseness=1.25] (14) to (16.center);
		\draw (17.center) to (10.center);
		\draw (19.center) to (12.center);
		\draw (20.center) to (16.center);
		\draw (22) to (15.center);
		\draw (28) to (27);
		\draw (33.center) to (29.center);
		\draw (32.center) to (30.center);
		\draw [in=165, out=-90, looseness=1.25] (29.center) to (28);
		\draw [in=-90, out=15, looseness=1.25] (28) to (30.center);
		\draw [in=165, out=-90, looseness=1.25] (35.center) to (34);
		\draw [in=-90, out=15, looseness=1.25] (34) to (36.center);
		\draw [in=165, out=-90, looseness=1.25] (39.center) to (38);
		\draw [in=-90, out=15, looseness=1.25] (38) to (40.center);
		\draw (41.center) to (35.center);
		\draw (42.center) to (36.center);
		\draw (43.center) to (40.center);
		\draw (50.center) to (39.center);
	\end{pgfonlayer}
\end{tikzpicture}

%% file: figures/f_invariant_eq.tikz
\begin{tikzpicture}
	\begin{pgfonlayer}{nodelayer}
		\node [style=state] (0) at (-3.75, -3) {$p$};
		\node [style=bn] (1) at (-3.75, -2) {};
		\node [style=none] (2) at (-5.25, -0.5) {};
		\node [style=none] (3) at (-2.25, -0.5) {};
		\node [style=morphism] (4) at (-1.25, 1) {$m$};
		\node [style=none] (8) at (-3.25, 0.5) {};
		\node [style=none] (9) at (-5.25, -0.5) {};
		\node [style=bn] (11) at (-2.25, -0.5) {};
		\node [style=none] (12) at (-3.25, 0.5) {};
		\node [style=none] (13) at (-1.25, 0.5) {};
		\node [style=none] (14) at (-3.25, 2.25) {};
		\node [style=none] (15) at (-5.25, 2.25) {};
		\node [style=none] (16) at (-1.25, 2.25) {};
		\node [style=morphism] (17) at (-3.25, 1) {$f$};
		\node [style=none] (19) at (-3.25, 2.75) {$Y$};
		\node [style=none] (20) at (-5.25, 2.75) {$X$};
		\node [style=none] (21) at (-1.25, 2.75) {$X$};
		\node [style=none] (22) at (0, 0.25) {=};
		\node [style=state] (23) at (2.5, -3.75) {$p$};
		\node [style=bn] (24) at (2.5, -2.75) {};
		\node [style=none] (25) at (1, -1.25) {};
		\node [style=none] (26) at (4, -1.25) {};
		\node [style=morphism] (27) at (4, -1.25) {$m$};
		\node [style=none] (28) at (4, 0) {};
		\node [style=none] (32) at (1, -1.25) {};
		\node [style=none] (33) at (4, 0) {};
		\node [style=bn] (34) at (4, 0) {};
		\node [style=none] (35) at (3, 1) {};
		\node [style=none] (36) at (5, 1) {};
		\node [style=none] (38) at (1, 2.75) {};
		\node [style=none] (39) at (5, 2.75) {};
		\node [style=morphism] (40) at (3, 1.5) {$f$};
		\node [style=none] (41) at (1, 3.25) {$X$};
		\node [style=none] (42) at (5, 3.25) {$X$};
		\node [style=none] (44) at (3, 2.75) {};
		\node [style=none] (45) at (3, 3.25) {$Y$};
	\end{pgfonlayer}
	\begin{pgfonlayer}{edgelayer}
		\draw (1) to (0);
		\draw [in=165, out=-90, looseness=1.25] (2.center) to (1);
		\draw [in=-90, out=15, looseness=1.25] (1) to (3.center);
		\draw [in=165, out=-90, looseness=1.25] (12.center) to (11);
		\draw [in=-90, out=15, looseness=1.25] (11) to (13.center);
		\draw (14.center) to (8.center);
		\draw (15.center) to (9.center);
		\draw (16.center) to (13.center);
		\draw (24) to (23);
		\draw (28.center) to (26.center);
		\draw [in=165, out=-90, looseness=1.25] (25.center) to (24);
		\draw [in=-90, out=15, looseness=1.25] (24) to (26.center);
		\draw [in=165, out=-90, looseness=1.25] (35.center) to (34);
		\draw [in=-90, out=15, looseness=1.25] (34) to (36.center);
		\draw (38.center) to (32.center);
		\draw (39.center) to (36.center);
		\draw (44.center) to (35.center);
	\end{pgfonlayer}
\end{tikzpicture}

%% file: figures/CS_cond_1.tikz
\begin{tikzpicture}
	\begin{pgfonlayer}{nodelayer}
		\node [style=none] (0) at (-2.75, -2.5) {};
		\node [style=bn] (1) at (-2.75, 0) {};
		\node [style=morphism] (2) at (-2.75, -1.25) {$\tilde{f}$};
		\node [style=none] (3) at (-3.75, 1) {};
		\node [style=none] (4) at (-1.75, 1) {};
		\node [style=morphism] (5) at (-3.75, 1.5) {$\tilde{g}$};
		\node [style=morphism] (6) at (-1.75, 1.5) {$\tilde{g}$};
		\node [style=morphism] (7) at (-3.75, 3.25) {$\tilde{h}$};
		\node [style=morphism] (8) at (-1.75, 3.25) {$\tilde{h}$};
		\node [style=none] (9) at (-3.75, 4.5) {};
		\node [style=none] (10) at (-1.75, 4.5) {};
		\node [style=none] (11) at (-0.25, 1) {=};
		\node [style=state] (12) at (4, -2.5) {$p$};
		\node [style=bn] (13) at (4, -1.5) {};
		\node [style=bn] (14) at (3, -0.5) {};
		\node [style=bn] (15) at (5, -0.5) {};
		\node [style=none] (16) at (1, 1.5) {};
		\node [style=none] (18) at (5, 1.5) {};
		\node [style=bn] (19) at (5, -0.5) {};
		\node [style=none] (20) at (3, 1.5) {};
		\node [style=none] (21) at (7, 1.5) {};
		\node [style=morphism] (22) at (1, 2) {$m$};
		\node [style=morphism] (23) at (5, 2) {$m$};
		\node [style=bn] (28) at (3, 4) {};
		\node [style=morphism] (29) at (1, 3.5) {$f$};
		\node [style=morphism] (30) at (5, 3.5) {$f$};
		\node [style=bn] (31) at (7, 4) {};
		\node [style=none] (32) at (5, 4.75) {};
		\node [style=none] (33) at (1, 4.75) {};
		\node [style=none] (34) at (1, 5.25) {$Y$};
		\node [style=none] (35) at (5, 5.25) {$Y$};
		\node [style=none] (36) at (8, 1) {=};
		\node [style=state] (37) at (10.5, -1.5) {$p$};
		\node [style=bn] (38) at (10.5, -0.5) {};
		\node [style=bn] (39) at (10.5, -0.5) {};
		\node [style=none] (41) at (9.5, 0.5) {};
		\node [style=none] (42) at (11.5, 0.5) {};
		\node [style=morphism] (46) at (9.5, 1) {$m$};
		\node [style=morphism] (47) at (11.5, 1) {$m$};
		\node [style=morphism] (49) at (9.5, 2.5) {$f$};
		\node [style=morphism] (50) at (11.5, 2.5) {$f$};
		\node [style=none] (52) at (11.5, 3.75) {};
		\node [style=none] (53) at (9.5, 3.75) {};
		\node [style=none] (54) at (9.5, 4.25) {$Y$};
		\node [style=none] (55) at (11.5, 4.25) {$Y$};
	\end{pgfonlayer}
	\begin{pgfonlayer}{edgelayer}
		\draw (1) to (0.center);
		\draw (10.center) to (4.center);
		\draw (3.center) to (9.center);
		\draw [in=165, out=-90, looseness=1.25] (3.center) to (1);
		\draw [in=-90, out=15, looseness=1.25] (1) to (4.center);
		\draw (13) to (12);
		\draw [in=165, out=-90, looseness=1.25] (14) to (13);
		\draw [in=-90, out=15, looseness=1.25] (13) to (15);
		\draw [in=165, out=-90, looseness=1.25] (16.center) to (14);
		\draw [in=-90, out=15, looseness=1.25] (14) to (18.center);
		\draw [in=165, out=-90, looseness=1.25] (20.center) to (19);
		\draw [in=-90, out=15, looseness=1.25] (19) to (21.center);
		\draw (33.center) to (16.center);
		\draw (28) to (20.center);
		\draw (32.center) to (18.center);
		\draw (31) to (21.center);
		\draw (38) to (37);
		\draw [in=165, out=-90, looseness=1.25] (41.center) to (39);
		\draw [in=-90, out=15, looseness=1.25] (39) to (42.center);
		\draw (53.center) to (41.center);
		\draw (52.center) to (42.center);
	\end{pgfonlayer}
\end{tikzpicture}

%% file: figures/CS_cond_2.tikz
\begin{tikzpicture}
	\begin{pgfonlayer}{nodelayer}
		\node [style=none] (0) at (-2, -2.75) {};
		\node [style=bn] (1) at (-2, 1.5) {};
		\node [style=morphism] (2) at (-2, -1.5) {$\tilde{f}$};
		\node [style=none] (3) at (-3, 2.5) {};
		\node [style=none] (4) at (-1, 2.5) {};
		\node [style=morphism] (6) at (-2, 0.25) {$\tilde{g}$};
		\node [style=morphism] (7) at (-3, 3) {$\tilde{h}$};
		\node [style=morphism] (8) at (-1, 3) {$\tilde{h}$};
		\node [style=none] (9) at (-3, 4.25) {};
		\node [style=none] (10) at (-1, 4.25) {};
		\node [style=none] (11) at (0, 1) {=};
		\node [style=state] (12) at (4, -2.75) {$p$};
		\node [style=bn] (13) at (4, -1.75) {};
		\node [style=bn] (14) at (3, 1) {};
		\node [style=none] (15) at (5, -0.75) {};
		\node [style=none] (16) at (1, 3) {};
		\node [style=none] (17) at (5, 3) {};
		\node [style=bn] (18) at (5, 1) {};
		\node [style=none] (19) at (3, 3) {};
		\node [style=none] (20) at (7, 3) {};
		\node [style=bn] (23) at (3, 4) {};
		\node [style=morphism] (24) at (1, 3.5) {$f$};
		\node [style=morphism] (25) at (5, 3.5) {$f$};
		\node [style=bn] (26) at (7, 4) {};
		\node [style=none] (27) at (5, 4.75) {};
		\node [style=none] (28) at (1, 4.75) {};
		\node [style=none] (29) at (1, 5.25) {$Y$};
		\node [style=none] (30) at (5, 5.25) {$Y$};
		\node [style=none] (31) at (3, -0.75) {};
		\node [style=morphism] (32) at (3, -0.25) {$m$};
		\node [style=none] (34) at (7.75, 1) {=};
		\node [style=state] (35) at (10, -1.75) {$p$};
		\node [style=bn] (37) at (10, 1) {};
		\node [style=none] (39) at (9, 2) {};
		\node [style=none] (40) at (11, 2) {};
		\node [style=morphism] (45) at (9, 2.5) {$f$};
		\node [style=morphism] (46) at (11, 2.5) {$f$};
		\node [style=none] (48) at (11, 3.75) {};
		\node [style=none] (49) at (9, 3.75) {};
		\node [style=none] (50) at (9, 4.25) {$Y$};
		\node [style=none] (51) at (11, 4.25) {$Y$};
		\node [style=none] (52) at (10, -0.75) {};
		\node [style=morphism] (53) at (10, -0.25) {$m$};
		\node [style=none] (54) at (12, 1) {=};
		\node [style=state] (55) at (14.5, -1) {$p$};
		\node [style=none] (58) at (13.5, 1) {};
		\node [style=none] (59) at (15.5, 1) {};
		\node [style=morphism] (60) at (13.5, 1.5) {$f$};
		\node [style=morphism] (61) at (15.5, 1.5) {$f$};
		\node [style=none] (62) at (15.5, 2.75) {};
		\node [style=none] (63) at (13.5, 2.75) {};
		\node [style=none] (64) at (13.5, 3.25) {$Y$};
		\node [style=none] (65) at (15.5, 3.25) {$Y$};
		\node [style=bn] (66) at (14.5, 0) {};
	\end{pgfonlayer}
	\begin{pgfonlayer}{edgelayer}
		\draw (1) to (0.center);
		\draw (10.center) to (4.center);
		\draw (3.center) to (9.center);
		\draw [in=165, out=-90, looseness=1.25] (3.center) to (1);
		\draw [in=-90, out=15, looseness=1.25] (1) to (4.center);
		\draw (13) to (12);
		\draw [in=-90, out=15, looseness=1.25] (13) to (15.center);
		\draw [in=165, out=-90, looseness=1.25] (16.center) to (14);
		\draw [in=-90, out=15, looseness=1.25] (14) to (17.center);
		\draw [in=165, out=-90, looseness=1.25] (19.center) to (18);
		\draw [in=-90, out=15, looseness=1.25] (18) to (20.center);
		\draw (28.center) to (16.center);
		\draw (23) to (19.center);
		\draw (27.center) to (17.center);
		\draw (26) to (20.center);
		\draw [in=165, out=-90, looseness=1.25] (31.center) to (13);
		\draw (14) to (31.center);
		\draw (18) to (15.center);
		\draw [in=165, out=-90, looseness=1.25] (39.center) to (37);
		\draw [in=-90, out=15, looseness=1.25] (37) to (40.center);
		\draw (49.center) to (39.center);
		\draw (48.center) to (40.center);
		\draw (37) to (52.center);
		\draw (63.center) to (58.center);
		\draw (62.center) to (59.center);
		\draw (52.center) to (35);
		\draw (66) to (55);
		\draw [in=165, out=-90, looseness=1.25] (58.center) to (66);
		\draw [in=-90, out=15, looseness=1.25] (66) to (59.center);
	\end{pgfonlayer}
\end{tikzpicture}

%% file: figures/CS_csq_1.tikz
\begin{tikzpicture}
	\begin{pgfonlayer}{nodelayer}
		\node [style=none] (0) at (-2.25, -2.5) {};
		\node [style=bn] (1) at (-2.25, 0) {};
		\node [style=morphism] (2) at (-2.25, -1.25) {$\tilde{f}$};
		\node [style=none] (3) at (-3.25, 1) {};
		\node [style=none] (4) at (-1.25, 1) {};
		\node [style=morphism] (5) at (-3.25, 1.5) {$\tilde{g}$};
		\node [style=morphism] (6) at (-1.25, 1.5) {$\tilde{g}$};
		\node [style=morphism] (7) at (-3.25, 3.25) {$\tilde{h}$};
		\node [style=none] (9) at (-3.25, 4.5) {};
		\node [style=none] (10) at (-1.25, 4.5) {};
		\node [style=none] (11) at (0, 1) {=};
		\node [style=state] (12) at (4, -2.5) {$p$};
		\node [style=bn] (13) at (4, -1.5) {};
		\node [style=bn] (14) at (3, -0.5) {};
		\node [style=bn] (15) at (5, -0.5) {};
		\node [style=none] (16) at (1, 1.5) {};
		\node [style=none] (17) at (5, 1.5) {};
		\node [style=bn] (18) at (5, -0.5) {};
		\node [style=none] (19) at (3, 1.5) {};
		\node [style=none] (20) at (7, 1.5) {};
		\node [style=morphism] (21) at (1, 2) {$m$};
		\node [style=morphism] (22) at (5, 2) {$m$};
		\node [style=bn] (23) at (3, 4) {};
		\node [style=morphism] (24) at (1, 3.5) {$f$};
		\node [style=none] (27) at (5, 4.75) {};
		\node [style=none] (28) at (1, 4.75) {};
		\node [style=none] (29) at (1, 5.25) {$Y$};
		\node [style=none] (30) at (5, 5.25) {$X$};
		\node [style=none] (31) at (8, 1) {=};
		\node [style=state] (32) at (11.75, -2.5) {$p$};
		\node [style=bn] (33) at (10.25, 0) {};
		\node [style=bn] (34) at (10.25, 0) {};
		\node [style=none] (35) at (9.25, 1) {};
		\node [style=none] (36) at (11.25, 1) {};
		\node [style=morphism] (37) at (9.25, 1.5) {$m$};
		\node [style=morphism] (38) at (11.25, 1.5) {$m$};
		\node [style=morphism] (39) at (9.25, 3) {$f$};
		\node [style=none] (41) at (11.25, 4.25) {};
		\node [style=none] (42) at (9.25, 4.25) {};
		\node [style=none] (43) at (9.25, 4.75) {$Y$};
		\node [style=none] (44) at (11.25, 4.75) {$X$};
		\node [style=none] (45) at (7, 4.75) {};
		\node [style=none] (46) at (7, 5.25) {$X$};
		\node [style=bn] (47) at (11.75, -1.5) {};
		\node [style=none] (48) at (13.25, 0) {};
		\node [style=none] (49) at (13.25, 4.75) {$X$};
		\node [style=none] (50) at (13.25, 4.25) {};
	\end{pgfonlayer}
	\begin{pgfonlayer}{edgelayer}
		\draw (1) to (0.center);
		\draw (10.center) to (4.center);
		\draw (3.center) to (9.center);
		\draw [in=165, out=-90, looseness=1.25] (3.center) to (1);
		\draw [in=-90, out=15, looseness=1.25] (1) to (4.center);
		\draw (13) to (12);
		\draw [in=165, out=-90, looseness=1.25] (14) to (13);
		\draw [in=-90, out=15, looseness=1.25] (13) to (15);
		\draw [in=165, out=-90, looseness=1.25] (16.center) to (14);
		\draw [in=-90, out=15, looseness=1.25] (14) to (17.center);
		\draw [in=165, out=-90, looseness=1.25] (19.center) to (18);
		\draw [in=-90, out=15, looseness=1.25] (18) to (20.center);
		\draw (28.center) to (16.center);
		\draw (23) to (19.center);
		\draw (27.center) to (17.center);
		\draw [in=165, out=-90, looseness=1.25] (35.center) to (34);
		\draw [in=-90, out=15, looseness=1.25] (34) to (36.center);
		\draw (42.center) to (35.center);
		\draw (41.center) to (36.center);
		\draw (45.center) to (20.center);
		\draw (50.center) to (48.center);
		\draw [in=15, out=-90, looseness=1.25] (48.center) to (47);
		\draw [in=-90, out=165, looseness=1.25] (47) to (34);
		\draw (47) to (32);
	\end{pgfonlayer}
\end{tikzpicture}

%% file: figures/CS_csq_2.tikz
\begin{tikzpicture}
	\begin{pgfonlayer}{nodelayer}
		\node [style=none] (0) at (-2, -2.75) {};
		\node [style=bn] (1) at (-2, 1.5) {};
		\node [style=morphism] (2) at (-2, -1.5) {$\tilde{f}$};
		\node [style=none] (3) at (-3, 2.5) {};
		\node [style=none] (4) at (-1, 2.5) {};
		\node [style=morphism] (5) at (-2, 0.25) {$\tilde{g}$};
		\node [style=morphism] (6) at (-3, 3) {$\tilde{h}$};
		\node [style=none] (8) at (-3, 4.25) {};
		\node [style=none] (9) at (-1, 4.25) {};
		\node [style=none] (10) at (0, 1) {=};
		\node [style=state] (11) at (4, -2.75) {$p$};
		\node [style=bn] (12) at (4, -1.75) {};
		\node [style=bn] (13) at (3, 1) {};
		\node [style=none] (14) at (5, -0.75) {};
		\node [style=none] (15) at (1, 3) {};
		\node [style=none] (16) at (5, 3) {};
		\node [style=bn] (17) at (5, 1) {};
		\node [style=none] (18) at (3, 3) {};
		\node [style=none] (19) at (7, 3) {};
		\node [style=bn] (20) at (3, 4) {};
		\node [style=morphism] (21) at (1, 3.5) {$f$};
		\node [style=none] (24) at (5, 4.75) {};
		\node [style=none] (25) at (1, 4.75) {};
		\node [style=none] (26) at (1, 5.25) {$Y$};
		\node [style=none] (27) at (5, 5.25) {$X$};
		\node [style=none] (28) at (3, -0.75) {};
		\node [style=morphism] (29) at (3, -0.25) {$m$};
		\node [style=none] (31) at (8, 1) {=};
		\node [style=state] (32) at (11.5, -2.75) {$p$};
		\node [style=none] (33) at (10, -0.25) {};
		\node [style=bn] (34) at (10, 1.5) {};
		\node [style=none] (35) at (9, 2.5) {};
		\node [style=none] (36) at (11, 2.5) {};
		\node [style=morphism] (37) at (9, 3) {$f$};
		\node [style=none] (39) at (11, 4.25) {};
		\node [style=none] (40) at (9, 4.25) {};
		\node [style=none] (41) at (9, 4.75) {$Y$};
		\node [style=none] (42) at (11, 4.75) {$X$};
		\node [style=none] (43) at (10, -0.25) {};
		\node [style=morphism] (44) at (10, 0.25) {$m$};
		\node [style=none] (58) at (7, 4.75) {};
		\node [style=none] (59) at (7, 5.25) {$X$};
		\node [style=bn] (60) at (11.5, -1.75) {};
		\node [style=none] (61) at (13, -0.25) {};
		\node [style=none] (62) at (13, 4.25) {};
		\node [style=none] (63) at (13, 4.75) {$X$};
	\end{pgfonlayer}
	\begin{pgfonlayer}{edgelayer}
		\draw (1) to (0.center);
		\draw (9.center) to (4.center);
		\draw (3.center) to (8.center);
		\draw [in=165, out=-90, looseness=1.25] (3.center) to (1);
		\draw [in=-90, out=15, looseness=1.25] (1) to (4.center);
		\draw (12) to (11);
		\draw [in=-90, out=15, looseness=1.25] (12) to (14.center);
		\draw [in=165, out=-90, looseness=1.25] (15.center) to (13);
		\draw [in=-90, out=15, looseness=1.25] (13) to (16.center);
		\draw [in=165, out=-90, looseness=1.25] (18.center) to (17);
		\draw [in=-90, out=15, looseness=1.25] (17) to (19.center);
		\draw (25.center) to (15.center);
		\draw (20) to (18.center);
		\draw (24.center) to (16.center);
		\draw [in=165, out=-90, looseness=1.25] (28.center) to (12);
		\draw (13) to (28.center);
		\draw (17) to (14.center);
		\draw [in=165, out=-90, looseness=1.25] (35.center) to (34);
		\draw [in=-90, out=15, looseness=1.25] (34) to (36.center);
		\draw (40.center) to (35.center);
		\draw (39.center) to (36.center);
		\draw (34) to (43.center);
		\draw (19.center) to (58.center);
		\draw [in=165, out=-90, looseness=1.25] (43.center) to (60);
		\draw [in=-90, out=15, looseness=1.25] (60) to (61.center);
		\draw (61.center) to (62.center);
		\draw (60) to (32);
	\end{pgfonlayer}
\end{tikzpicture}

%% file: figures/rm_relative_positivity_csq.tikz
\begin{tikzpicture}
	\begin{pgfonlayer}{nodelayer}
		\node [style=bn] (0) at (-2, 0.5) {};
		\node [style=none] (1) at (-1, 1.5) {};
		\node [style=none] (2) at (-3, 1.5) {};
		\node [style=morphism] (3) at (-3, 2) {$r$};
		\node [style=none] (4) at (-1, 3.25) {};
		\node [style=none] (5) at (-1, 3.75) {$X$};
		\node [style=none] (6) at (-3, 3.75) {$X_\mathrm{inv}$};
		\node [style=none] (7) at (-3, 3.25) {};
		\node [style=morphism] (9) at (-2, -0.75) {$m$};
		\node [style=state] (10) at (-2, -2) {$p$};
		\node [style=none] (16) at (3.5, 0) {};
		\node [style=none] (17) at (1.5, 0) {};
		\node [style=morphism] (18) at (1.5, 2) {$r$};
		\node [style=none] (19) at (3.5, 3.25) {};
		\node [style=none] (20) at (3.5, 3.75) {$X$};
		\node [style=none] (21) at (1.5, 3.75) {$X_\mathrm{inv}$};
		\node [style=none] (22) at (1.5, 3.25) {};
		\node [style=state] (23) at (2.5, -2) {$p$};
		\node [style=morphism] (24) at (1.5, 0.5) {$m$};
		\node [style=bn] (27) at (2.5, -1) {};
		\node [style=none] (30) at (0, 0) {=};
		\node [style=morphism] (31) at (3.5, 0.5) {$m$};
	\end{pgfonlayer}
	\begin{pgfonlayer}{edgelayer}
		\draw (7.center) to (2.center);
		\draw (1.center) to (4.center);
		\draw [in=165, out=-90, looseness=1.25] (2.center) to (0);
		\draw [in=-90, out=15, looseness=1.25] (0) to (1.center);
		\draw (0) to (10);
		\draw (22.center) to (17.center);
		\draw (16.center) to (19.center);
		\draw (27) to (23);
		\draw [in=165, out=-90, looseness=1.25] (17.center) to (27);
		\draw [in=-90, out=15, looseness=1.25] (27) to (16.center);
	\end{pgfonlayer}
\end{tikzpicture}

%% file: figures/alt-erg.tikz
\begin{tikzpicture}
	\begin{pgfonlayer}{nodelayer}
		\node [style=state] (0) at (-3, -2) {$p$};
		\node [style=bn] (1) at (-3, -1) {};
		\node [style=none] (2) at (-4, 0) {};
		\node [style=none] (3) at (-2, 0) {};
		\node [style=none] (4) at (-4, 1.25) {};
		\node [style=none] (5) at (-2, 1.25) {};
		\node [style=morphism] (6) at (-4, 0.25) {$e_D$};
		\node [style=none] (8) at (-4, 1.75) {$X$};
		\node [style=none] (9) at (-2, 1.75) {$X$};
		\node [style=none] (10) at (0, -0.25) {$=$};
		\node [style=state] (11) at (2, -2) {$p$};
		\node [style=state] (12) at (4, -2) {$p$};
		\node [style=none] (13) at (2, 1.25) {};
		\node [style=none] (14) at (4, 1.25) {};
		\node [style=none] (15) at (2, 1.75) {$X$};
		\node [style=none] (16) at (4, 1.75) {$X$};
	\end{pgfonlayer}
	\begin{pgfonlayer}{edgelayer}
		\draw (1) to (0);
		\draw [in=165, out=-90] (2.center) to (1);
		\draw [in=270, out=15] (1) to (3.center);
		\draw (4.center) to (2.center);
		\draw (5.center) to (3.center);
		\draw (11) to (13.center);
		\draw (14.center) to (12);
	\end{pgfonlayer}
\end{tikzpicture}

%% file: paper.bbl
\newcommand{\etalchar}[1]{$^{#1}$}
\begin{thebibliography}{FKM{\etalchar{+}}24}

\bibitem[AC04]{quantumprotocols}
Samson Abramsky and Bob Coecke.
\newblock A categorical semantics of quantum protocols.
\newblock In {\em 2004 19nd {A}nnual {ACM}/{IEEE} {S}ymposium on {L}ogic in
  {C}omputer {S}cience ({LICS})}. IEEE, [Piscataway], NJ, 2004.
\newblock
  \href{https://arxiv.org/abs/quant-ph/0402130}{arXiv:quant-ph/0402130}.

\bibitem[BD86]{cauchycompletion}
Francis Borceux and Dominique Dejean.
\newblock Cauchy completion in category theory.
\newblock {\em Cahiers de Topologie et Géométrie Différentielle
  Catégoriques}, 27(2):133--146, 1986.

\bibitem[Bla42]{blackwell1942idempotent}
David Blackwell.
\newblock Idempotent {M}arkoff chains.
\newblock {\em Ann. of Math. (2)}, 43:560--567, 1942.

\bibitem[Bog07]{BogachevVladimirI2007MT}
Vladimir~I Bogachev.
\newblock {\em Measure Theory}, volume~1.
\newblock Springer-Verlag, Berlin, Heidelberg, 1. aufl. edition, 2007.

\bibitem[CJ19]{chojacobs2019strings}
Kenta Cho and Bart Jacobs.
\newblock Disintegration and {B}ayesian inversion via string diagrams.
\newblock {\em Math. Structures Comput. Sci.}, 29:938--971, 2019.
\newblock \href{https://arxiv.org/abs/1709.00322}{arXiv:1709.00322}.

\bibitem[CP07]{quantumwithoutsums}
Bob Coecke and Dusko Pavlovic.
\newblock {\em Quantum measurements without sums}, pages 567--604.
\newblock 2007.

\bibitem[DDGS18]{dahlqvist2018borel}
Fredrik Dahlqvist, Vincent Danos, Ilias Garnier, and Alexandra Silva.
\newblock Borel kernels and their approximation, categorically.
\newblock In {\em MFPS 34: Proceedings of the Thirty-Fourth Conference on the
  Mathematical Foundations of Programming Semantics}, volume 341, pages
  91--119, 2018.

\bibitem[DMPS18]{markovchains}
Randal Douc, Eric Moulines, Pierre Priouret, and Philippe Soulier.
\newblock {\em Markov Chains}.
\newblock Springer, 2018.

\bibitem[FGL{\etalchar{+}}23]{supports}
Tobias Fritz, Tom{\'a}{\v{s}} Gonda, Antonio Lorenzin, Paolo Perrone, and Dario
  Stein.
\newblock Absolute continuity, supports and idempotent splitting in categorical
  probability, 2023.
\newblock \href{https://arxiv.org/abs/2308.00651}{arXiv:2308.00651}.

\bibitem[FGP21]{fritz2021definetti}
Tobias Fritz, Tom\'{a}\v{s} Gonda, and Paolo Perrone.
\newblock de {F}inetti's theorem in categorical probability.
\newblock {\em J. Stoch. Anal.}, 2(4), 2021.
\newblock \href{https://arxiv.org/abs/2105.02639}{arXiv:2105.02639}.

\bibitem[FGPR]{fritz2020representable}
Tobias Fritz, Tom{\'a}{\v{s}} Gonda, Paolo Perrone, and Eigil~Fjeldgren
  Rischel.
\newblock Representable {M}arkov categories and comparison of statistical
  experiments in categorical probability.
\newblock \href{https://arxiv.org/abs/2010.07416}{arxiv.org/abs/2010.07416}.

\bibitem[FGPS]{fritz2022dilations}
Tobias Fritz, Tom{\'a}{\v{s}} Gonda, Paolo Perrone, and Dario Stein.
\newblock Dilations and information flow axioms in categorical probability.
\newblock \href{https://arxiv.org/abs/2211.02507}{arXiv:2211.02507}.

\bibitem[FK23]{fritz2022dseparation}
Tobias Fritz and Andreas Klingler.
\newblock The $d$-separation criterion in categorical probability.
\newblock {\em J. Mach. Learn. Res.}, 24(46):1--49, 2023.
\newblock \href{https://arxiv.org/abs/2207.05740}{arXiv:2207.05740}.

\bibitem[FKM{\etalchar{+}}24]{hiddenmarkov}
Tobias Fritz, Andreas Klingler, Drew McNeely, Areeb Shah-Mohammed, and Yuwen
  Wang.
\newblock Hidden markov models and the bayes filter in categorical probability,
  2024.
\newblock \href{https://arxiv.org/abs/2401.14669}{arXiv:2401.14669}.

\bibitem[FP19]{fritz2019kantorovich}
Tobias Fritz and Paolo Perrone.
\newblock A probability monad as the colimit of spaces of finite samples.
\newblock {\em Theory and Applications of Categories}, 34(7):170--220, 2019.

\bibitem[FR20]{fritzrischel2019zeroone}
Tobias Fritz and Eigil~Fjeldgren Rischel.
\newblock Infinite products and zero-one laws in categorical probability.
\newblock {\em Compositionality}, 2:3, 2020.
\newblock
  \href{https://compositionality-journal.org/papers/compositionality-2-3/}{compositionality-journal.org/papers/compositionality-2-3}.

\bibitem[Fri]{fritz2021topos}
Tobias Fritz.
\newblock The law of large numbers in categorical probability.
\newblock \url{https://www.youtube.com/watch?v=Gvf3H4e7l8s}.

\bibitem[Fri20]{fritz2019synthetic}
Tobias Fritz.
\newblock A synthetic approach to {M}arkov kernels, conditional independence
  and theorems on sufficient statistics.
\newblock {\em Adv. Math.}, 370:107239, 2020.
\newblock \href{https://arxiv.org/abs/1908.07021}{arXiv:1908.07021}.

\bibitem[Gad96]{gadducci1996}
Fabio Gadducci.
\newblock {\em On The Algebraic Approach To Concurrent Term Rewriting}.
\newblock PhD thesis, 11 1996.

\bibitem[Gir82]{giry}
Michèle Giry.
\newblock A categorical approach to probability theory.
\newblock In {\em Categorical aspects of topology and analysis}, volume 915 of
  {\em Lecture Notes in Mathematics}. Springer, 1982.

\bibitem[HV19]{cqt}
Chris Heunen and Jamie Vicary.
\newblock {\em Categories for Quantum Theory}.
\newblock Oxford University Press, 2019.

\bibitem[Jac17]{jacobs-pmonads}
Bart Jacobs.
\newblock {From Probability Monads to Commutative Effectuses}.
\newblock {\em Journ. of Logical and Algebraic Methods in Programming}, 2017.
\newblock In press. Available at
  \href{http://www.cs.ru.nl/B.Jacobs/PAPERS/probability-monads.pdf}{http://www.cs.ru.nl/B.Jacobs/PAPERS/probability-monads.pdf}.

\bibitem[Jac21]{jacobs2021multinomial}
Bart Jacobs.
\newblock Multinomial and hypergeometric distributions in {M}arkov categories.
\newblock In {\em Proceedings of the {T}hirty-{S}eventh {C}onference on the
  {M}athematical {F}oundations of {P}rogramming {S}emantics ({MFPS}}, volume
  351 of {\em Electron. Notes Theor. Comput. Sci.}, pages 98--115, 2021.
\newblock \href{https://arxiv.org/abs/2112.14044}{arXiv:2112.14044}.

\bibitem[Jac23]{suffstat}
Bart Jacobs.
\newblock Sufficient statistics and split idempotents in discrete probability
  theory.
\newblock In {\em Proceedings of the {T}hirty-{E}ighth {C}onference on the
  {M}athematical {F}oundations of {P}rogramming {S}emantics ({MFPS}
  {XXXVIII})}, volume~1 of {\em Electron. Notes Theor. Comput. Sci.} Elsevier
  Sci. B. V., Amsterdam, 2023.
\newblock \href{https://arxiv.org/abs/2212.09191}{arXiv:2212.09191}.

\bibitem[JP89]{jones-plotkin}
C.~Jones and J.~D. Plotkin.
\newblock {A Probabilistic Powerdomain of Evaluations}.
\newblock {\em Proceedings of the Fourth Annual Symposium of Logics in Computer
  Science}, 1989.

\bibitem[Kar18]{wayofdagger}
Martti Karvonen.
\newblock {\em The Way of the Dagger}.
\newblock PhD thesis, University of Edinburgh, 2018.

\bibitem[Koc12]{kock2012distributions}
Anders Kock.
\newblock Commutative monads as a theory of distributions.
\newblock {\em Theory Appl. Categ.}, 26:No. 4, 97--131, 2012.
\newblock \href{https://arxiv.org/abs/1108.5952}{arXiv:1108.5952}.

\bibitem[Law]{lawvere1962}
F.~W. Lawvere.
\newblock The category of probabilistic mappings.
\newblock \url{https://ncatlab.org/nlab/files/lawvereprobability1962.pdf}.

\bibitem[MP22a]{moss2022ergodic}
Sean Moss and Paolo Perrone.
\newblock A category-theoretic proof of the ergodic decomposition theorem,
  2022.
\newblock \href{https://arxiv.org/abs/2207.07353}{arXiv:2207.07353}.

\bibitem[MP22b]{moss2022probability}
Sean Moss and Paolo Perrone.
\newblock Probability monads with submonads of deterministic states.
\newblock In {\em Proceedings of the 37th Annual ACM/IEEE Symposium on Logic in
  Computer Science}, pages 1--13, 2022.
\newblock \href{https://arxiv.org/abs/2204.07003}{arXiv:2204.07003}.

\bibitem[MT93]{stochstability}
Sean~P. Meyn and Richard~L. Tweedie.
\newblock {\em Markov Chains and Stochastic Stability}.
\newblock Springer, 1993.

\bibitem[Par20]{quantum-markov}
Arthur~J. Parzygnat.
\newblock Inverses, disintegrations, and bayesian inversion in quantum {M}arkov
  categories, 2020.
\newblock \href{https://arxiv.org/abs/2001.08375}{arXiv:2001.08375}.

\bibitem[Per18]{paolo-phdthesis}
Paolo Perrone.
\newblock {\em {Categorical Probability and Stochastic Dominance in Metric
  Spaces}}.
\newblock PhD thesis, University of Leipzig, 2018.
\newblock Submitted. Available at
  \href{http://personal-homepages.mis.mpg.de/perrone/phdthesis.pdf}{http://personal-homepages.mis.mpg.de/perrone/phdthesis.pdf}.

\bibitem[Per21]{lifting}
Paolo Perrone.
\newblock Lifting couplings in wasserstein spaces, 2021.
\newblock \href{https://arxiv.org/abs/2110.06591}{2110.06591}.

\bibitem[Sel07]{dagger-compact}
Peter Selinger.
\newblock Dagger compact closed categories and completely positive maps:
  (extended abstract).
\newblock {\em Electronic Notes in Theoretical Computer Science}, 170:139--163,
  2007.
\newblock Proceedings of the 3rd International Workshop on Quantum Programming
  Languages (QPL 2005).

\bibitem[Sel08]{daggeridempotents}
Peter Selinger.
\newblock Idempotents in dagger categories (extended abstract).
\newblock {\em Electronic Notes in Theoretical Computer Science}, 210:107--122,
  2008.

\bibitem[SS21]{stein2021conditioning}
Dario Stein and Sam Staton.
\newblock Compositional semantics for probabilistic programs with exact
  conditioning.
\newblock In {\em Logic in Computer Science}. IEEE, 2021.
\newblock \href{https://arxiv.org/abs/2101.11351}{arXiv:2101.11351}.

\bibitem[Ste21]{stein2021structural}
Dario Stein.
\newblock {\em Structural Foundations for Probabilistic Programming Languages}.
\newblock PhD thesis, University of Oxford, 2021.
\newblock \href{https://dario-stein.de/thesis.pdf}{dario-stein.de/thesis.pdf}.

\bibitem[Tao08]{tao}
Terence Tao.
\newblock What's new. {Ergodicity}, 254a lecture 9, 2008.
\newblock Mathematical blog with proofs,
  \href{https://terrytao.wordpress.com/2008/02/04/254a-lecture-9-ergodicity}{https://terrytao.wordpress.com/2008/02/04/254a-lecture-9-ergodicity}.

\bibitem[\v{C}65]{cencovcategories}
N.~N. \v{C}encov.
\newblock The categories of mathematical statistics.
\newblock {\em Dokl. Akad. Nauk SSSR}, 164:511--514, 1965.
\newblock \href{https://www.mathnet.ru/rus/dan31602}{mathnet.ru/rus/dan31602}.

\bibitem[Vil09]{villani}
Cédric Villani.
\newblock {\em {Optimal transport: old and new}}, volume 338 of {\em
  {Grundlehren der mathematischen Wissenschaften}}.
\newblock Springer, 2009.

\bibitem[Św74]{swirszcz}
Tadeusz Świrszcz.
\newblock Monadic functors and convexity.
\newblock {\em Bull. Acad. Polon. Sci. Sér. Sci. Math. Astronom. Phys.}, 22,
  1974.

\end{thebibliography}
